\documentclass[11pt, oneside, reqno]{amsart}
\usepackage[utf8]{inputenc}
\usepackage{mathrsfs}
\expandafter\newif\csname ifGin@setpagesize\endcsname 
\usepackage{ytableau}
\usepackage{bm}
\usepackage{amstext, amsthm, amssymb, mathtools, mathdots, subcaption}
\usepackage[all]{xy}
\usepackage{cleveref}
\usepackage[pdftex]{graphics}
\makeatletter
\@namedef{subjclassname@2020}{\textup{2020} Mathematics Subject Classification}
\makeatother
\numberwithin{equation}{section}
\numberwithin{figure}{section}
\theoremstyle{plain}
\newtheorem{thm}{Theorem}[section]
  \crefname{thm}{Theorem}{Theorems}
  \newtheorem{lem}[thm]{Lemma}
  \crefname{lem}{Lemma}{Lemmas}
  \newtheorem{prop}[thm]{Proposition}
  \crefname{prop}{Proposition}{Propositions}
  
	\crefname{cor}{Corollary}{Corollaries}
  \newtheorem*{theor1}{Theorem 1}
  \newtheorem*{theor2}{Theorem 2}

  \newtheorem*{ack}{Acknowledgments}
\theoremstyle{definition}
  \newtheorem{defi}[thm]{Definition}
  \crefname{defi}{Definition}{Definitions}
  \theoremstyle{remark}
  
  \crefname{ntn}{Notation}{Notations}
	 \theoremstyle{remark}
  \newtheorem{rem}[thm]{Remark}
  \crefname{rem}{Remark}{Remarks}
  \newtheorem{ex}[thm]{Example}
  \crefname{ex}{Example}{Examples}
  
\usepackage{a4wide}

\usepackage{tikz-cd}
\usepackage{tikz}
\usetikzlibrary{snakes,3d, matrix,decorations.pathreplacing,calc,decorations.pathmorphing,fit, patterns,trees,positioning}
\usepackage[T1]{fontenc}
\usepackage{lmodern}

\def\r{\mathbb{R}}
\def\c{\mathbb{C}}
\def\q{\mathbb{Q}}
\def\z{\mathbb{Z}}

\makeatother

\makeatletter
\ifcsname phantomsection\endcsname
    \newcommand*{\qrr@gobblenexttocentry}[5]{}
\else
    \newcommand*{\qrr@gobblenexttocentry}[4]{}
\fi
\newcommand*{\addsubsection}{%
    \addtocontents{toc}{\protect\qrr@gobblenexttocentry}%
    \subsection}
\makeatother%acknowledgement消す
 \makeatletter
    
    \@addtoreset{equation}{section}
  \makeatother%数式番号にsection number
\begin{document}
\title[Combinatorics of semi-toric degenerations of Schubert varieties in type C]{Combinatorics of semi-toric degenerations of Schubert varieties in type C}

\author{Naoki FUJITA}

\address[Naoki FUJITA]{Faculty of Advanced Science and Technology, Kumamoto University, 2-39-1 Kurokami, Chuo-ku, Kumamoto 860-8555, Japan.}

\email{fnaoki@kumamoto-u.ac.jp}

\author{Yuta NISHIYAMA}

\address[Yuta NISHIYAMA]{Mathematical Science Education Center, Headquarters for Admissions and Education, Kumamoto University, 2-40-1 Kurokami, Chuo-ku, Kumamoto 860-8555, Japan.}

\email{ynishiyama@kumamoto-u.ac.jp}



\thanks{The work of the first named author was supported by JSPS Grant-in-Aid for Early-Career Scientists (No.\ 20K14281) and by MEXT Japan Leading Initiative for Excellent Young Researchers (LEADER) Project.}

\date{}

\begin{abstract}
An approach to Schubert calculus is to realize Schubert classes as concrete combinatorial objects such as Schubert polynomials.
Using the polytope ring of the Gelfand--Tsetlin polytopes, Kiritchenko--Smirnov--Timorin realized each Schubert class as a sum of reduced Kogan faces.  
The first named author introduced a generalization of reduced Kogan faces to symplectic Gelfand--Tsetlin polytopes using a semi-toric degeneration of a Schubert variety, and extended the result of Kiritchenko--Smirnov--Timorin to type C case.  
In this paper, we introduce a combinatorial model to this type C generalization using a kind of pipe dream with self-crossings. 
As an application, we prove that the type C generalization can be constructed by skew mitosis operators.
\end{abstract}
\maketitle
\ytableausetup{smalltableaux}
\tableofcontents 
\section{Introduction}\label{s:section}

A goal of Schubert calculus is to compute the structure constants of the cohomology ring of a flag variety with respect to the basis consisting of Schubert classes; see \cite{KST, KnM, Man} for the history of Schubert calculus.
One approach to such computation is to realize Schubert classes as concrete combinatorial models such as Schubert polynomials. 
To study such combinatorial models, toric degenerations are useful. 
A toric degeneration is a flat degeneration of a projective variety to a toric variety, which can be used to apply the theory of toric varieties to other projective varieties. 
In the case of a flag variety, its toric degeneration with desirable properties induces degenerations of Schubert and opposite Schubert varieties to not necessarily irreducible torus-invariant closed subvarieties, called semi-toric degenerations. 
Kogan--Miller \cite{KoM} constructed semi-toric degenerations of opposite Schubert varieties as some quotients of Knutson--Miller's semi-toric degenerations \cite{KnM} of opposite matrix Schubert varieties which give a geometric proof of the pipe dream formula of Schubert polynomials. 
Since semi-toric degenerations can be constructed for Schubert and opposite Schubert varieties in general Lie type, they give us a hint to extend theories of Schubert calculus in type $A$ to other Lie types.
Using semi-toric degenerations of Schubert varieties for symplectic Gelfand--Tsetlin polytopes, the first named author \cite{Fuj} extended Kiritchenko--Smirnov--Timorin's combinatorial model \cite{KST} in type A to type $C$ case. 
The purpose of the present paper is to develop combinatorics of this type C model of Schubert classes. 

To be more precise, we first consider $G = SL_{n+1}(\c)$ (of type $A$). 
Then the Weyl group $W$ of $G$ is given as the symmetric group $S_{n+1}$. 
Let $B \subseteq G$ be the subgroup of upper triangular matrices, and $G/B$ the full flag variety. 
For $w \in S_{n+1}$, denote by $\ell(w)$ the length of $w$, and by $X_w \subseteq G/B$ (resp., $X^w \subseteq G/B$) the Schubert variety (resp., the opposite Schubert variety) with $\dim_\c(X_w) = \dim_\c(G/B) - \dim_\c (X^w) = \ell(w)$ (see \cite[Section 1.2]{Bri} for the precise definitions of $X_w$ and $X^w$). 
Then the cohomology class $[X^w]$, called a Schubert class, coincides with $[X_{w_0 w}]$ for each $w \in S_{n+1}$ (see, for instance, \cite[Section 1.3]{Bri}), where $w_0 \in S_{n+1}$ denotes the longest element. 
In addition, the set $\{[X^w] \mid w \in S_{n+1}\}$ of Schubert classes forms a $\z$-basis of the cohomology ring $H^\ast(G/B; \z)$ of $G/B$.
To study this basis, the \emph{Borel description} of the (rational) cohomology ring $H^\ast(G/B; \q)$ is useful, which states that $H^\ast(G/B; \q)$ is isomorphic to the coinvariant algebra of $S_{n+1}$ (see \cite{Man}).
This description allows us to represent each Schubert class as a polynomial. 
Lascoux--Sch\"{u}tzenberger \cite{LS} gave a specific choice $\{\mathfrak{S}_w({\bm x}) \mid w \in S_{n+1}\}$ of representatives, called Schubert polynomials, which have good combinatorial properties. 
Billey--Jockusch--Stanley \cite{BJS} and Fomin--Stanley \cite{FS} gave an explicit combinatorial formula 
\begin{equation}\label{eq:pipe_dream_formula}
\begin{aligned}
\mathfrak{S}_w({\bm x}) = \sum_{D \in RP(w)} {\bm x}^D
\end{aligned}
\end{equation}
of $\mathfrak{S}_w({\bm x})$, called the pipe dream formula, where we used the notation in \cite[Corollary 2.1.3]{KnM}. 
A diagrammatic interpretation of the index set $RP(w)$ was invented by Fomin--Kirillov \cite{FK} and developed by Bergeron--Billey \cite{BerBil} and by Knutson--Miller \cite{KnM, Mil}.
Under this diagrammatic interpretation, each element of $RP(w)$ is called an rc-graph or a reduced pipe dream.
Let $P_+$ be the set of dominant integral weights, $P_{++} \subseteq P_+$ the set of regular dominant integral weights, and $GT(\lambda)$ the Gelfand--Tsetlin polytope for $\lambda \in P_+$. 
By definition, each reduced pipe dream $D \in RP(w)$ naturally corresponds to specific faces $F_D(GT(\lambda))$ and $F_D^\vee(GT(\lambda))$ of $GT(\lambda)$, called a reduced Kogan face and a reduced dual Kogan face, respectively, such that
\[w(F_D(GT(\lambda)))^{-1} = w(F_D^\vee(GT(\lambda))) = w\] 
in the notation of \cite[Sections 3.3 and 4.3]{KST} (see \cite[Section 2.2.1]{Kog} and \cite[Section 4]{KoM}). 
For $\lambda \in P_{++}$, Kogan--Miller \cite{KoM} constructed a semi-toric degeneration of $X^{w_0 w w_0}$ whose limit corresponds to the union of reduced dual Kogan faces $F_D^\vee(GT(\lambda))$, $D \in RP(w)$.
Note that the union of reduced Kogan faces $F_D(GT(\lambda))$, $D \in RP(w)$, appears as a semi-toric degeneration of the Schubert variety $X_{w_0 w^{-1}}$ (see \cite{Fuj}). 
Relations between Schubert classes and reduced (dual) Kogan faces were studied in \cite{Kog, Kir_IMRN, KST}. 
Using an isomorphism between $H^0(SL_{n+1}(\c)/B; \z)$ and the polytope ring of $\{GT(\lambda) \mid \lambda \in P_+\}$ (see \cite{Kav1}), Kiritchenko--Smirnov--Timorin \cite{KST} formulated the following equalities: 
\begin{equation}\label{eq:KST_realization}
\begin{aligned}
[X^w] = \sum_{D \in RP(w^{-1})} [F_D(GT(\lambda))] = \sum_{D \in RP(w_0 w w_0)} [F_D^\vee(GT(\lambda))].
\end{aligned}
\end{equation}

We next consider $G = Sp_{2n}(\c)$ (of type $C$). 
Then the Weyl group $W$ of $G$ is given as the group of signed permutations (see Section \ref{s:symplectic_group} for more details). 
We use the notations similar to the case of type $A$ such as $G/B$, $X_w$, $X^w$, $P_+$, and $P_{++}$. 
Denote by $R(w)$ the set of reduced words for $w \in W$, and take ${\bm i}_C \in R(w_0)$ as in \eqref{eq:reduced_word_type_C}. 
Then Littelmann \cite{Lit} proved that the string polytope $\Delta_{{\bm i}_C}(\lambda)$ associated with ${\bm i}_C$ and $\lambda \in P_+$ is unimodularly equivalent to the symplectic Gelfand--Tsetlin polytope $SGT(\lambda)$. 
For $\lambda \in P_{++}$, Caldero \cite{Cal} constructed a toric degeneration of $G/B$ to the normal projective toric variety $Z(\Delta_{{\bm i}_C}(\lambda))$ corresponding to $\Delta_{{\bm i}_C}(\lambda)$ (see also \cite{Kav2}). 
Using string parametrizations of Demazure and opposite Demazure crystals, Morier-Genoud \cite{Mor} showed that Caldero's toric degeneration \cite{Cal} induces semi-toric degenerations of $X_w$ and $X^w$ for $w \in W$. 
The first named author \cite{Fuj} determined the limits of $X_w$ and $X^w$ in the string polytope $\Delta_{{\bm i}_C}(\lambda)$, and generalized the realization \eqref{eq:KST_realization} by Kiritchenko--Smirnov--Timorin \cite{KST} using such limits.
When $n = 2, 3$, such generalization was previously studied in \cite{Kir, KP}. 
Unlike the case of type $A$, the limit of $X_w$ and that of $X^w$ have different kinds of combinatorial properties. 
Indeed, the irreducible components of the limit of $X^w$ are parametrized by some reduced subwords of ${\bm i}_C$ (see \cite[Section 4]{Fuj}) while those of $X_w$ seem not have any good relation with reduced subwords even in the case $G = Sp_4(\c)$ (see \cref{ex:skew_pipe_dreams_rank_2}). 
In the present paper, we develop combinatorics of the irreducible components of the limit of $X_w$, which inherit information on \emph{skew mitosis operators} as we see below. 
These irreducible components can be parametrized by a certain set $\mathscr{M}(w)$ of \emph{skew pipe dreams} (see Section \ref{s:skew_pipe_dream} for the precise definitions of skew pipe dreams and $\mathscr{M}(w)$). 
For each skew pipe dream $D$, let $F_D(\Delta_{{\bm i}_C}(\lambda))$ denote the corresponding face of $\Delta_{{\bm i}_C}(\lambda)$.
As we have seen above, the set $\mathscr{M}(w)$ and the corresponding faces $F_D(\Delta_{{\bm i}_C}(\lambda))$, $D \in \mathscr{M}(w)$, have the following three geometric or representation-theoretic properties (see \cite[Section 6]{Fuj} for more details): 
\begin{itemize}
\item for $\lambda \in P_{++}$, the union $\bigcup_{D \in \mathscr{M}(w)} F_D(\Delta_{{\bm i}_C}(\lambda))$ corresponds to the limit of a semi-toric degeneration of $X_w$;
\item for $\lambda \in P_+$, the set $\bigcup_{D \in \mathscr{M}(w)} (F_D(\Delta_{{\bm i}_C}(\lambda)) \cap \z^N)$ of lattice points gives a string parametrization of a Demazure crystal, where $N \coloneqq \ell(w_0)$; 
\item for $\lambda \in P_{++}$, the Schubert class $[X_w]$ can be realized as the sum $\sum_{D \in \mathscr{M}(w)} [F_D(\Delta_{{\bm i}_C}(\lambda))]$ through the theory of polytope rings. 
\end{itemize}
In the present paper, we introduce a combinatorial model of the set $\mathscr{M}(w)$ using a kind of pipe dream with self-crossings. 
More precisely, we associate to each skew pipe dream $D$ a diagram $\mathscr{G}(D)$ of ``$n$ pipes with self-crossings'' (see \cref{d:path_model_skew} for the precise definition).
The diagram naturally gives a signed permutation $w_D$, where we have $w_D(j) < 0$ if and only if the $j$-th pipe $L_j$ has an even number of self-crossings for $1 \leq j \leq n$. 
For $w \in W$, a skew pipe dream $D$ with $w_D = w$ is said to be reduced if the cardinality $|D|$ is the minimum among skew pipe dreams $D^\prime$ with $w_{D^\prime} = w$.
The following is the first main result of the present paper, which gives a path model to $\mathscr{M}(w)$. 

\begin{theor1}[{see Theorems \ref{t:main_result_1} and \ref{t:main_result_2}}]
For all $w \in W$, the set $\mathscr{M}(w)$ coincides with the set of reduced skew pipe dreams $D$ such that $w_D = w$. 
\end{theor1}

The set $RP(w)$ of reduced pipe dreams in type $A$ have the following two kinds of combinatorial constructions: 
\begin{itemize}
\item the set $RP(w)$ is stable under ladder moves and is obtained from a specific element $D(w) \in RP(w)$ by applying sequences of ladder moves (see \cite{BerBil});
\item the set $RP(w)$ is obtained from the only one element of $RP(w_0)$ by applying a sequence of (transposed) mitosis operators (see \cite{KnM, Mil}). 
\end{itemize}
In the present paper, we show that these constructions can be generalized to the set $\mathscr{M}(w)$ in type $C$. 
The notion of transposed mitosis operators ${\rm mitosis}_j^\top$ for skew pipe dreams, called transposed skew mitosis operators, was introduced by Kiritchenko \cite{Kir}, which implicitly leads to the definition of ladder moves for skew pipe dreams. 
In addition, the specific element $D(w)$ in type $A$ is generalized by the first named author \cite{Fuj} to an element $D(w)$ of the set $\mathscr{M}(w)$ in type $C$. 
Let $\mathscr{L}(D(w))$ be the set of skew pipe dreams obtained from $D(w)$ by applying sequences of ladder moves.
Then the following is the second main result of the present paper. 

\begin{theor2}[{see \cref{t:main_result_2}}]
For $w \in W$, the set $\mathscr{M}(w)$ is stable under ladder moves. 
In addition, for $(j_1, \ldots, j_\ell) \in R(w)$, the following equalities hold:
\[\mathscr{L}(D(w)) = \mathscr{M}(w) = {\rm mitosis}_{j_\ell}^\top \cdots {\rm mitosis}_{j_1}^\top (SY_n),\]
where $SY_n$ is the only one element of $\mathscr{M}(e)$ for the identity element $e \in W$.
\end{theor2}

In our proofs of Theorems 1 and 2, we first prove the assertion of Theorem 1 for $\mathscr{L}(D(w))$ in \cref{t:main_result_1}. 
Then we show Theorem 2 in \cref{t:main_result_2}, which implies Theorem 1 for $\mathscr{M}(w)$. 
Finally, we mention some previous works. 
Type $C$ pipe dreams were already considered in the studies of Schubert polynomials \cite{FK_typeB, ST} and of double Grothendieck polynomials \cite{KiNa}. 
However, such type $C$ pipe dreams are different from our reduced skew pipe dreams. 
The former inherits information on reduced subwords while the latter does not. 
Type $C$ pipe dreams studied in \cite{FK_typeB, KiNa, ST} are closely related to semi-toric degenerations of opposite Schubert varieties (not of Schubert varieties); see \cite[Section 4]{Fuj}.
A convex-geometric version of mitosis operators was introduced by Kiritchenko \cite{Kir}. 
Its relation with skew mitosis operators was studied in \cite{Kir2}. 

This paper is organized as follows. 
In Section 2, we review some basic definitions on symplectic groups and their Weyl groups. 
In Section 3, we recall some basic notions on skew pipe dreams and skew mitosis operators. 
Some combinatorial properties of $D(w)$ are also proved in this section. 
In Section 4, we introduce a path model to reduced skew pipe dreams using a kind of pipe dream with self-crossings.
In Section 5, we realize reduced skew pipe dreams using ladder moves and skew mitosis operators. 

\begin{ack}\normalfont
The authors are grateful to Tomoo Matsumura for useful comments and fruitful discussions.
\end{ack}

\section{Basic definitions on symplectic groups}\label{s:symplectic_group}

In this section, we review some basic definitions on symplectic groups, following \cite[Section 3]{LakSes} and \cite[Section 8.1]{BjoBre}. 
Geometric and representation-theoretic parts of this and the next sections are not necessary to prove the main results of the present paper.
Hence readers who are mainly interested in the proofs of the main results may skip them. 
Such parts are included because they give geometric and representation-theoretic applications of the main results. 
For $n \in \mathbb{Z}_{\geq 2}$, let $\widetilde{G} = SL_{2n}(\mathbb{C})$ be the complex special linear group of degree $2n$, $\widetilde{B} \subseteq \widetilde{G}$ the Borel subgroup consisting of the upper-triangular matrices, and $\widetilde{H}$ the maximal torus of $\widetilde{B}$ consisting of the diagonal matrices. 
Denoting by $N_{\widetilde{G}}(\widetilde{H})$ the normalizer of $\widetilde{H}$ in $\widetilde{G}$, the Weyl group of type $A_{2n-1}$ is defined by $\widetilde{W} \coloneqq N_{\widetilde{G}}(\widetilde{H})/\widetilde{H}$. 
Let $\{{\bm e}_1, \ldots, {\bm e}_{2n}\}$ be the standard basis of $\mathbb{C}^{2n}$. 
Under the standard representation of $\widetilde{G}$ on $\mathbb{C}^{2n}$, the Weyl group $\widetilde{W}$ is regarded as the symmetric group on $\{\c {\bm e}_1, \ldots, \c {\bm e}_{2n}\}$.
Let 
\[J_n \coloneqq \begin{pmatrix}
0 & 0 & \cdots & 1 \\
\vdots & \vdots & \iddots & \vdots \\
0 & 1 & \cdots & 0 \\
1 & 0 & \cdots & 0
\end{pmatrix}\]
be an integer $n \times n$ anti-diagonal matrix whose anti-diagonal entries are all $1$. 
We take an integer $2n \times 2n$ anti-diagonal matrix $\overline{w}_0$ as
\[\overline{w}_0 \coloneqq \begin{pmatrix}
O_n & J_n \\
-J_n & O_n
\end{pmatrix},\]
where $O_n$ denotes the $n \times n$ zero matrix. 
Define an algebraic group automorphism $\omega \colon \widetilde{G} \xrightarrow{\sim} \widetilde{G}$ of $\widetilde{G} = SL_{2n}(\mathbb{C})$ by 
\[\omega(g) \coloneqq (\overline{w}_0)^{-1} \cdot (g^T)^{-1} \cdot \overline{w}_0\] 
for $g \in \widetilde{G}$, where $g^T$ denotes the transpose of $g$. 
Then the fixed point subgroup 
\[G \coloneqq (\widetilde{G})^\omega = \{g \in \widetilde{G} \mid \omega(g) = g\}\]
of $\widetilde{G}$ coincides with the symplectic group 
\[Sp_{2n}(\mathbb{C}) = \{g \in SL_{2n}(\mathbb{C}) \mid g^T \overline{w}_0 g = \overline{w}_0\}\]
with respect to the skew-symmetric matrix $\overline{w}_0$. 
This is the connected simply-connected simple algebraic group of type $C_n$. 
We set $B \coloneqq \widetilde{B} \cap (\widetilde{G})^\omega$ and $H \coloneqq \widetilde{H} \cap (\widetilde{G})^\omega$.
Then $B$ is a Borel subgroup of $G$, and $H$ is a maximal torus of $G$. 
Hence the quotient variety $G/B$ is the full flag variety of type $C_n$. 
We identify the set $I$ of vertices of the Dynkin diagram with $[n] \coloneqq \{1, 2, \ldots, n\}$ as follows:
\begin{align*}
&C_n\ \begin{xy}
\ar@{=>} (50,0) *++!D{1} *\cir<3pt>{};
(60,0) *++!D{2} *\cir<3pt>{}="C"
\ar@{-} "C";(65,0) \ar@{.} (65,0);(70,0)^*!U{}
\ar@{-} (70,0);(75,0) *++!D{n-1} *\cir<3pt>{}="D"
\ar@{-} "D";(85,0) *++!D{n} *\cir<3pt>{}="E"
\end{xy}.
\end{align*}
Then the Cartan matrix $(c_{i, j})_{i, j \in I}$ of type $C_n$ is given by $c_{2, 1} = -2$, $c_{1, 2} = c_{k, k+1} = c_{k+1, k} = -1$ for $2 \leq k \leq n-1$, $c_{k, k} = 2$ for $1 \leq k \leq n$, and $c_{i, j} = 0$ when $|i-j| \geq 2$. 
Let $N_G(H)$ denote the normalizer of $H$ in $G$, and $W \coloneqq N_G(H)/H$ the Weyl group of type $C_n$.
Since $\omega(\widetilde{H}) = \widetilde{H}$ and $\omega(N_{\widetilde{G}}(\widetilde{H})) = N_{\widetilde{G}}(\widetilde{H})$, the automorphism $\omega$ induces an automorphism $\omega \colon \widetilde{W} \xrightarrow{\sim} \widetilde{W}$. 
Then the Weyl group $W$ of type $C_n$ is identified with the fixed point subgroup $(\widetilde{W})^\omega$. 
Write ${\bm f}_{-n} \coloneqq {\bm e}_1, \ldots, {\bm f}_{-1} \coloneqq {\bm e}_n, {\bm f}_1 \coloneqq {\bm e}_{n+1}, \ldots, {\bm f}_n \coloneqq {\bm e}_{2n}$. 
By the action of $\widetilde{W}$ on $\{\c {\bm e}_1, \ldots, \c {\bm e}_{2n}\} = \{\c {\bm f}_{-n}, \ldots, \c {\bm f}_{-1}, \c {\bm f}_1, \ldots, \c {\bm f}_n\}$, the Weyl group $\widetilde{W}$ is regarded as the symmetric group on $\tilde{I} \coloneqq \{\pm i \mid i \in [n]\}$.
A permutation $\sigma$ of $\tilde{I}$ is called a \emph{signed permutation} if $\sigma(-i) = -\sigma(i)$ for all $i \in [n]$. 
Then the Weyl group $W = (\widetilde{W})^\omega$ is identified with the set of signed permutations of $\tilde{I}$. 
Let $\{s_i \mid i \in I = [n]\} \subseteq W$ denote the set of simple reflections in $W$.
As signed permutations, these are given as follows: 
\[s_1 = (1\quad {\textstyle -1}),\quad s_k = ({\textstyle k-1}\quad k) ({\textstyle -(k-1)}\quad {\textstyle -k}),\ 2 \leq k \leq n,\] 
where $(i\quad j)$ denotes the transposition of $i$ and $j$. 
Note that $W$ is generated by $s_1, s_2, \ldots, s_n$ as a group. 
A sequence ${\bm i} = (i_1, i_2, \ldots, i_r) \in I^r$ is called a \emph{reduced word} for $w \in W$ if $w = s_{i_1} s_{i_2} \cdots s_{i_r}$ and if $r$ is the minimum number among such expressions of $w$. 
In this case, we call $r$ the \emph{length} of $w$, denoted by $\ell(w)$. 
Let $R(w)$ be the set of reduced words for $w$.

\begin{prop}[{see, for instance, \cite[Proposition 8.1.1]{BjoBre}}]\label{p:length_formula}
For each $w \in W$, it holds that 
\[\ell(w) = |\{(i, j) \in \tilde{I} \times I \mid -j \leq i < j,\ w(i) > w(j)\}|.\]
\end{prop}

Define $w_0 \in W$ by $w_0 = (1\quad {\textstyle -1}) (2\quad {\textstyle -2}) \cdots (n\quad {\textstyle -n})$, which is called the \emph{longest element} of $W$. 
Note that the skew-symmetric matrix $\overline{w}_0$ is an element of $N_G(H)$, and we have $w_0 = \overline{w}_0 \bmod H$. 
Write $N \coloneqq n^2$, and define ${\bm i}_C = (i_1, \ldots, i_N) \in I^N$ by
\begin{equation}\label{eq:reduced_word_type_C}
\begin{aligned}
{\bm i}_C \coloneqq (1, \underbrace{2, 1, 2}_3, \underbrace{3, 2, 1, 2, 3}_5, \ldots, \underbrace{n, n-1, \ldots, 1, \ldots, n-1, n}_{2n-1}).
\end{aligned}
\end{equation}
Since we have $\ell(w_0) = n^2$ by \cref{p:length_formula}, it follows that ${\bm i}_C \in R(w_0)$. 
Define a rational convex polyhedral cone $\mathcal{C}_{{\bm i}_C}$ to be the set of
\begin{align*}
(a_1 ^{(1)}, \underbrace{b_1 ^{(2)}, a_2 ^{(1)}, a_1 ^{(2)}}_3, \underbrace{b_1 ^{(3)}, b_2 ^{(2)},  a_3 ^{(1)},  a_2 ^{(2)},  a_1 ^{(3)}}_5, \ldots, \underbrace{b_1 ^{(n)}, \ldots, b_{n-1} ^{(2)}, a_n ^{(1)}, \ldots, a_1 ^{(n)}}_{2n-1}) \in \r^N
\end{align*}
satisfying the following inequalities:
\begin{align*}
&a_1 ^{(1)} \geq 0, & &b_1 ^{(2)} \geq a_2 ^{(1)} \geq a_1 ^{(2)} \geq 0, & &\ldots, & &b_1 ^{(n)} \geq \cdots \geq b_{n-1} ^{(2)} \geq a_n ^{(1)} \geq \cdots \geq a_1 ^{(n)} \geq 0.
\end{align*}
Then Littelmann \cite[Theorem 6.1]{Lit} proved that $\mathcal{C}_{{\bm i}_C}$ coincides with the string cone associated with the reduced word ${\bm i}_C$ (see also \cite[Section 3.2]{BZ}). 
Let $P_+$ be the set of dominant integral weights, and $P_{++} \subseteq P_+$ the set of regular dominant integral weights. 
Denoting the set of fundamental weights by $\{\varpi_i \mid i \in I\}$, we have $P_+ = \sum_{i \in I} \z_{\geq 0} \varpi_i$ and $P_{++} = \sum_{i \in I} \z_{> 0} \varpi_i$. 
For $\lambda = \lambda_1 \varpi_1 + \cdots + \lambda_n \varpi_n \in P_+$, define a rational convex polytope $\Delta_{{\bm i}_C}(\lambda) \subseteq \r^N$ to be the set of $(a_1, \ldots, a_N) \in \mathcal{C}_{{\bm i}_C}$ satisfying
\[a_j \leq \lambda_{i_j} - a_{j+1} c_{i_j, i_{j+1}} - \cdots - a_N c_{i_j, i_N}\]
for all $1 \leq j \leq N$. 
Then we see by \cite[Section 1]{Lit} that $\Delta_{{\bm i}_C}(\lambda)$ coincides with the string polytope associated with ${\bm i}_C \in R(w_0)$ and $\lambda \in P_+$.
We do not review the original definitions of string cones and string polytopes since we do not use them in the present paper, but note that these are defined from certain representation-theoretic objects, called \emph{Kashiwara crystal bases} (see \cite{Kas5} for a survey on crystal bases). 
More precisely, the sets $\mathcal{C}_{{\bm i}_C} \cap \z^N$ and $\Delta_{{\bm i}_C}(\lambda) \cap \z^N$ of their lattice points give string parametrizations of crystal bases.
See \cite[Section 1]{Lit} and \cite[Section 3.2]{BZ} for more details on string cones and string polytopes. 
The string polytope $\Delta_{{\bm i}_C}(\lambda)$ is an integral polytope for all $\lambda \in P_+$, and it is $N$-dimensional if $\lambda \in P_{++}$.
Let us arrange the equations of the facets of $\mathcal{C}_{{\bm i}_C}$ as 
\begin{align*}
a_1 ^{(1)} = 0,\ b_1 ^{(2)} = a_2 ^{(1)},\ a_2 ^{(1)} = a_1 ^{(2)},\ a_1 ^{(2)} = 0,\ b_1 ^{(3)} = b_2 ^{(2)},\ \ldots,\ a_2 ^{(n-1)} = a_1 ^{(n)},\ a_1 ^{(n)} = 0,
\end{align*}
and denote the corresponding facets of $\mathcal{C}_{{\bm i}_C}$ by $F_1(\mathcal{C}_{{\bm i}_C}), \ldots, F_N(\mathcal{C}_{{\bm i}_C})$, respectively. 
For $1 \leq j \leq N$ and $\lambda \in P_+$, let $F_j(\Delta_{{\bm i}_C}(\lambda))$ be the face of $\Delta_{{\bm i}_C}(\lambda)$ given by $F_j(\Delta_{{\bm i}_C}(\lambda)) \coloneqq \Delta_{{\bm i}_C}(\lambda) \cap F_j(\mathcal{C}_{{\bm i}_C})$.
For ${\bm k} = (k_1, \ldots, k_\ell)$ with $\ell \geq 0$ and with $1 \leq k_1 < \cdots < k_\ell \leq N$, we write $F_{\bm k}(\Delta_{{\bm i}_C}(\lambda)) \coloneqq F_{k_1}(\Delta_{{\bm i}_C}(\lambda)) \cap \cdots \cap F_{k_\ell}(\Delta_{{\bm i}_C}(\lambda))$.
For $\lambda \in P_+$, Littelmann \cite[Corollary 7]{Lit} gave a unimodular affine transformation from the string polytope $\Delta_{{\bm i}_C}(\lambda)$ to the symplectic Gelfand--Tsetlin polytope $SGT(\lambda)$ (also known as the type $C$ Gelfand--Tsetlin polytope). 
Under this transformation, the face $F_{\bm k}(\Delta_{{\bm i}_C}(\lambda))$ of $\Delta_{{\bm i}_C}(\lambda)$ corresponds to a symplectic Kogan face of $SGT(\lambda)$ which is an analog of a Kogan face of a type $A$ Gelfand--Tsetlin polytope (see \cite[Section 6]{Fuj}). 
For $w \in W = N_G(H)/H$, take an element $\overline{w} \in N_G(H)$ such that $w = \overline{w} \bmod H$. 
Then we set 
\[X_w \coloneqq \overline{B \overline{w} B/B} \subseteq G/B,\]
which is called a Schubert variety. 
The variety $X_w$ is an irreducible normal closed subvariety of $G/B$ and does not depend on the choice of $\overline{w}$ (see, for instance, \cite[Section I\hspace{-.1em}I.14.15]{Jan}). 
For $\lambda \in P_{++}$, Caldero \cite{Cal} constructed a flat degeneration of $G/B$ to the normal projective toric variety $Z(\Delta_{{\bm i}_C}(\lambda))$ corresponding to $\Delta_{{\bm i}_C}(\lambda)$. 
Then Morier-Genoud \cite{Mor} proved that Caldero's toric degeneration of $G/B$ induces a degeneration of $X_w$ to a (not necessarily irreducible) closed subvariety of $Z(\Delta_{{\bm i}_C}(\lambda))$ that corresponds to a union of faces of $\Delta_{{\bm i}_C}(\lambda)$. 
Such degeneration and degenerated limit of $X_w$ are called a \emph{semi-toric degeneration} and a \emph{semi-toric limit} of $X_w$, respectively. 
Let $\Delta_{{\bm i}_C}(\lambda, X_w)$ denote the union of faces of $\Delta_{{\bm i}_C}(\lambda)$ corresponding to the semi-toric limit of $X_w$. 
Morier-Genoud \cite{Mor} obtained the set $\Delta_{{\bm i}_C}(\lambda, X_w)$ using a \emph{Demazure crystal} introduced in \cite{Kas4} that is a specific subset of a crystal basis.
More precisely, the set $\Delta_{{\bm i}_C}(\lambda, X_w) \cap \z^N$ coincides with the string parametrizations of a Demazure crystal.

\section{Skew pipe dreams and skew mitosis operators}\label{s:skew_pipe_dream}

In this section, we recall some basic notions on skew pipe dreams, following \cite{Kir, Fuj}. 
We also prove some combinatorial properties of a specific skew pipe dream $D(w)$ for $w \in W$. 
Let 
\[SY_n \coloneqq \{(i, j) \in \mathbb{Z}^2 \mid 1 \leq i \leq n,\ i \leq j \leq 2n-i\},\]
and regard it as a shifted Young diagram.
For instance, we write $SY_4$ as 
\[SY_4 = \begin{ytableau}
\mbox{} & \mbox{} & \mbox{} & \mbox{} & \mbox{} & \mbox{} & \mbox{}\\
\none & \mbox{} & \mbox{} & \mbox{} & \mbox{} & \mbox{} & \none \\
\none & \none & \mbox{} & \mbox{} & \mbox{} & \none & \none \\
\none & \none & \none & \mbox{} & \none & \none & \none
\end{ytableau}.\]
Denote by $\mathcal{SPD}_n$ the power set of $SY_n$, and describe $D \in \mathcal{SPD}_n$ by putting $+$ in the boxes corresponding to the elements of $D$. 
For instance, $D = \{(1, 2), (1, 4), (1, 5), (2, 4)\} \in \mathcal{SPD}_3$ is described as 
\[D = \begin{ytableau}
\mbox{} & + & \mbox{} & + & + \\
\none & \mbox{} & \mbox{} & + & \none \\
\none & \none & \mbox{} & \none & \none
\end{ytableau}.\]
We call an element of $\mathcal{SPD}_n$ a \emph{skew pipe dream} (see \cite[Section 5.2]{Kir}).
Recall the reduced word ${\bm i}_C = (i_1, \ldots, i_N)$ in \eqref{eq:reduced_word_type_C}. 
For $w \in W$, we write 
\[R({\bm i}_C, w) \coloneqq \{(k_1, \ldots, k_{\ell(w)}) \in [N]^{\ell(w)} \mid k_1 < \cdots < k_{\ell(w)},\ (i_{k_1}, \ldots, i_{k_{\ell(w)}}) \in R(w)\},\]
and denote by ${\bm k}_w$ the minimum element of $R({\bm i}_C, w)$ with respect to the lexicographic order. 
If $w$ is the identity element $e$, then we set ${\bm k}_w = \emptyset$. 

\begin{ex}\label{ex:reduced_subwords_rank2}
Let $n = 2$. 
Then we have
\[W = \{e, s_1, s_2, s_1 s_2, s_2 s_1, s_1 s_2 s_1, s_2 s_1 s_2, w_0\}.\] 
Since ${\bm i}_C = (1, 2, 1, 2)$, it follows that
\begin{align*}
&R({\bm i}_C, s_1) = \{(1), (3)\},\quad {\bm k}_{s_1} = (1),\quad R({\bm i}_C, s_2) = \{(2), (4)\},\quad {\bm k}_{s_2} = (2),\\ 
&R({\bm i}_C, s_1 s_2) = \{(1, 2), (1, 4), (3, 4)\},\quad {\bm k}_{s_1 s_2} = (1, 2), \quad R({\bm i}_C, s_2 s_1) = \{(2, 3)\},\quad {\bm k}_{s_2 s_1} = (2, 3),\\
&R({\bm i}_C, s_1 s_2 s_1) = \{(1, 2, 3)\},\quad {\bm k}_{s_1 s_2 s_1} = (1, 2, 3),\quad R({\bm i}_C, s_2 s_1 s_2) = \{(2, 3, 4)\},\quad {\bm k}_{s_2 s_1 s_2} = (2, 3, 4). 
\end{align*}
\end{ex}

\begin{ex}
Let $n = 3$, and $w = s_1 s_3 s_2 = s_3 s_1 s_2$. 
Then we have ${\bm i}_C = (1, 2, 1, 2, 3, 2, 1, 2, 3)$ and
\[R({\bm i}_C, w) = \{(1, 5, 6), (1, 5, 8), (3, 5, 6), (3, 5, 8), (5, 7, 8)\}.\]
Hence it follows that ${\bm k}_w = (1, 5, 6)$. 
\end{ex}

Let us arrange the elements of $SY_n$ as follows:
\[((p_1, q_1), \ldots, (p_N, q_N)) \coloneqq ((n, n), (n-1, n+1), (n-1, n), \ldots, (1, 3), (1, 2), (1, 1)).\]
Note that $q_k \in \{n-i_k+1, n+i_k-1\}$ for $1 \leq k \leq N$.
For each $D \in \mathcal{SPD}_n$, we associate a sequence ${\bm k}_D = (k_1, \ldots, k_{|D|})$ with $1 \leq k_1 < \cdots < k_{|D|} \leq N$ by arranging $1 \leq k \leq N$ such that $(p_k, q_k) \in D$ in ascending order.
For instance, a skew pipe dream $D = \{(1, 2), (1, 4), (1, 5), (2, 4)\} \in \mathcal{SPD}_3$ corresponds to a sequence ${\bm k}_D = (2, 5, 6, 8)$.
For each $D \in \mathcal{SPD}_n$, we also associate another sequence ${\bm k}_D^\prime = (k_1, \ldots, k_{N-|D|})$ with $1 \leq k_1 < \cdots < k_{N-|D|} \leq N$ by arranging $1 \leq k \leq N$ such that $(p_k, q_k) \in SY_n \setminus D$ in ascending order.
For instance, a skew pipe dream $D = \{(1, 2), (1, 4), (1, 5), (2, 4)\} \in \mathcal{SPD}_3$ gives a sequence ${\bm k}_D^\prime = (1, 3, 4, 7, 9)$.
For $w \in W$, let $D(w) \in \mathcal{SPD}_n$ denote the skew pipe dream determined by the condition that ${\bm k}_{D(w)}^\prime = {\bm k}_w$. 
If $w = e$, then we define $D(e)$ to be $SY_n$.
It follows by \cite[Proposition 6.3]{Fuj} that there exist unique elements $m(w, 1), \ldots, m(w, n) \in \mathbb{Z}_{\geq 0}$ such that 
\[D(w) = \{(i, j) \in SY_n \mid j \leq m(w, i) + i - 1\}\]
and such that $(i, m(w, i) + i) \notin D(w)$ for $i \in [n]$.  

\begin{prop}\label{p:description_D(w)}
For $w \in W$ and $i \in [n]$, the number $m(w, i)$ coincides with the number of $j \in \tilde{I} = \{\pm k \mid k \in [n]\}$ such that $-(n-i+1) \leq j < n-i+1$ and such that $w^{-1}(j) < w^{-1}(n-i+1)$.
\end{prop}

\begin{proof}
Let $\hat{m}(w, i)$ denote the number of $j \in \tilde{I}$ such that $-(n-i+1) \leq j < n-i+1$ and such that $w^{-1}(j) < w^{-1}(n-i+1)$.
We write ${\bm i}_C = (i^{(1)}_1, \underbrace{i^{(2)}_1, i^{(2)}_2, i^{(2)}_3}_3, \ldots, \underbrace{i^{(n)}_1, \ldots, i^{(n)}_{2n-1}}_{2n-1})$, where $(i^{(k)}_1, \ldots, i^{(k)}_{2k-1}) \coloneqq (k, k-1, \ldots, 1, \ldots, k-1, k)$ for $1 \leq k \leq n$. 
Set $w_k \coloneqq s_{i^{(k)}_1} s_{i^{(k)}_2} \cdots s_{i^{(k)}_{m^\prime(w, n-k+1)}}$ and $w_{\leq k} \coloneqq w_1 w_2 \cdots w_k$ for $1 \leq k \leq n$, where $m^\prime(w, n-k+1) \coloneqq 2k-1-m(w, n-k+1)$. 
Then we see by the definition of $D(w)$ that $w = w_{\leq n}$. 
Let us prove that $m(w, i) = \hat{m}(w_{\leq k}, i)$ for all $n-k+1 \leq i \leq n$ by induction on $k$. 
If $k = 1$, then the element $w_{\leq 1} = w_1$ is the identity element $e$ or $s_1$ depending on whether $(n, n) \in D(w)$ or not. 
Hence the assertion is obvious in this case. 
If $k \geq 2$, then we see that $w_{\leq k-1} (j) = j$ for all $j \in \{\pm q \mid k \leq q \leq n\}$ by the definition of $w_{\leq k-1}$.
Hence it holds that $\hat{m}(w_{\leq k-1}, n-k+1) = 2k-1$. 
By the definition of $w_k$, it follows that 
\[w_k^{-1}(k) = \begin{cases}
k-m^\prime(w, n-k+1) &(\text{if}\quad m^\prime(w, n-k+1) < k),\\
-(1+m^\prime(w, n-k+1)-k) &(\text{if}\quad m^\prime(w, n-k+1) \geq k),
\end{cases}\]
and the sequence $(w_k^{-1}(-(k-1)), \ldots, w_k^{-1}(-1), w_k^{-1}(1), \ldots, w_k^{-1}(k-1))$ is given by arranging $\{-k, \ldots, -1, 1, \ldots, k\} \setminus \{\pm w_k^{-1}(k)\}$ in ascending order. 
From these, it follows that $\hat{m}(w_{\leq k}, i) = \hat{m}(w_{\leq k-1}, i)= m(w, i)$ for all $n-k+2 \leq i \leq n$ and that $\hat{m}(w_{\leq k}, n-k+1) = \hat{m}(w_{\leq k-1}, n-k+1) - m^\prime(w, n-k+1) = m(w, n-k+1)$. 
This proves the proposition. 
\end{proof}

\begin{ex}\label{ex:skew_pipe_rank_3}
Let $n = 3$, and 
\[w = \begin{pmatrix}
-3 & -2 & -1 & 1 & 2 & 3 \\
-1 & 2 & 3 & -3 & -2 & 1
\end{pmatrix} \in W.\]
Then we have $\hat{m}(w, 1) = 2$ and $\hat{m}(w, 2) = \hat{m}(w, 3) = 1$. 
This implies by \cref{p:description_D(w)} that
\[D(w) = \begin{ytableau}
+ & + & \mbox{} & \mbox{} & \mbox{} \\
\none & + & \mbox{} & \mbox{} & \none \\
\none & \none & + & \none & \none
\end{ytableau}.\]
From this, we obtain a reduced word $(2, 1, 3, 2, 1)$ for $w$. 
\end{ex}

Following \cite[Section 6]{Fuj}, we define the \emph{(symplectic) ladder move} $L_{i, j}$ for $(i, j) \in SY_n$, which is an analog of ladder moves for pipe dreams introduced by Bergeron--Billey \cite{BerBil}. 
Let $1 \leq k \leq N$ be the integer given by $(p_k, q_k) = (i, j)$. 
Take $D \in \mathcal{SPD}_n$, and assume the following: 
\begin{itemize}
\item $(i, j) \in D$, $(i, j+1) \notin D$,
\item there exists $k + 1 \leq \ell \leq N$ with $q_\ell \in \{j, 2n-j\}$ such that $(p_\ell, q_\ell), (p_\ell, q_\ell + 1) \in SY_n \setminus D$ and such that $(p_r, q_r), (p_r, q_r + 1) \in D$ for all $k +1 \leq r \leq \ell - 1$ with $q_r \in \{j, 2n-j\}$.
\end{itemize}
Then the ladder move $L_{i, j} (D) \in \mathcal{SPD}_n$ of $D$ at $(i, j)$ is defined by 
\[L_{i, j}(D) \coloneqq D \cup \{(p_\ell, q_\ell +1)\} \setminus \{(i, j)\}.\]
Similarly, we define the \emph{inverse (symplectic) ladder move} $L_{i, j}^{-1}$ for $(i, j) = (p_k, q_k) \in SY_n$ as follows. 
Take $D \in \mathcal{SPD}_n$, and assume the following: 
\begin{itemize}
\item $(i, j) \in D$, $(i, j-1) \notin D$,
\item there exists $1 \leq \ell \leq k-1$ with $q_\ell \in \{j-1, 2n -j+1\}$ such that $(p_\ell, q_\ell), (p_\ell, q_\ell + 1) \notin D$ and such that $(p_r, q_r), (p_r, q_r + 1) \in D$ for all $\ell +1 \leq r \leq k - 1$ with $q_r \in \{j-1, 2n-j+1\}$.
\end{itemize}
Then the inverse ladder move $L_{i, j}^{-1} (D) \in \mathcal{SPD}_n$ of $D$ at $(i, j)$ is defined by 
\[L_{i, j}^{-1}(D) \coloneqq D \cup \{(p_\ell, q_\ell)\} \setminus \{(i, j)\}.\]
Obviously, if the ladder move $L_{i, j}(D) = D \cup \{(p, q)\} \setminus \{(i, j)\}$ is defined, then the inverse ladder move $L_{p, q}^{-1} (L_{i, j}(D))$ is defined, and we have $L_{p, q}^{-1} (L_{i, j}(D)) = D$. 
Similarly, if the inverse ladder move $L_{i, j}^{-1}(D) = D \cup \{(p, q)\} \setminus \{(i, j)\}$ is defined, then the ladder move $L_{p, q} (L_{i, j}^{-1}(D))$ is defined, and we have $L_{p, q} (L_{i, j}^{-1}(D)) = D$. 

\begin{ex}\label{ex:ladder_moves_type_C}
Let $n = 3$. 
Then we obtain the following: 
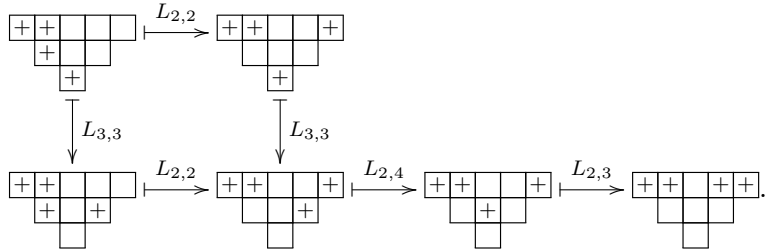
\begin{align*}
\xymatrix{
{\begin{ytableau}
+ & + & \mbox{} & \mbox{} & \mbox{} \\
\none & + & \mbox{} & \mbox{} & \none \\
\none & \none & + & \none & \none
\end{ytableau}} \ar@{|->}[d]^-{L_{3, 3}} \ar@{|->}[r]^-{L_{2, 2}} & {\begin{ytableau}
+ & + & \mbox{} & \mbox{} & + \\
\none & \mbox{} & \mbox{} & \mbox{} & \none \\
\none & \none & + & \none & \none
\end{ytableau}} \ar@{|->}[d]^-{L_{3, 3}} & & \\
{\begin{ytableau}
+ & + & \mbox{} & \mbox{} & \mbox{} \\
\none & + & \mbox{} & + & \none \\
\none & \none & \mbox{} & \none & \none
\end{ytableau}} \ar@{|->}[r]^-{L_{2, 2}} & {\begin{ytableau}
+ & + & \mbox{} & \mbox{} & + \\
\none & \mbox{} & \mbox{} & + & \none \\
\none & \none & \mbox{} & \none & \none
\end{ytableau}} \ar@{|->}[r]^-{L_{2, 4}} & {\begin{ytableau}
+ & + & \mbox{} & \mbox{} & + \\
\none & \mbox{} & + & \mbox{} & \none \\
\none & \none & \mbox{} & \none & \none
\end{ytableau}} \ar@{|->}[r]^-{L_{2, 3}} & {\begin{ytableau}
+ & + & \mbox{} & + & + \\
\none & \mbox{} & \mbox{} & \mbox{} & \none \\
\none & \none & \mbox{} & \none & \none
\end{ytableau}.} 
}
\end{align*}
\end{ex}

The notion of (transposed skew) mitosis operators for skew pipe dreams was introduced by Kiritchenko \cite[Section 5.2]{Kir}, which is an analog of (transposed) mitosis operators for pipe dreams introduced by Knutson--Miller \cite{KnM}.
The first named author \cite[Section 6]{Fuj} used an analogous operator $M_i$ for $i \in [n]$ to study string parametrizations of Demazure crystals. 
Let us recall their definitions. 
For $D \in \mathcal{SPD}_n$ and $i \in [n]$, we set 
\begin{equation}\label{eq:mitosis_removed_point}
r_0 \coloneqq \max\{1\leq r \leq N \mid q_r \in \{n-i+1, n+i-1\},\ (p_r, q_r+1) \notin D\},
\end{equation}
and consider the following condition: 
\begin{equation}\label{cond:mitosis}
\begin{aligned}
\{(p_r, q_r) \mid q_r \in \{n-i+1, n+i-1\},\ r_0 \leq r \leq N\} \subseteq D. 
\end{aligned}\tag{$\dagger$}
\end{equation}
If \eqref{cond:mitosis} is not satisfied, then set $M_i(D) \coloneqq \emptyset$. 
If \eqref{cond:mitosis} is satisfied, then the operator $M_i$ sends $D$ to the set $M_i(D)$ of elements of $\mathcal{SPD}_n$ obtained from $D \setminus \{(p_{r_0}, q_{r_0})\}$ by applying sequences of ladder moves $L_{p, q}$ such that $q \in \{n-i+1, n+i-1\}$.
Note that 
\[\{(p_r, q_r+1) \mid q_r \in \{n-i+1, n+i-1\},\ r_0+1 \leq r \leq N\} \subseteq D\]
by the definition of $r_0$.
Similarly, the \emph{transposed skew mitosis operator} ${\rm mitosis}_i^\top$ is defined as follows. 
If \eqref{cond:mitosis} is not satisfied, then set ${\rm mitosis}_i^\top(D) \coloneqq \emptyset$. 
If \eqref{cond:mitosis} is satisfied, then we write 
\[{\rm start}_i^\top (D) \coloneqq \min(\{p \mid (p, n-i+1) \in SY_n \setminus D\} \cup \{p+1 \mid (p, n+i-1) \in SY_n \setminus D\} \cup \{n-i+2\}).\]
The operator ${\rm mitosis}_i^\top$ sends $D$ to the set ${\rm mitosis}_i^\top(D)$ of elements of $\mathcal{SPD}_n$ obtained from $D \setminus \{(p_{r_0}, q_{r_0})\}$ by applying sequences of ladder moves $L_{p, q}$ such that $q \in \{n-i+1, n+i-1\}$ and such that $p < {\rm start}_i^\top (D)$. 
By definition, we have ${\rm mitosis}_i^\top(D) \subseteq M_i(D)$. 
For a subset $\mathcal{A} \subseteq \mathcal{SPD}_n$, we set $M_i(\mathcal{A}) \coloneqq \bigcup_{D \in \mathcal{A}} M_i(D)$ and ${\rm mitosis}_i^\top(\mathcal{A}) \coloneqq \bigcup_{D \in \mathcal{A}} {\rm mitosis}_i^\top(D)$. 

\begin{rem}
The transposed skew mitosis operator ${\rm mitosis}_i^\top$ is obtained from Kiritchenko's mitosis operator in \cite[Section 5.2]{Kir} by rotating her skew pipe dreams counterclockwise at 90 degrees. 
\end{rem}

For $w \in W$, we define a set $\mathscr{M}(w)$ of skew pipe dreams by 
\[\mathscr{M}(w) \coloneqq M_{i_{k_\ell}} \cdots M_{i_{k_1}} (SY_n),\]
where we write ${\bm k}^\prime_{D(w)} = (k_1, \ldots, k_\ell)$.
By the definition of $D(w)$, the set $\mathscr{M}(w)$ coincides with the set of skew pipe dreams of the form $\widetilde{L}_\ell \widetilde{L}_{\ell-1} \cdots \widetilde{L}_2 D(w)$, where $\widetilde{L}_r$ for $2 \leq r \leq \ell$ is given by a sequence of ladder moves $L_{p_t, q_t}$ such that $t < k_r$ and such that $q_t \in \{n-i_{k_r}+1, n+i_{k_r}-1\}$.
In particular, we have $\mathscr{M}(w) \subseteq \mathscr{L}(D(w))$, where $\mathscr{L}(D(w))$ denotes the set of skew pipe dreams in $\mathcal{SPD}_n$ obtained from $D(w)$ by applying sequences of ladder moves. 

\begin{ex}\label{ex:length_one_type_C}
For $i \in [n]$, we have $D(s_i) = SY_n \setminus \{(n-i+1, n+i-1)\}$ and $\mathscr{L}(D(s_i)) = \mathscr{M}(s_i) = \{D(s_i)\}$. 
\end{ex}

\begin{ex}[{see also \cite[Example 2.9]{Kir} and \cite[Example 6.6]{Fuj}}]\label{ex:skew_pipe_dreams_rank_2}
Let $n = 2$. Then it follows that 
\begin{align*}
&\mathscr{M}(s_1) = \left\{\begin{ytableau}
+ & + & +\\
\none & \mbox{} & \none 
\end{ytableau}\right\}, & &\mathscr{M}(s_2) = \left\{\begin{ytableau}
+ & + & \mbox{} \\
\none & + & \none 
\end{ytableau}\right\},\\
&\mathscr{M}(s_1 s_2) = \left\{\begin{ytableau}
+ & + & \mbox{}\\
\none & \mbox{} & \none 
\end{ytableau}\right\}, & &\mathscr{M}(s_2 s_1) = \left\{\begin{ytableau}
+ & \mbox{} & \mbox{}\\
\none & + & \none 
\end{ytableau}, \begin{ytableau}
+ & \mbox{} & + \\
\none & \mbox{} & \none 
\end{ytableau}\right\},\\
&\mathscr{M}(s_1 s_2 s_1) = \left\{\begin{ytableau}
+ & \mbox{} & \mbox{}\\
\none & \mbox{} & \none 
\end{ytableau}\right\}, & &\mathscr{M}(s_2 s_1 s_2) = \left\{\begin{ytableau}
\mbox{} & \mbox{} & \mbox{}\\
\none & + & \none 
\end{ytableau}, \begin{ytableau}
\mbox{} & \mbox{} & +\\
\none & \mbox{} & \none 
\end{ytableau}, \begin{ytableau}
\mbox{} & + & \mbox{}\\
\none & \mbox{} & \none 
\end{ytableau}\right\}.
\end{align*}
Comparing with \cref{ex:reduced_subwords_rank2}, we see that $\mathscr{M}(w)$ seems not to have good relations with reduced subwords of ${\bm i}_C$. 
\end{ex}

\begin{ex}
Let $n = 3$, and take $w \in W$ as in \cref{ex:skew_pipe_rank_3}. 
Then the set $\mathscr{M}(w)$ consists of the six skew pipe dreams given in \cref{ex:ladder_moves_type_C}.
\end{ex}

For $D \in \mathcal{SPD}_n$ and $\lambda \in P_+$, we define a face $F_D(\Delta_{{\bm i}_C}(\lambda))$ of the string polytope $\Delta_{{\bm i}_C}(\lambda)$ by $F_D(\Delta_{{\bm i}_C}(\lambda)) \coloneqq F_{{\bm k}_D}(\Delta_{{\bm i}_C}(\lambda))$. 
The first named author \cite[Theorem 6.8]{Fuj} proved that the set $\mathscr{M}(w)$ parametrizes irreducible components of the semi-toric limit of $X_w$ under the toric degeneration of $G/B$ to $Z(\Delta_{{\bm i}_C}(\lambda))$ for $\lambda \in P_{++}$. 
More precisely, the set $\Delta_{{\bm i}_C}(\lambda, X_w)$ corresponding to the semi-toric limit of $X_w$ is given by
\[\Delta_{{\bm i}_C}(\lambda, X_w) = \bigcup_{D \in \mathscr{M}(w)} F_D(\Delta_{{\bm i}_C}(\lambda)).\]

\begin{prop}\label{p:only_one_skew_pipe}
For $w \in W$, the equalities $\mathscr{L}(D(w)) = \mathscr{M}(w) = \{D(w)\}$ hold if and only if the following conditions are satisfied$:$
\begin{enumerate}
\item[(i)] there exists $1 \leq j \leq n+1$ such that 
\[m(w,1) > m(w,2) > \cdots > m(w,j-1) > m(w,j) = m(w,j+1) = \cdots = 0,\]
where $m(w,n+1) \coloneqq 0$ if $j = n+1;$
\item[(ii)] if $m(w,i) < n-i+1$ for some $1 \leq i \leq n-1$, then $m(w,i+1) = 0;$
\item[(iii)] if $m(w,i) = n-i+1$ for some $1 \leq i \leq n-1$, then $m(w,i+1)$ is $n-i$ or $0$.
\end{enumerate} 
\end{prop}

\begin{proof}
Since $D(w) \in \mathscr{M}(w) \subseteq \mathscr{L}(D(w))$, the condition $\mathscr{L}(D(w)) = \mathscr{M}(w) = \{D(w)\}$ is equivalent to the equality $\mathscr{L}(D(w)) = \{D(w)\}$. 
We first prove the ``only if'' part, that is, let us show the conditions (i)--(iii) under the assumption that $\mathscr{L}(D(w)) = \{D(w)\}$. 
If $m(w,\ell) \leq m(w,\ell+1) \neq 0$ for some $1 \leq \ell \leq n-1$, then it is obvious that the ladder move $L_{\ell+1, \ell+m(w,\ell+1)} (D(w))$ of $D(w)$ is defined, which contradicts to $\mathscr{L}(D(w)) = \{D(w)\}$. 
Hence the condition (i) is satisfied. 
If $m(w,i) < n-i+1$ for some $1 \leq i \leq n-1$ and $m(w,i+1) \neq 0$, then we can define $L_{i+1, i+m(w,i+1)} (D(w))$, which gives a contradiction. 
This proves the condition (ii). 
Similarly, we deduce the condition (iii). 
The ``if'' part of the proposition is an immediate consequence of the definition of ladder moves. 
This proves the proposition. 
\end{proof}

\begin{rem}\label{r:toric_degeneration_X_w}
Recall that the set $\mathscr{M}(w)$ parametrizes irreducible components of the semi-toric limit of $X_w$ under the toric degeneration of $G/B$ to $Z(\Delta_{{\bm i}_C}(\lambda))$ for $\lambda \in P_{++}$. 
Hence if $w \in W$ satisfies the conditions in \cref{p:only_one_skew_pipe}, then the semi-toric limit of $X_w$ is irreducible, that is, it is a toric degeneration of $X_w$. 
\end{rem}

\begin{ex}
Fix $1 \leq k \leq N$, and set $w \coloneqq s_{i_1} s_{i_2} \cdots s_{i_k} \in W$. 
Then it is obvious that $w$ satisfies the conditions in \cref{p:only_one_skew_pipe}.
In addition, the face of $\Delta_{{\bm i}_C}(\lambda)$ corresponding to $D(w)$ is precisely the string polytope associated with $\lambda$ and the reduced word $(i_1, i_2, \ldots, i_k)$ for $w$ (see \cite[Section 1]{Lit}). 
Hence the toric degeneration of $X_w$ in \cref{r:toric_degeneration_X_w} is precisely the one given in \cite[Section 3]{Cal}. 
\end{ex}

\begin{ex}
Let $n = 3$, and $w = s_1 s_2 s_1 s_3 s_2 \in W$. 
Then we have 
\[D(w) = \begin{ytableau}
+ & + & + & \mbox{} & \mbox{} \\
\none & + & \mbox{} & \mbox{} & \none \\
\none & \none & \mbox{} & \none & \none
\end{ytableau},\]
which satisfies the conditions (i) and (ii). 
However, the ladder move 
\[L_{2, 2} (D(w)) = \begin{ytableau}
+ & + & + & \mbox{} & + \\
\none & \mbox{} & \mbox{} & \mbox{} & \none \\
\none & \none & \mbox{} & \none & \none
\end{ytableau}\]
is defined since $w$ does not satisfy the condition (iii). 
\end{ex}

\begin{defi}[{cf.\ \cite[Definition 18]{Mil}}]
Take $j_1, j_2, \ldots, j_m \in I$ such that $\ell(s_{j_1} \cdots s_{j_k}) = k$ for all $1 \leq k \leq m$.
Then the corresponding sequence $(e, s_{j_1}, \ldots, s_{j_1} \cdots s_{j_m})$ in $W$ is said to be \emph{poptotic} if ${\rm mitosis}_{j_k}^\top (D) \neq \emptyset$ for all $1 \leq k \leq m$ and $D \in {\rm mitosis}_{j_{k-1}}^\top \cdots {\rm mitosis}_{j_1}^\top (SY_n)$.
\end{defi}

By the definition of $D(w)$, we obtain the following. 

\begin{prop}[{cf.\ \cite[Proposition 19]{Mil}}]
For $w \in W$, write ${\bm k}_{D(w)}^\prime = {\bm k}_w = (k_1, \ldots, k_\ell)$. 
Then the sequence $(e, s_{i_{k_1}}, \ldots, s_{i_{k_1}} \cdots s_{i_{k_\ell}} = w)$ is poptotic. 
In addition, it holds that $M_{j_r} (D) \neq \emptyset$ for all $1 \leq r \leq \ell$ and $D \in M_{j_{r-1}} \cdots M_{j_1} (SY_n)$.
\end{prop}

\section{A path model to skew pipe dreams}

In this section, we introduce a path model to skew pipe dreams. 
Let $\widetilde{SY}_n$ be the set of $(i, j) \in \z^2$ such that $1 \leq j \leq n$ and such that $1 \leq i \leq 2j-1$.
We visualize $\widetilde{SY}_n$ in a way similar to $SY_n$. 
For instance, write $\widetilde{SY}_3$ as 
\[\widetilde{SY}_3 = \begin{ytableau}
\mbox{} & \mbox{} & \mbox{} \\
\none & \mbox{} & \mbox{} \\
\none & \mbox{} & \mbox{} \\
\none & \none & \mbox{} \\
\none & \none & \mbox{}
\end{ytableau}.\]
Note that the Gelfand--Tsetlin pattern for $G = Sp_{2n}(\c)$ in \cite[Section 6]{Lit} is arranged as the boxes in $\widetilde{SY}_n$. 
Denote by $\widetilde{\mathcal{SPD}}_n$ the power set of $\widetilde{SY}_n$, and describe $D \in \widetilde{\mathcal{SPD}}_n$ by putting $+$ in the boxes corresponding to the elements of $D$. 
For each $D \in \mathcal{SPD}_n$, we define its \emph{Gelfand--Tsetlin type rearrangement} $\Omega(D) \in \widetilde{\mathcal{SPD}}_n$ as follows. 
For $k \in \z_{>0}$, the $(2k-1)$-st row of $\Omega(D)$ is defined to be the first $(n-k+1)$ boxes in the $k$-th row of $D$. 
The $2k$-th row of $\Omega(D)$ is given by reversing the last $(n-k)$ boxes in the $k$-th row of $D$. 

\begin{ex}\label{ex:GT_type_skew_pipe}
Let $n = 3$, and $D = \{(1, 1), (1, 2), (1, 5), (2, 2), (2, 3), (3, 3)\} \in \mathcal{SPD}_3$. 
Then we have 
\[D = \begin{ytableau}
+ & + & \mbox{} & \mbox{} & + \\
\none & + & + & \mbox{} & \none \\
\none & \none & + & \none & \none 
\end{ytableau} \xmapsto{\Omega} \Omega(D) = \begin{ytableau}
+ & + & \mbox{} \\
\none & + & \mbox{} \\
\none & + & + \\
\none & \none & \mbox{} \\
\none & \none & +
\end{ytableau}.\]
\end{ex}

\begin{defi}\label{d:path_model_skew}
Set $\widetilde{SY}_n^{\rm (ex)} \coloneqq \widetilde{SY}_n \sqcup \{(2k, k) \mid 1 \leq k \leq n\} \sqcup \{(k,n+1) \mid 1 \leq k \leq 2n\} \subseteq \z_{>0}^2$ that is an extension of $\widetilde{SY}_n$. 
Let us define a \emph{path model} $\mathscr{G}(D)$ of a skew pipe dream $D \in \mathcal{SPD}_n$ by representing $\Omega(D)$ using a diagram of ``pipes'' as follows. 
\begin{itemize}
\item We first replace all boxes $\begin{ytableau}
+
\end{ytableau}$ in $\Omega(D)$ with ``crossings''
\scalebox{0.12}{
	\begin{tikzpicture}

   \fill (0,1) coordinate (1) node[left=10pt]  {};
   \fill (-1,1) coordinate (2) node[left=10pt]  {};
   \fill (1,2) coordinate (3) node[right=3pt]  {};
   \fill (1,3) coordinate (4) node[right=3pt]  {};
   \fill (1,0) coordinate (5) node[right=3pt]  {};
   \fill (1,-1) coordinate (6) node[left=10pt]  {};
   \fill (2,1) coordinate (7) node[right=3pt]  {};
   \fill (3,1) coordinate (8) node[right=3pt]  {};
   \fill (0,1) coordinate (9) node[left=10pt]  {};
   \fill (1,0) coordinate (10) node[right=3pt]  {};
   \fill (1,2) coordinate (12) node[right=3pt]  {};
   \fill (2,1) coordinate (11) node[right=3pt]  {};
%   \fill (4,5) coordinate (12) node[right=3pt]  {};
%   \fill (4,8) coordinate (13) node[right=3pt]  {};
%   \fill (5,6) coordinate (14) node[right=3pt]  {};
%   \fill (5,7) coordinate (15) node[right=3pt]  {};
%   \fill (5,9) coordinate (16) node[above=3pt]  {\Huge \bf $2$};
%   \fill (8,9) coordinate (17) node[above=3pt]  {\Huge \bf $3$};
%   \fill (6,8) coordinate (18) node[right=3pt]  {};
%   \fill (7,8) coordinate (19) node[right=3pt]  {};
	
	\draw[ultra thick] (1)--(2);
	\draw[ultra thick] (3)--(4);
	\draw[ultra thick] (7)--(8);
	\draw[ultra thick] (6)--(5);
	\draw[ultra thick] (9)--(11);
	\draw[ultra thick] (10)--(12);
%	\draw[thick] (13)--(10);
%	\draw[thick] (14)--(15);
%	\draw[thick] (18)--(19);

%\draw[thick] ([shift={(1,3)}]270:1.0) arc[radius=1.0, start angle=270, end angle= 360];
%\draw[thick] ([shift={(6,7)}]90:1.0) arc[radius=1.0, start angle=90, end angle= 180];
%\draw[thick] ([shift={(4,6)}]270:1.0) arc[radius=1.0, start angle=270, end angle= 360];
%\draw[thick] ([shift={(4,7)}]90:1.0) arc[radius=1.0, start angle=270, end angle= 360];
%\draw[thick] ([shift={(7,7)}]90:1.0) arc[radius=1.0, start angle=270, end angle= 360];

	\end{tikzpicture}
}.
%\begin{minipage}[b]{0.06\linewidth}
%\includegraphics[height=0.6cm,width=0.6cm,bb=82mm 189mm 127mm 231mm,clip]{skew_pipe_dream_2.pdf}\end{minipage}.
\item For $(i, j) \in \widetilde{SY}_n^{\rm (ex)} \setminus \Omega(D)$ with $j \leq n$ and $i$ is odd or with $j = n+1$ and $i$ is even, the $(i, j)$-th empty box is changed to an ``elbow joint'' 
\scalebox{0.12}{
	\begin{tikzpicture}

   \fill (0,1) coordinate (1) node[left=10pt]  {};
   \fill (-1,1) coordinate (2) node[left=10pt]  {};
   \fill (1,2) coordinate (3) node[right=3pt]  {};
   \fill (1,3) coordinate (4) node[right=3pt]  {};
   \fill (1,0) coordinate (5) node[right=3pt]  {};
   \fill (1,-1) coordinate (6) node[left=10pt]  {};
   \fill (2,1) coordinate (7) node[right=3pt]  {};
   \fill (3,1) coordinate (8) node[right=3pt]  {};
	
	\draw[ultra thick] (1)--(2);
	\draw[ultra thick] (3)--(4);
	\draw[ultra thick] (7)--(8);
	\draw[ultra thick] (6)--(5);

\draw[ultra thick] ([shift={(0,2)}]270:1.0) arc[radius=1.0, start angle=270, end angle= 360];
\draw[ultra thick] ([shift={(2,0)}]90:1.0) arc[radius=1.0, start angle=90, end angle= 180];

	\end{tikzpicture}
}.
%\begin{minipage}[b]{0.06\linewidth}
%\includegraphics[height=0.6cm,width=0.6cm,bb=82mm 189mm 127mm 231mm,clip]{skew_pipe_dream_3.pdf}\end{minipage}.
\item For $(i, j) \in \widetilde{SY}_n^{\rm (ex)} \setminus \Omega(D)$ with $j \leq n$ and $i$ is even or with $j = n+1$ and $i$ is odd, the $(i, j)$-th empty box is replaced with a ``reversed elbow joint'' 
\scalebox{0.12}{
	\begin{tikzpicture}

   \fill (0,1) coordinate (1) node[left=10pt]  {};
   \fill (-1,1) coordinate (2) node[left=10pt]  {};
   \fill (1,2) coordinate (3) node[right=3pt]  {};
   \fill (1,3) coordinate (4) node[right=3pt]  {};
   \fill (1,0) coordinate (5) node[right=3pt]  {};
   \fill (1,-1) coordinate (6) node[left=10pt]  {};
   \fill (2,1) coordinate (7) node[right=3pt]  {};
   \fill (3,1) coordinate (8) node[right=3pt]  {};
	
	\draw[ultra thick] (1)--(2);
	\draw[ultra thick] (3)--(4);
	\draw[ultra thick] (7)--(8);
	\draw[ultra thick] (6)--(5);

\draw[ultra thick] ([shift={(1,1)}]270:1.0) arc[radius=1.0, start angle=0, end angle= 90];
\draw[ultra thick] ([shift={(1,1)}]90:1.0) arc[radius=1.0, start angle=180, end angle= 270];

	\end{tikzpicture}
}.
%\begin{minipage}[b]{0.06\linewidth}
%\includegraphics[height=0.6cm,width=0.6cm,bb=82mm 189mm 127mm 231mm,clip]{skew_pipe_dream_4.pdf}\end{minipage}.
\end{itemize}
For each $1 \leq k \leq n$, let $\ell_k$ denote the pipe in $\mathscr{G}(D)$ that has an edge at the left end of the $(2k-1)$-st row of $\widetilde{SY}_n^{\rm (ex)}$. 
\end{defi}

By definition, there exists a unique permutation $v_D \colon [n] \rightarrow [n]$ such that the pipe $\ell_k$ has an edge at the top of the $(n+1-v_D(k))$-th column of $\widetilde{SY}_n^{\rm (ex)}$. 
We write $\ell_k = L_{v_D(k)}$. 

\begin{defi}
For $D \in \mathcal{SPD}_n$, we define $w_D \in W$ by 
\[w_D(j) \coloneqq (-1)^{{\rm sign}(L_j)+1} v_D^{-1}(j)\]
and $w_D(-j) = -w_D(j)$ for $j \in [n]$, where ${\rm sign}(L_j)$ denotes the number of self-crossings of the pipe $L_j$ in $\mathscr{G}(D)$.
The skew pipe dream $D$ is said to be \emph{reduced} if $|D| = \ell(w_0)-\ell(w_D)$. 
\end{defi}

By definition, we have $\ell_{|w_D(j)|} = L_j$ for all $j \in [n]$.

\begin{ex}
Let $n = 3$, and $D = \{(1, 1), (1, 2), (1, 5), (2, 2), (2, 3), (3, 3)\} \in \mathcal{SPD}_3$ as in \cref{ex:GT_type_skew_pipe}.
Then the diagram $\mathscr{G}(D)$ is given as in Figure \ref{fig:path_model}.
Since $\ell_1 = L_2$, $\ell_2 = L_3$, and $\ell_3 = L_1$, we have 
\[v_D = \begin{pmatrix}
1 & 2 & 3\\
2 & 3 & 1
\end{pmatrix}\quad \text{and}\quad
w_D = 
\begin{pmatrix}
-3 & -2 & -1 & 1 & 2 & 3\\
-2 & -1 & 3 & -3 & 1 & 2
\end{pmatrix}.\]
In particular, it follows that $\ell(w_0)-\ell(w_D) = 9-3 = 6 = |D|$. 
Hence $D$ is reduced. 
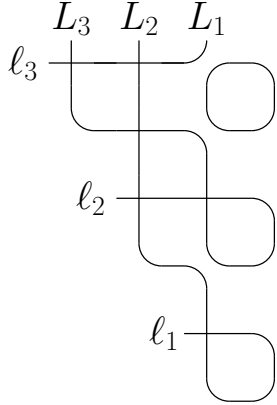
\begin{figure}[h]
\begin{center}
%\tdplotsetmaincoords{60}{60}
\scalebox{0.3}{
	\begin{tikzpicture}

%   \fill (1,5) coordinate (1) node[left=10pt]  {};
   \fill (4,2) coordinate (2) node[left=10pt]  {\fontsize{50pt}{15pt}\selectfont \bf $\ell_2$};
   \fill (2,2) coordinate (3) node[right=3pt]  {};
   \fill (2,3) coordinate (4) node[right=3pt]  {};
   \fill (2,4) coordinate (5) node[right=3pt]  {};
   \fill (1,8) coordinate (6) node[left=10pt]  {\fontsize{50pt}{15pt}\selectfont \bf $\ell_3$};
   \fill (2,6) coordinate (7) node[right=3pt]  {};
   \fill (2,7) coordinate (8) node[right=3pt]  {};
   \fill (2,9) coordinate (9) node[above=3pt]  {\fontsize{50pt}{15pt}\selectfont \bf $L_3$};
   \fill (3,8) coordinate (10) node[right=3pt]  {};
   \fill (3,5) coordinate (11) node[right=3pt]  {};
   \fill (4,5) coordinate (12) node[right=3pt]  {};
   \fill (4,8) coordinate (13) node[right=3pt]  {};
   \fill (5,6) coordinate (14) node[right=3pt]  {};
   \fill (5,7) coordinate (15) node[right=3pt]  {};
   \fill (5,9) coordinate (16) node[above=3pt]  {\fontsize{50pt}{15pt}\selectfont \bf $L_2$};
   \fill (8,9) coordinate (17) node[above=3pt]  {\fontsize{50pt}{15pt}\selectfont \bf $L_1$};
   \fill (6,8) coordinate (18) node[right=3pt]  {};
   \fill (7,8) coordinate (19) node[right=3pt]  {};
   \fill (5,3) coordinate (20) node[right=3pt]  {};
   \fill (8,7) coordinate (21) node[right=3pt]  {};
   \fill (8,6) coordinate (22) node[right=3pt]  {};
   \fill (9,8) coordinate (23) node[right=3pt]  {};
   \fill (10,8) coordinate (24) node[right=3pt]  {};
   \fill (9,5) coordinate (25) node[right=3pt]  {};
   \fill (10,5) coordinate (26) node[right=3pt]  {};
   \fill (11,7) coordinate (27) node[right=3pt]  {};
   \fill (11,6) coordinate (28) node[right=3pt]  {};
   \fill (7,5) coordinate (29) node[right=3pt]  {};
   \fill (7,-4) coordinate (30) node[left=3pt]  {\fontsize{50pt}{15pt}\selectfont \bf $\ell_1$};
   \fill (5,0) coordinate (31) node[right=3pt]  {};
   \fill (10,2) coordinate (32) node[right=3pt]  {};
   \fill (8,4) coordinate (33) node[right=3pt]  {};
   \fill (8,0) coordinate (34) node[right=3pt]  {};
   \fill (6,-1) coordinate (35) node[right=3pt]  {};
   \fill (7,-1) coordinate (36) node[right=3pt]  {};
   \fill (9,-1) coordinate (37) node[right=3pt]  {};
   \fill (10,-1) coordinate (38) node[right=3pt]  {};
   \fill (11,0) coordinate (39) node[right=3pt]  {};
   \fill (11,1) coordinate (40) node[right=3pt]  {};
  \fill (8,-2) coordinate (41) node[right=3pt]  {};
  \fill (8,-6) coordinate (42) node[right=3pt]  {};
  \fill (10,-4) coordinate (43) node[right=3pt]  {};
  \fill (9,-7) coordinate (44) node[right=3pt]  {};
  \fill (10,-7) coordinate (45) node[right=3pt]  {};
  \fill (11,-6) coordinate (46) node[right=3pt]  {};
  \fill (11,-5) coordinate (47) node[right=3pt]  {};

	\draw[ultra thick] (7)--(8);
	\draw[ultra thick] (9)--(8);
	\draw[ultra thick] (6)--(10);
	\draw[ultra thick] (10)--(19);
	\draw[ultra thick] (11)--(12);
	\draw[ultra thick] (21)--(22);
	\draw[ultra thick] (23)--(24);
	\draw[ultra thick] (25)--(26);
	\draw[ultra thick] (27)--(28);
	\draw[ultra thick] (13)--(10);
	\draw[ultra thick] (14)--(15);
	\draw[ultra thick] (29)--(12);
	\draw[ultra thick] (14)--(20);
	\draw[ultra thick] (31)--(20);
	\draw[ultra thick] (33)--(34);
	\draw[ultra thick] (35)--(36);
	\draw[ultra thick] (38)--(37);
	\draw[ultra thick] (39)--(40);
	\draw[ultra thick] (16)--(15);
	\draw[ultra thick] (18)--(19);
	\draw[ultra thick] (30)--(43);
	\draw[ultra thick] (32)--(2);
	\draw[ultra thick] (41)--(42);
	\draw[ultra thick] (44)--(45);
	\draw[ultra thick] (47)--(46);

%\draw[ultra thick] ([shift={(1,3)}]270:1.0) arc[radius=1.0, start angle=270, end angle= 360];
%\draw[ultra thick] ([shift={(6,7)}]90:1.0) arc[radius=1.0, start angle=90, end angle= 180];
\draw[ultra thick] ([shift={(9,7)}]90:1.0) arc[radius=1.0, start angle=90, end angle= 180];
\draw[ultra thick] ([shift={(8,5)}]90:1.0) arc[radius=1.0, start angle=180, end angle= 270];
\draw[ultra thick] ([shift={(11,6)}]90:1.0) arc[radius=1.0, start angle=0, end angle= 90];
\draw[ultra thick] ([shift={(10,4)}]90:1.0) arc[radius=1.0, start angle=270, end angle= 360];
\draw[ultra thick] ([shift={(11,0)}]90:1.0) arc[radius=1.0, start angle=0, end angle= 90];
\draw[ultra thick] ([shift={(8,-1)}]90:1.0) arc[radius=1.0, start angle=180, end angle= 270];
\draw[ultra thick] ([shift={(10,-2)}]90:1.0) arc[radius=1.0, start angle=270, end angle= 360];
\draw[ultra thick] ([shift={(8,3)}]90:1.0) arc[radius=1.0, start angle=0, end angle= 90];
\draw[ultra thick] ([shift={(5,-1)}]90:1.0) arc[radius=1.0, start angle=180, end angle= 270];
\draw[ultra thick] ([shift={(2,5)}]90:1.0) arc[radius=1.0, start angle=180, end angle= 270];
\draw[ultra thick] ([shift={(7,7)}]90:1.0) arc[radius=1.0, start angle=270, end angle= 360];
\draw[ultra thick] ([shift={(8,-3)}]90:1.0) arc[radius=1.0, start angle=0, end angle= 90];
\draw[ultra thick] ([shift={(11,-6)}]90:1.0) arc[radius=1.0, start angle=0, end angle= 90];
\draw[ultra thick] ([shift={(8,-7)}]90:1.0) arc[radius=1.0, start angle=180, end angle= 270];
\draw[ultra thick] ([shift={(10,-8)}]90:1.0) arc[radius=1.0, start angle=270, end angle= 360];

	\end{tikzpicture}
}
\end{center}
%\centering
%\includegraphics[width=4.0cm,bb=45mm 45mm 170mm 230mm,clip]{skew_pipe_dream_1.pdf}
\caption{\label{fig:path_model} The diagram $\mathscr{G}(D)$ for $D = \{(1, 1), (1, 2), (1, 5), (2, 2), (2, 3), (3, 3)\} \in \mathcal{SPD}_3$.}
\end{figure}
\end{ex}

The following is the main result of this section. 

\begin{thm}\label{t:main_result_1}
For all $w \in W$, the set $\mathscr{L}(D(w))$ coincides with the set of reduced skew pipe dreams $D \in \mathcal{SPD}_n$ such that $w_D = w$.
\end{thm}

To prove \cref{t:main_result_1}, we prepare some lemmas. 

\begin{lem}\label{l:reduced_D(w)}
For every $w \in W$, the skew pipe dream $D(w)$ is reduced, and the equality $w_{D(w)} = w$ holds. 
\end{lem}

\begin{proof}
By the definition of $D(w)$, it follows that $|D(w)| = \ell(w_0) - \ell(w)$. 
Hence it suffices to show that $w_{D(w)} = w$. 
We write $w = w_1 w_2 \cdots w_n$ as in the proof of \cref{p:description_D(w)}. 
For each $1 \leq j \leq n$, focus on the $(2j-1)$-st and $2j$-th rows of $\Omega(D(w))$. 
Then we define a subdiagram $\mathscr{G}^{(j)}(D(w))$ of $\mathscr{G}(D(w))$ by forgetting the pieces corresponding to other rows. 
The subdiagram $\mathscr{G}^{(j)}(D(w))$ is given as in Figure \ref{fig:path_model_j}. 
In particular, $\ell_j$ does not have a self-crossing if and only if $w_j^{-1}(j) < 0$. 
Define pipes $L_1^{(j)}, \ldots, L_j^{(j)}, \ell_1^{(j)}, \ldots, \ell_{j-1}^{(j)}$ in $\mathscr{G}^{(j)}(D(w))$ as in Figure \ref{fig:path_model_j}. 
Then it follows by Figure \ref{fig:path_model_j} that $\ell_j = L_{\sigma_j}^{(j)}$ and that $\ell_k^{(j)} = L^{(j)}_{w_j^{-1}(k)}$ for all $1 \leq k \leq j-1$, where $\sigma_j \coloneqq |w_j^{-1}(j)|$.
Hence we have $|w_{D(w)}(k)| = |w(k)|$ for all $k \in [n]$. 
In addition, since $w^{-1}(j) < 0$ if and only if $w_j^{-1}(j) < 0$, it holds that $w^{-1}(j) < 0$ if and only if $\ell_j$ does not have a self-crossing, which is also equivalent to the condition that $w_{D(w)}^{-1}(j) < 0$. 
From these, we deduce that $w_{D(w)} = w$. 
This proves the lemma. 
\begin{figure}[h]
\begin{minipage}[b]{0.4\linewidth}
\begin{center}
%\tdplotsetmaincoords{60}{60}
\scalebox{0.25}{
	\begin{tikzpicture}

   \fill (-4,8) coordinate (1) node[left=10pt]  {\fontsize{50pt}{15pt}\selectfont \bf $\ell_j$};
   \fill (-3,6) coordinate (2) node[right=3pt]  {};
   \fill (-3,7) coordinate (3) node[right=3pt]  {};
   \fill (-3,9) coordinate (4) node[above=3pt]  {\fontsize{50pt}{15pt}\selectfont \bf $L_j^{(j)}$};
   \fill (-2,8) coordinate (5) node[right=3pt]  {};
   \fill (-2,5) coordinate (6) node[right=3pt]  {};
   \fill (-1,5) coordinate (7) node[right=3pt]  {};
   \fill (-1,8) coordinate (8) node[right=3pt]  {};
   \fill (0,4) coordinate (9) node[right=3pt]  {};
   \fill (0,6) coordinate (10) node[right=3pt]  {};
   \fill (0,7) coordinate (11) node[right=3pt]  {};
   \fill (0,9) coordinate (12) node[above=3pt]  {};
   \fill (3,9) coordinate (13) node[above=3pt]  {\fontsize{50pt}{15pt}\selectfont \bf $L_{\sigma_j}^{(j)}$};
   \fill (3,4) coordinate (14) node[right=3pt]  {};
   \fill (3,3) coordinate (15) node[right=3pt]  {};
   \fill (1,5) coordinate (16) node[right=3pt]  {};
   \fill (2,5) coordinate (17) node[right=3pt]  {};
   \fill (1,8) coordinate (18) node[right=3pt]  {};
   \fill (2,8) coordinate (19) node[right=3pt]  {};
   \fill (0,3) coordinate (20) node[below=3pt]  {\fontsize{50pt}{15pt}\selectfont \bf $\ell_{j-1}^{(j)}$};
   \fill (12,7) coordinate (21) node[right=3pt]  {};
   \fill (12,6) coordinate (22) node[right=3pt]  {};
   \fill (13,8) coordinate (23) node[right=3pt]  {};
   \fill (14,8) coordinate (24) node[right=3pt]  {};
   \fill (13,5) coordinate (25) node[right=3pt]  {};
   \fill (14,5) coordinate (26) node[right=3pt]  {};
   \fill (15,7) coordinate (27) node[right=3pt]  {};
   \fill (15,6) coordinate (28) node[right=3pt]  {};
   \fill (4,5) coordinate (29) node[right=3pt]  {};
   \fill (5,5) coordinate (30) node[right=3pt]  {};
   \fill (7,5) coordinate (31) node[right=3pt]  {};
   \fill (8,5) coordinate (32) node[right=3pt]  {};
   \fill (10,5) coordinate (33) node[right=3pt]  {};
   \fill (11,5) coordinate (34) node[right=3pt]  {};
   \fill (6,4) coordinate (35) node[right=3pt]  {};
   \fill (6,3) coordinate (36) node[right=3pt]  {};
   \fill (9,4) coordinate (37) node[right=3pt]  {};
   \fill (9,3) coordinate (38) node[below=3pt]  {\fontsize{50pt}{15pt}\selectfont \bf $\ell_2^{(j)}$};
   \fill (12,4) coordinate (39) node[right=3pt]  {};
   \fill (12,3) coordinate (40) node[below=3pt]  {\fontsize{50pt}{15pt}\selectfont \bf $\ell_1^{(j)}$};
   \fill (3,6) coordinate (41) node[right=3pt]  {};
   \fill (3,7) coordinate (42) node[right=3pt]  {};
   \fill (6,6) coordinate (43) node[right=3pt]  {};
   \fill (6,7) coordinate (44) node[right=3pt]  {};
   \fill (9,6) coordinate (45) node[right=3pt]  {};
   \fill (9,7) coordinate (46) node[right=3pt]  {};
   \fill (4,8) coordinate (47) node[right=3pt]  {};
   \fill (5,8) coordinate (48) node[right=3pt]  {};
   \fill (7,8) coordinate (49) node[right=3pt]  {};
   \fill (8,8) coordinate (50) node[right=3pt]  {};
   \fill (10,8) coordinate (51) node[right=3pt]  {};
   \fill (11,8) coordinate (52) node[right=3pt]  {};
   \fill (9,9) coordinate (53) node[above=3pt]  {\fontsize{50pt}{15pt}\selectfont \bf $L_2^{(j)}$};
   \fill (12,9) coordinate (54) node[above=3pt]  {\fontsize{50pt}{15pt}\selectfont \bf $L_1^{(j)}$};

	\draw[ultra thick] (2)--(3);
	\draw[ultra thick] (4)--(3);
	\draw[ultra thick] (1)--(5);
	\draw[ultra thick] (5)--(19);
	\draw[ultra thick] (6)--(7);
	\draw[ultra thick] (21)--(22);
	\draw[ultra thick] (23)--(24);
	\draw[ultra thick] (25)--(26);
	\draw[ultra thick] (27)--(28);
	\draw[ultra thick] (8)--(5);
	\draw[ultra thick] (14)--(15);
	\draw[ultra thick] (10)--(11);
	\draw[ultra thick] (9)--(20);
	\draw[ultra thick] (12)--(11);
	\draw[ultra thick] (18)--(19);
	\draw[ultra thick] (16)--(17);
	\draw[ultra thick] (30)--(29);
	\draw[ultra thick] (31)--(32);
	\draw[ultra thick] (33)--(34);
	\draw[ultra thick] (35)--(36);
	\draw[ultra thick] (37)--(38);
	\draw[ultra thick] (40)--(39);
	\draw[ultra thick] (41)--(42);
	\draw[ultra thick] (43)--(44);
	\draw[ultra thick] (45)--(46);
	\draw[ultra thick] (47)--(48);
	\draw[ultra thick] (49)--(50);
	\draw[ultra thick] (51)--(52);

\draw[ultra thick] ([shift={(13,7)}]90:1.0) arc[radius=1.0, start angle=90, end angle= 180];
\draw[ultra thick] ([shift={(12,5)}]90:1.0) arc[radius=1.0, start angle=180, end angle= 270];
\draw[ultra thick] ([shift={(15,6)}]90:1.0) arc[radius=1.0, start angle=0, end angle= 90];
\draw[ultra thick] ([shift={(14,4)}]90:1.0) arc[radius=1.0, start angle=270, end angle= 360];
\draw[ultra thick] ([shift={(3,3)}]90:1.0) arc[radius=1.0, start angle=0, end angle= 90];
\draw[ultra thick] ([shift={(-3,5)}]90:1.0) arc[radius=1.0, start angle=180, end angle= 270];
\draw[ultra thick] ([shift={(0,5)}]90:1.0) arc[radius=1.0, start angle=180, end angle= 270];
\draw[ultra thick] ([shift={(0,3)}]90:1.0) arc[radius=1.0, start angle=0, end angle= 90];
\draw[ultra thick] ([shift={(2,7)}]90:1.0) arc[radius=1.0, start angle=270, end angle= 360];
\draw[ultra thick] ([shift={(11,7)}]90:1.0) arc[radius=1.0, start angle=270, end angle= 360];
\draw[ultra thick] ([shift={(6,3)}]90:1.0) arc[radius=1.0, start angle=0, end angle= 90];
\draw[ultra thick] ([shift={(3,5)}]90:1.0) arc[radius=1.0, start angle=180, end angle= 270];
\draw[ultra thick] ([shift={(9,3)}]90:1.0) arc[radius=1.0, start angle=0, end angle= 90];
\draw[ultra thick] ([shift={(6,5)}]90:1.0) arc[radius=1.0, start angle=180, end angle= 270];
\draw[ultra thick] ([shift={(12,3)}]90:1.0) arc[radius=1.0, start angle=0, end angle= 90];
\draw[ultra thick] ([shift={(9,5)}]90:1.0) arc[radius=1.0, start angle=180, end angle= 270];
\draw[ultra thick] ([shift={(10,7)}]90:1.0) arc[radius=1.0, start angle=90, end angle= 180];
\draw[ultra thick] ([shift={(7,7)}]90:1.0) arc[radius=1.0, start angle=90, end angle= 180];
\draw[ultra thick] ([shift={(4,7)}]90:1.0) arc[radius=1.0, start angle=90, end angle= 180];
\draw[ultra thick] ([shift={(8,7)}]90:1.0) arc[radius=1.0, start angle=270, end angle= 360];
\draw[ultra thick] ([shift={(5,7)}]90:1.0) arc[radius=1.0, start angle=270, end angle= 360];

	\end{tikzpicture}
}
\end{center}

%\centering
%\includegraphics[width=6.0cm,bb=45mm 40mm 320mm 170mm,clip]{skew_pipe_dream_5.pdf}
\subcaption{$w_j^{-1}(j) < 0$.}
\end{minipage}
\begin{minipage}[b]{0.4\linewidth}
\begin{center}
\scalebox{0.25}{
%\tdplotsetmaincoords{60}{60}
	\begin{tikzpicture}

   \fill (-4,8) coordinate (1) node[left=10pt]  {\fontsize{50pt}{15pt}\selectfont \bf $\ell_j$};
   \fill (-3,6) coordinate (2) node[right=3pt]  {};
   \fill (-3,7) coordinate (3) node[right=3pt]  {};
   \fill (-3,9) coordinate (4) node[above=3pt]  {\fontsize{50pt}{15pt}\selectfont \bf $L_j^{(j)}$};
   \fill (-2,8) coordinate (5) node[right=3pt]  {};
   \fill (-2,5) coordinate (6) node[right=3pt]  {};
   \fill (-1,5) coordinate (7) node[right=3pt]  {};
   \fill (-1,8) coordinate (8) node[right=3pt]  {};
   \fill (0,4) coordinate (9) node[right=3pt]  {};
   \fill (0,6) coordinate (10) node[right=3pt]  {};
   \fill (0,7) coordinate (11) node[right=3pt]  {};
   \fill (0,9) coordinate (12) node[above=3pt]  {};
   \fill (3,9) coordinate (13) node[above=3pt]  {\fontsize{50pt}{15pt}\selectfont \bf $L_{\sigma_j}^{(j)}$};
   \fill (3,4) coordinate (14) node[right=3pt]  {};
   \fill (3,3) coordinate (15) node[right=3pt]  {};
   \fill (1,5) coordinate (16) node[right=3pt]  {};
   \fill (2,5) coordinate (17) node[right=3pt]  {};
   \fill (1,8) coordinate (18) node[right=3pt]  {};
   \fill (2,8) coordinate (19) node[right=3pt]  {};
   \fill (0,3) coordinate (20) node[below=3pt]  {\fontsize{50pt}{15pt}\selectfont \bf $\ell_{j-1}^{(j)}$};
   \fill (12,7) coordinate (21) node[right=3pt]  {};
   \fill (12,6) coordinate (22) node[right=3pt]  {};
   \fill (13,8) coordinate (23) node[right=3pt]  {};
   \fill (14,8) coordinate (24) node[right=3pt]  {};
   \fill (13,5) coordinate (25) node[right=3pt]  {};
   \fill (14,5) coordinate (26) node[right=3pt]  {};
   \fill (15,7) coordinate (27) node[right=3pt]  {};
   \fill (15,6) coordinate (28) node[right=3pt]  {};
   \fill (4,5) coordinate (29) node[right=3pt]  {};
   \fill (5,5) coordinate (30) node[right=3pt]  {};
   \fill (7,5) coordinate (31) node[right=3pt]  {};
   \fill (8,5) coordinate (32) node[right=3pt]  {};
   \fill (10,5) coordinate (33) node[right=3pt]  {};
   \fill (11,5) coordinate (34) node[right=3pt]  {};
   \fill (6,4) coordinate (35) node[right=3pt]  {};
   \fill (6,3) coordinate (36) node[right=3pt]  {};
   \fill (9,4) coordinate (37) node[right=3pt]  {};
   \fill (9,3) coordinate (38) node[below=3pt]  {\fontsize{50pt}{15pt}\selectfont \bf $\ell_2^{(j)}$};
   \fill (12,4) coordinate (39) node[right=3pt]  {};
   \fill (12,3) coordinate (40) node[below=3pt]  {\fontsize{50pt}{15pt}\selectfont \bf $\ell_1^{(j)}$};
   \fill (3,6) coordinate (41) node[right=3pt]  {};
   \fill (3,7) coordinate (42) node[right=3pt]  {};
   \fill (6,6) coordinate (43) node[right=3pt]  {};
   \fill (6,7) coordinate (44) node[right=3pt]  {};
   \fill (9,6) coordinate (45) node[right=3pt]  {};
   \fill (9,7) coordinate (46) node[right=3pt]  {};
   \fill (4,8) coordinate (47) node[right=3pt]  {};
   \fill (5,8) coordinate (48) node[right=3pt]  {};
   \fill (7,8) coordinate (49) node[right=3pt]  {};
   \fill (8,8) coordinate (50) node[right=3pt]  {};
   \fill (10,8) coordinate (51) node[right=3pt]  {};
   \fill (11,8) coordinate (52) node[right=3pt]  {};
   \fill (9,9) coordinate (53) node[above=3pt]  {\fontsize{50pt}{15pt}\selectfont \bf $L_2^{(j)}$};
   \fill (12,9) coordinate (54) node[above=3pt]  {\fontsize{50pt}{15pt}\selectfont \bf $L_1^{(j)}$};
  \fill (6,9) coordinate (55) node[right=3pt]  {};

	\draw[ultra thick] (2)--(3);
	\draw[ultra thick] (4)--(3);
	\draw[ultra thick] (1)--(5);
	\draw[ultra thick] (5)--(19);
	\draw[ultra thick] (6)--(7);
	\draw[ultra thick] (41)--(13);
	\draw[ultra thick] (23)--(24);
	\draw[ultra thick] (25)--(26);
	\draw[ultra thick] (27)--(28);
	\draw[ultra thick] (8)--(5);
	\draw[ultra thick] (14)--(15);
	\draw[ultra thick] (10)--(11);
	\draw[ultra thick] (9)--(20);
	\draw[ultra thick] (12)--(11);
	\draw[ultra thick] (18)--(19);
	\draw[ultra thick] (16)--(17);
	\draw[ultra thick] (23)--(19);
	\draw[ultra thick] (25)--(29);
	\draw[ultra thick] (55)--(36);
	\draw[ultra thick] (53)--(38);
	\draw[ultra thick] (40)--(54);

%\draw[ultra thick] ([shift={(13,7)}]90:1.0) arc[radius=1.0, start angle=90, end angle= 180];
%\draw[ultra thick] ([shift={(12,5)}]90:1.0) arc[radius=1.0, start angle=180, end angle= 270];
\draw[ultra thick] ([shift={(15,6)}]90:1.0) arc[radius=1.0, start angle=0, end angle= 90];
\draw[ultra thick] ([shift={(14,4)}]90:1.0) arc[radius=1.0, start angle=270, end angle= 360];
\draw[ultra thick] ([shift={(3,3)}]90:1.0) arc[radius=1.0, start angle=0, end angle= 90];
\draw[ultra thick] ([shift={(-3,5)}]90:1.0) arc[radius=1.0, start angle=180, end angle= 270];
\draw[ultra thick] ([shift={(0,5)}]90:1.0) arc[radius=1.0, start angle=180, end angle= 270];
\draw[ultra thick] ([shift={(0,3)}]90:1.0) arc[radius=1.0, start angle=0, end angle= 90];
\draw[ultra thick] ([shift={(3,5)}]90:1.0) arc[radius=1.0, start angle=180, end angle= 270];

	\end{tikzpicture}
}
\end{center}
\subcaption{$w_j^{-1}(j) > 0$.}
\end{minipage}
\caption{\label{fig:path_model_j} The subdiagram $\mathscr{G}^{(j)}(D(w))$ of $\mathscr{G}(D(w))$.}
\end{figure}
\end{proof}

\begin{lem}\label{l:ladder_move_of_path_model}
Let $D \in \mathcal{SPD}_n$, and assume that the ladder move $L_{i,j}(D)$ is defined. 
Then the equality $w_{L_{i,j}(D)} = w_D$ holds. 
Similarly, if the inverse ladder move $L_{i,j}^{-1}(D)$ is defined, then it holds that $w_{L_{i,j}^{-1}(D)} = w_D$.
\end{lem}

\begin{proof}
We prove the assertion only for the ladder move $L_{i,j}(D)$; a proof of the assertion for $L_{i,j}^{-1}(D)$ is similar. 
If $j = n$, the diagram $\mathscr{G}(L_{i, j}(D))$ of the ladder move $L_{i, j}(D)$ is obtained from $\mathscr{G}(D)$ as in Figure \ref{fig:ladder_move_path_1}. 
This implies the assertion when $j = n$. 

Let $j \neq n$. 
We prove the assertion only when $j < n$; a proof of the assertion when $j > n$ is similar. 
In this case, the diagram $\mathscr{G}(L_{i, j}(D))$ is obtained from $\mathscr{G}(D)$ as in Figure \ref{fig:ladder_move_path_2}, where we write $L_{i, j}(D) = D \cup \{(p, q+1)\} \setminus \{(i, j)\}$. 
Hence we know the assertion when $j < n$. 
This completes the proof of the lemma. 
\begin{figure}[h]
\begin{center}
\scalebox{0.2}{
	\begin{tikzpicture}

  \fill (8,-2) coordinate (1) node[right=3pt]  {};
  \fill (8,-6) coordinate (2) node[right=3pt]  {};
  \fill (10,-4) coordinate (3) node[right=3pt]  {};
  \fill (9,-7) coordinate (4) node[right=3pt]  {};
  \fill (10,-7) coordinate (5) node[right=3pt]  {};
  \fill (11,-6) coordinate (6) node[right=3pt]  {};
  \fill (11,-5) coordinate (7) node[right=3pt]  {};
   \fill (6,5) coordinate (8) node[right=3pt]  {};
  \fill (8,10) coordinate (9) node[right=3pt]  {};
 \fill (6,-4) coordinate (10) node[right=3pt]  {};
 \fill (6,-7) coordinate (11) node[right=3pt]  {};
 \fill (7,-7) coordinate (12) node[right=3pt]  {};
 \fill (8,-8) coordinate (13) node[right=3pt]  {};
 \fill (8,-9) coordinate (14) node[right=3pt]  {};
   \fill (8,9) coordinate (15) node[above=3pt]  {};
   \fill (6,8) coordinate (16) node[right=3pt]  {};
   \fill (7,8) coordinate (17) node[right=3pt]  {};
   \fill (10,2) coordinate (18) node[right=3pt]  {};
   \fill (8,4) coordinate (19) node[right=3pt]  {};
   \fill (6,-1) coordinate (20) node[right=3pt]  {};
   \fill (8,7) coordinate (21) node[right=3pt]  {};
   \fill (8,6) coordinate (22) node[right=3pt]  {};
   \fill (9,8) coordinate (23) node[right=3pt]  {};
   \fill (10,8) coordinate (24) node[right=3pt]  {};
   \fill (9,5) coordinate (25) node[right=3pt]  {};
   \fill (10,5) coordinate (26) node[right=3pt]  {};
   \fill (11,7) coordinate (27) node[right=3pt]  {};
   \fill (11,6) coordinate (28) node[right=3pt]  {};
   \fill (7,5) coordinate (29) node[right=3pt]  {};
   \fill (6,2) coordinate (30) node[right=3pt]  {};
   \fill (10,-1) coordinate (31) node[right=3pt]  {};
   \fill (11,0) coordinate (32) node[right=3pt]  {};
   \fill (11,1) coordinate (33) node[right=3pt]  {};

	\draw[ultra thick] (21)--(22);
	\draw[ultra thick] (23)--(24);
	\draw[ultra thick] (25)--(26);
	\draw[ultra thick] (27)--(28);
	\draw[ultra thick] (29)--(8);

	\draw[ultra thick] (19)--(1);
	\draw[ultra thick] (18)--(30);
	\draw[ultra thick] (31)--(20);
	\draw[ultra thick] (32)--(33);

	\draw[ultra thick] (16)--(17);
	\draw[ultra thick] (10)--(3);
	\draw[ultra thick] (11)--(12);
	\draw[ultra thick] (1)--(2);
	\draw[ultra thick] (4)--(5);
	\draw[ultra thick] (7)--(6);
	\draw[ultra thick] (15)--(9);
	\draw[ultra thick] (13)--(14);

\draw[ultra thick] ([shift={(9,7)}]90:1.0) arc[radius=1.0, start angle=90, end angle= 180];
\draw[ultra thick] ([shift={(8,5)}]90:1.0) arc[radius=1.0, start angle=180, end angle= 270];
\draw[ultra thick] ([shift={(11,6)}]90:1.0) arc[radius=1.0, start angle=0, end angle= 90];
\draw[ultra thick] ([shift={(10,4)}]90:1.0) arc[radius=1.0, start angle=270, end angle= 360];
\draw[ultra thick] ([shift={(11,0)}]90:1.0) arc[radius=1.0, start angle=0, end angle= 90];
\draw[ultra thick] ([shift={(10,-2)}]90:1.0) arc[radius=1.0, start angle=270, end angle= 360];
\draw[ultra thick] ([shift={(8,3)}]90:1.0) arc[radius=1.0, start angle=0, end angle= 90];
\draw[ultra thick] ([shift={(7,7)}]90:1.0) arc[radius=1.0, start angle=270, end angle= 360];
\draw[ultra thick] ([shift={(8,-9)}]90:1.0) arc[radius=1.0, start angle=0, end angle= 90];
\draw[ultra thick] ([shift={(11,-6)}]90:1.0) arc[radius=1.0, start angle=0, end angle= 90];
\draw[ultra thick] ([shift={(8,-7)}]90:1.0) arc[radius=1.0, start angle=180, end angle= 270];
\draw[ultra thick] ([shift={(10,-8)}]90:1.0) arc[radius=1.0, start angle=270, end angle= 360];

\draw[ultra thick] [arrows = {|[scale=2]-Stealth[scale=2]}] (12.4,0) -- (14.4,0);
	\end{tikzpicture}
\hspace{10mm}
	\begin{tikzpicture}

  \fill (8,-2) coordinate (1) node[right=3pt]  {};
  \fill (8,-6) coordinate (2) node[right=3pt]  {};
  \fill (10,-4) coordinate (3) node[right=3pt]  {};
  \fill (9,-7) coordinate (4) node[right=3pt]  {};
  \fill (10,-7) coordinate (5) node[right=3pt]  {};
  \fill (11,-6) coordinate (6) node[right=3pt]  {};
  \fill (11,-5) coordinate (7) node[right=3pt]  {};
   \fill (6,5) coordinate (8) node[right=3pt]  {};
  \fill (8,10) coordinate (9) node[right=3pt]  {};
 \fill (6,-4) coordinate (10) node[right=3pt]  {};
 \fill (6,-7) coordinate (11) node[right=3pt]  {};
 \fill (7,-7) coordinate (12) node[right=3pt]  {};
 \fill (8,-8) coordinate (13) node[right=3pt]  {};
 \fill (8,-9) coordinate (14) node[right=3pt]  {};
   \fill (8,9) coordinate (15) node[above=3pt]  {};
   \fill (6,8) coordinate (16) node[right=3pt]  {};
   \fill (7,8) coordinate (17) node[right=3pt]  {};
   \fill (10,2) coordinate (18) node[right=3pt]  {};
   \fill (8,4) coordinate (19) node[right=3pt]  {};
   \fill (6,-1) coordinate (20) node[right=3pt]  {};
   \fill (8,7) coordinate (21) node[right=3pt]  {};
   \fill (8,6) coordinate (22) node[right=3pt]  {};
   \fill (9,8) coordinate (23) node[right=3pt]  {};
   \fill (10,8) coordinate (24) node[right=3pt]  {};
   \fill (9,5) coordinate (25) node[right=3pt]  {};
   \fill (10,5) coordinate (26) node[right=3pt]  {};
   \fill (11,7) coordinate (27) node[right=3pt]  {};
   \fill (11,6) coordinate (28) node[right=3pt]  {};
   \fill (7,5) coordinate (29) node[right=3pt]  {};
   \fill (6,2) coordinate (30) node[right=3pt]  {};
   \fill (10,-1) coordinate (31) node[right=3pt]  {};
   \fill (11,0) coordinate (32) node[right=3pt]  {};
   \fill (11,1) coordinate (33) node[right=3pt]  {};
  \fill (9,-4) coordinate (34) node[right=3pt]  {};
  \fill (7,-4) coordinate (35) node[right=3pt]  {};
  \fill (8,-3) coordinate (36) node[right=3pt]  {};
  \fill (8,-5) coordinate (37) node[right=3pt]  {};

	\draw[ultra thick] (21)--(22);
	\draw[ultra thick] (23)--(24);
	\draw[ultra thick] (25)--(26);
	\draw[ultra thick] (27)--(28);
	\draw[ultra thick] (29)--(8);
	\draw[ultra thick] (25)--(29);

	\draw[ultra thick] (19)--(22);
	\draw[ultra thick] (19)--(1);
	\draw[ultra thick] (18)--(30);
	\draw[ultra thick] (31)--(20);
	\draw[ultra thick] (32)--(33);

	\draw[ultra thick] (16)--(17);
	\draw[ultra thick] (34)--(3);
	\draw[ultra thick] (10)--(35);
	\draw[ultra thick] (11)--(12);
	\draw[ultra thick] (1)--(36);
	\draw[ultra thick] (37)--(2);
	\draw[ultra thick] (4)--(5);
	\draw[ultra thick] (7)--(6);
	\draw[ultra thick] (15)--(9);
	\draw[ultra thick] (13)--(14);

\draw[ultra thick] ([shift={(9,7)}]90:1.0) arc[radius=1.0, start angle=90, end angle= 180];
%\draw[ultra thick] ([shift={(8,5)}]90:1.0) arc[radius=1.0, start angle=180, end angle= 270];
\draw[ultra thick] ([shift={(11,6)}]90:1.0) arc[radius=1.0, start angle=0, end angle= 90];
\draw[ultra thick] ([shift={(10,4)}]90:1.0) arc[radius=1.0, start angle=270, end angle= 360];
\draw[ultra thick] ([shift={(11,0)}]90:1.0) arc[radius=1.0, start angle=0, end angle= 90];
\draw[ultra thick] ([shift={(10,-2)}]90:1.0) arc[radius=1.0, start angle=270, end angle= 360];
%\draw[ultra thick] ([shift={(8,3)}]90:1.0) arc[radius=1.0, start angle=0, end angle= 90];
\draw[ultra thick] ([shift={(7,7)}]90:1.0) arc[radius=1.0, start angle=270, end angle= 360];
\draw[ultra thick] ([shift={(8,-9)}]90:1.0) arc[radius=1.0, start angle=0, end angle= 90];
\draw[ultra thick] ([shift={(11,-6)}]90:1.0) arc[radius=1.0, start angle=0, end angle= 90];
\draw[ultra thick] ([shift={(8,-7)}]90:1.0) arc[radius=1.0, start angle=180, end angle= 270];
\draw[ultra thick] ([shift={(10,-8)}]90:1.0) arc[radius=1.0, start angle=270, end angle= 360];
\draw[ultra thick] ([shift={(9,-5)}]90:1.0) arc[radius=1.0, start angle=90, end angle= 180];
\draw[ultra thick] ([shift={(7,-5)}]90:1.0) arc[radius=1.0, start angle=270, end angle= 360];

	\end{tikzpicture}
}
\end{center}
\caption{\label{fig:ladder_move_path_1} The ladder move $D \mapsto L_{i,j}(D)$ when $j = n$.}
\end{figure}
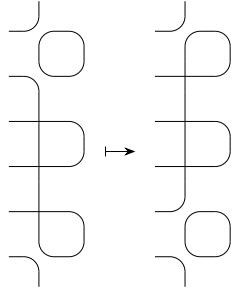
\begin{figure}[h]
\centering
\begin{minipage}[b]{0.4\linewidth}
\begin{center}
\scalebox{0.25}{
	\begin{tikzpicture}

   \fill (8,-2) coordinate (1) node[right=3pt]  {};
  \fill (8,-6) coordinate (2) node[right=3pt]  {};
  \fill (10,-4) coordinate (3) node[right=3pt]  {};
  \fill (13,-4) coordinate (4) node[right=3pt]  {};
  \fill (12,-4) coordinate (5) node[right=3pt]  {};
  \fill (11,-6) coordinate (6) node[right=3pt]  {};
  \fill (11,-5) coordinate (7) node[right=3pt]  {};
   \fill (11,3) coordinate (8) node[right=3pt]  {};
  \fill (11,1) coordinate (9) node[right=3pt]  {};
 \fill (6,-4) coordinate (10) node[right=3pt]  {};
 \fill (11,-3) coordinate (11) node[right=3pt]  {};
 \fill (11,4) coordinate (12) node[right=3pt]  {};
 \fill (9,-4) coordinate (13) node[right=3pt]  {};
 \fill (7,-4) coordinate (14) node[right=3pt]  {};
   \fill (8,-3) coordinate (15) node[above=3pt]  {};
  \fill (8,-5) coordinate (16) node[right=3pt]  {};
 %  \fill (7,8) coordinate (17) node[right=3pt]  {};
   \fill (13,2) coordinate (18) node[right=3pt]  {};
   \fill (8,4) coordinate (19) node[right=3pt]  {};
   \fill (6,-1) coordinate (20) node[right=3pt]  {};
   \fill (8,3) coordinate (21) node[right=3pt]  {};
   \fill (8,1) coordinate (22) node[right=3pt]  {};
   \fill (7,2) coordinate (23) node[right=3pt]  {};
   \fill (9,2) coordinate (24) node[right=3pt]  {};
   \fill (12,2) coordinate (25) node[right=3pt]  {};
   \fill (10,2) coordinate (26) node[right=3pt]  {};
   \fill (6,2) coordinate (30) node[right=3pt]  {};
   \fill (13,-1) coordinate (31) node[right=3pt]  {};

	\draw[ultra thick] (19)--(21);
	\draw[ultra thick] (22)--(1);
	\draw[ultra thick] (23)--(30);
	\draw[ultra thick] (18)--(25);
	\draw[ultra thick] (26)--(24);
	\draw[ultra thick] (31)--(20);

%	\draw[ultra thick] (19)--(22);
	\draw[ultra thick] (13)--(3);
	\draw[ultra thick] (10)--(13);
	\draw[ultra thick] (8)--(12);
	\draw[ultra thick] (11)--(9);
	\draw[ultra thick] (1)--(15);
	\draw[ultra thick] (15)--(2);
	\draw[ultra thick] (4)--(5);
	\draw[ultra thick] (7)--(6);

\draw[ultra thick] ([shift={(9,1)}]90:1.0) arc[radius=1.0, start angle=90, end angle= 180];
\draw[ultra thick] ([shift={(7,1)}]90:1.0) arc[radius=1.0, start angle=270, end angle= 360];
\draw[ultra thick] ([shift={(12,1)}]90:1.0) arc[radius=1.0, start angle=90, end angle= 180];
\draw[ultra thick] ([shift={(10,1)}]90:1.0) arc[radius=1.0, start angle=270, end angle= 360];
\draw[ultra thick] ([shift={(10,-5)}]90:1.0) arc[radius=1.0, start angle=270, end angle= 360];
\draw[ultra thick] ([shift={(12,-5)}]90:1.0) arc[radius=1.0, start angle=90, end angle= 180];

\draw[ultra thick] [arrows = {|[scale=2]-Stealth[scale=2]}] (14.4,-1) -- (16.4,-1);
	\end{tikzpicture}
\hspace{10mm}
	\begin{tikzpicture}

  \fill (8,-2) coordinate (1) node[right=3pt]  {};
  \fill (8,-6) coordinate (2) node[right=3pt]  {};
  \fill (10,-4) coordinate (3) node[right=3pt]  {};
  \fill (13,-4) coordinate (4) node[right=3pt]  {};
  \fill (12,-4) coordinate (5) node[right=3pt]  {};
  \fill (11,-6) coordinate (6) node[right=3pt]  {};
  \fill (11,-5) coordinate (7) node[right=3pt]  {};
%   \fill (6,5) coordinate (8) node[right=3pt]  {};
%  \fill (8,10) coordinate (9) node[right=3pt]  {};
 \fill (6,-4) coordinate (10) node[right=3pt]  {};
 \fill (11,-3) coordinate (11) node[right=3pt]  {};
 \fill (11,4) coordinate (12) node[right=3pt]  {};
 \fill (9,-4) coordinate (13) node[right=3pt]  {};
 \fill (7,-4) coordinate (14) node[right=3pt]  {};
   \fill (8,-3) coordinate (15) node[above=3pt]  {};
  \fill (8,-5) coordinate (16) node[right=3pt]  {};
 %  \fill (7,8) coordinate (17) node[right=3pt]  {};
   \fill (13,2) coordinate (18) node[right=3pt]  {};
   \fill (8,4) coordinate (19) node[right=3pt]  {};
   \fill (6,-1) coordinate (20) node[right=3pt]  {};
   \fill (8,3) coordinate (21) node[right=3pt]  {};
   \fill (8,1) coordinate (22) node[right=3pt]  {};
   \fill (7,2) coordinate (23) node[right=3pt]  {};
   \fill (9,2) coordinate (24) node[right=3pt]  {};
   \fill (6,2) coordinate (30) node[right=3pt]  {};
   \fill (13,-1) coordinate (31) node[right=3pt]  {};

	\draw[ultra thick] (19)--(21);
	\draw[ultra thick] (22)--(1);
	\draw[ultra thick] (23)--(30);
	\draw[ultra thick] (18)--(24);
	\draw[ultra thick] (31)--(20);
	\draw[ultra thick] (13)--(3);
	\draw[ultra thick] (10)--(14);
	\draw[ultra thick] (11)--(12);
	\draw[ultra thick] (1)--(15);
	\draw[ultra thick] (16)--(2);
	\draw[ultra thick] (4)--(5);
	\draw[ultra thick] (7)--(6);

\draw[ultra thick] ([shift={(9,1)}]90:1.0) arc[radius=1.0, start angle=90, end angle= 180];
\draw[ultra thick] ([shift={(7,1)}]90:1.0) arc[radius=1.0, start angle=270, end angle= 360];
\draw[ultra thick] ([shift={(10,-5)}]90:1.0) arc[radius=1.0, start angle=270, end angle= 360];
\draw[ultra thick] ([shift={(12,-5)}]90:1.0) arc[radius=1.0, start angle=90, end angle= 180];
\draw[ultra thick] ([shift={(7,-5)}]90:1.0) arc[radius=1.0, start angle=270, end angle= 360];
\draw[ultra thick] ([shift={(9,-5)}]90:1.0) arc[radius=1.0, start angle=90, end angle= 180];

	\end{tikzpicture}
}
\end{center}
\subcaption{The case $q < n$.}
\end{minipage}
\begin{minipage}[b]{0.4\linewidth}
\begin{center}
\scalebox{0.25}{
	\begin{tikzpicture}

  \fill (8,-2) coordinate (1) node[right=3pt]  {};
  \fill (8,-6) coordinate (2) node[right=3pt]  {};
  \fill (10,-4) coordinate (3) node[right=3pt]  {};
  \fill (13,-4) coordinate (4) node[right=3pt]  {};
  \fill (12,-4) coordinate (5) node[right=3pt]  {};
  \fill (11,-6) coordinate (6) node[right=3pt]  {};
  \fill (11,-5) coordinate (7) node[right=3pt]  {};
   \fill (6,5) coordinate (8) node[right=3pt]  {};
 \fill (6,-4) coordinate (10) node[right=3pt]  {};
 \fill (11,-3) coordinate (11) node[right=3pt]  {};
 \fill (11,4) coordinate (12) node[right=3pt]  {};
   \fill (13,2) coordinate (18) node[right=3pt]  {};
   \fill (8,4) coordinate (19) node[right=3pt]  {};
   \fill (6,-1) coordinate (20) node[right=3pt]  {};
   \fill (8,7) coordinate (21) node[right=3pt]  {};
   \fill (8,6) coordinate (22) node[right=3pt]  {};
   \fill (12,5) coordinate (23) node[right=3pt]  {};
   \fill (13,5) coordinate (24) node[right=3pt]  {};
   \fill (9,5) coordinate (25) node[right=3pt]  {};
   \fill (10,5) coordinate (26) node[right=3pt]  {};
   \fill (11,7) coordinate (27) node[right=3pt]  {};
   \fill (11,6) coordinate (28) node[right=3pt]  {};
   \fill (7,5) coordinate (29) node[right=3pt]  {};
   \fill (6,2) coordinate (30) node[right=3pt]  {};
   \fill (13,-1) coordinate (31) node[right=3pt]  {};

	\draw[ultra thick] (21)--(22);
	\draw[ultra thick] (23)--(24);
	\draw[ultra thick] (25)--(26);
	\draw[ultra thick] (27)--(28);
	\draw[ultra thick] (29)--(8);
	\draw[ultra thick] (19)--(1);
	\draw[ultra thick] (18)--(30);
	\draw[ultra thick] (31)--(20);
	\draw[ultra thick] (10)--(3);
	\draw[ultra thick] (11)--(12);
	\draw[ultra thick] (1)--(2);
	\draw[ultra thick] (4)--(5);
	\draw[ultra thick] (7)--(6);

\draw[ultra thick] ([shift={(8,5)}]90:1.0) arc[radius=1.0, start angle=180, end angle= 270];
\draw[ultra thick] ([shift={(11,5)}]90:1.0) arc[radius=1.0, start angle=180, end angle= 270];
\draw[ultra thick] ([shift={(11,3)}]90:1.0) arc[radius=1.0, start angle=0, end angle= 90];
\draw[ultra thick] ([shift={(8,3)}]90:1.0) arc[radius=1.0, start angle=0, end angle= 90];
\draw[ultra thick] ([shift={(10,-5)}]90:1.0) arc[radius=1.0, start angle=270, end angle= 360];
\draw[ultra thick] ([shift={(12,-5)}]90:1.0) arc[radius=1.0, start angle=90, end angle= 180];

\draw[ultra thick] [arrows = {|[scale=2]-Stealth[scale=2]}] (14.4,0) -- (16.4,0);
	\end{tikzpicture}
\hspace{10mm}
	\begin{tikzpicture}

  \fill (8,-2) coordinate (1) node[right=3pt]  {};
  \fill (8,-6) coordinate (2) node[right=3pt]  {};
  \fill (10,-4) coordinate (3) node[right=3pt]  {};
  \fill (13,-4) coordinate (4) node[right=3pt]  {};
  \fill (12,-4) coordinate (5) node[right=3pt]  {};
  \fill (11,-6) coordinate (6) node[right=3pt]  {};
  \fill (11,-5) coordinate (7) node[right=3pt]  {};
   \fill (6,5) coordinate (8) node[right=3pt]  {};
%  \fill (8,10) coordinate (9) node[right=3pt]  {};
 \fill (6,-4) coordinate (10) node[right=3pt]  {};
 \fill (11,-3) coordinate (11) node[right=3pt]  {};
 \fill (11,4) coordinate (12) node[right=3pt]  {};
 \fill (9,-4) coordinate (13) node[right=3pt]  {};
 \fill (7,-4) coordinate (14) node[right=3pt]  {};
   \fill (8,-3) coordinate (15) node[above=3pt]  {};
  \fill (8,-5) coordinate (16) node[right=3pt]  {};
 %  \fill (7,8) coordinate (17) node[right=3pt]  {};
   \fill (13,2) coordinate (18) node[right=3pt]  {};
   \fill (8,4) coordinate (19) node[right=3pt]  {};
   \fill (6,-1) coordinate (20) node[right=3pt]  {};
   \fill (8,7) coordinate (21) node[right=3pt]  {};
   \fill (8,6) coordinate (22) node[right=3pt]  {};
   \fill (12,5) coordinate (23) node[right=3pt]  {};
   \fill (13,5) coordinate (24) node[right=3pt]  {};
   \fill (9,5) coordinate (25) node[right=3pt]  {};
   \fill (10,5) coordinate (26) node[right=3pt]  {};
   \fill (11,7) coordinate (27) node[right=3pt]  {};
   \fill (11,6) coordinate (28) node[right=3pt]  {};
   \fill (7,5) coordinate (29) node[right=3pt]  {};
   \fill (6,2) coordinate (30) node[right=3pt]  {};
   \fill (13,-1) coordinate (31) node[right=3pt]  {};

	\draw[ultra thick] (21)--(22);
	\draw[ultra thick] (23)--(24);
	\draw[ultra thick] (25)--(26);
	\draw[ultra thick] (27)--(28);
	\draw[ultra thick] (29)--(8);

	\draw[ultra thick] (19)--(1);
	\draw[ultra thick] (18)--(30);
	\draw[ultra thick] (31)--(20);

	\draw[ultra thick] (19)--(22);
	\draw[ultra thick] (13)--(3);
	\draw[ultra thick] (10)--(14);
	\draw[ultra thick] (11)--(12);
	\draw[ultra thick] (1)--(15);
	\draw[ultra thick] (16)--(2);
	\draw[ultra thick] (4)--(5);
	\draw[ultra thick] (7)--(6);
	\draw[ultra thick] (25)--(29);

\draw[ultra thick] ([shift={(11,5)}]90:1.0) arc[radius=1.0, start angle=180, end angle= 270];
\draw[ultra thick] ([shift={(11,3)}]90:1.0) arc[radius=1.0, start angle=0, end angle= 90];
\draw[ultra thick] ([shift={(10,-5)}]90:1.0) arc[radius=1.0, start angle=270, end angle= 360];
\draw[ultra thick] ([shift={(12,-5)}]90:1.0) arc[radius=1.0, start angle=90, end angle= 180];
\draw[ultra thick] ([shift={(7,-5)}]90:1.0) arc[radius=1.0, start angle=270, end angle= 360];
\draw[ultra thick] ([shift={(9,-5)}]90:1.0) arc[radius=1.0, start angle=90, end angle= 180];

	\end{tikzpicture}
}
\end{center}
\subcaption{The case $q \geq n$.}
\end{minipage}
\caption{\label{fig:ladder_move_path_2} The ladder move $D \mapsto L_{i,j}(D)$ when $j < n$.}
\end{figure}
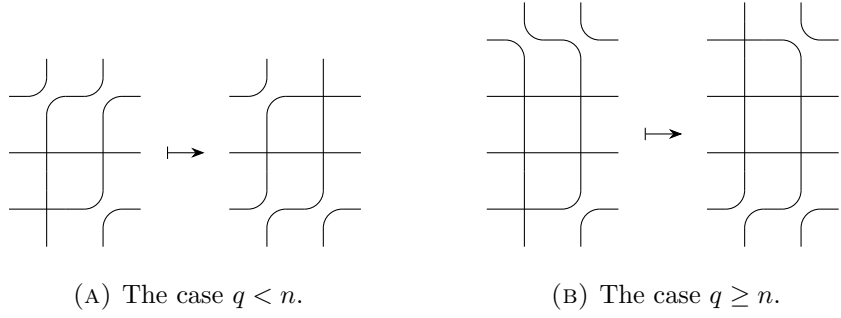
\end{proof}

Let $RSP(w)$ denote the set of reduced skew pipe dreams $D$ such that $w_D = w$. 
Since ladder moves and inverse ladder moves do not change the cardinality of skew pipe dreams, we obtain the following by \cref{l:ladder_move_of_path_model}. 

\begin{lem}\label{l:stable_ladder_RSP}
For $w \in W$, the set $RSP(w)$ is stable under ladder moves and inverse ladder moves.
\end{lem}

Now we show \cref{t:main_result_1}. 

\begin{proof}[{Proof of \cref{t:main_result_1}}]
We see by Lemmas \ref{l:reduced_D(w)} and \ref{l:stable_ladder_RSP} that $\mathscr{L}(D(w)) \subseteq RSP(w)$. 
Hence it suffices to show that $RSP(w) \subseteq \mathscr{L}(D(w))$. 
Take $D \in RSP(w)$. 
If there is no $(i,j) \in D$ such that $(i,j-1) \in SY_n \setminus D$, then the argument in the proof of \cref{l:reduced_D(w)} implies that $D = D(w)$.
Hence we may assume that there exists $(i,j) \in D$ such that $(i,j-1) \in SY_n \setminus D$. 
Take $1 \leq k \leq N$ to be the minimum positive integer such that $(p_k,q_k) \in SY_n \setminus D$ and such that $(p_k,q_k+1) \in D$.
Let us prove that the inverse ladder move $L_{(p_k,q_k+1)}^{-1}(D)$ is defined. 
We take two cases: $q_k = n$ or $q_k \neq n$.

If $q_k = n$, then set 
\[t \coloneqq \min\{p_k < t^\prime \leq n \mid (t^\prime, n+1) \notin D\}.\] 
By the definitions of $k$ and $t$, it follows that $(t^\prime, n), (t^\prime, n+1) \in D$ for all $p_k < t^\prime < t$. 
If $(t, n) \in D$, then $w_D = w_{D^\prime}$ for $D^\prime \coloneqq D \setminus \{(p_k, n+1), (t, n)\}$ (see Figure \ref{fig:non_reduced_1}), which gives a contradiction since $D$ is reduced. 
Hence we have $(t, n) \notin D$, which implies that the inverse ladder move $L_{(p_k,q_k+1)}^{-1}(D)$ is defined. 

Let $q_k \neq n$. 
We prove the assertion only when $q_k < n$; a proof of the assertion when $q_k > n$ is similar. 
Set 
\[k^\prime \coloneqq \max\{1 \leq \ell < k \mid q_\ell \in \{q_k, 2n-q_k\},\ (p_\ell, q_\ell+1) \notin D\}.\] 
By the definitions of $k$ and $k^\prime$, it follows that $(p_\ell, q_\ell), (p_\ell, q_\ell+1) \in D$ for all $k^\prime < \ell < k$ such that $q_\ell \in \{q_k, 2n-q_k\}$. 
If $(p_{k^\prime}, q_{k^\prime}) \in D$, then $w_D = w_{D^\prime}$ for $D^\prime \coloneqq D \setminus \{(p_k,q_k+1), (p_{k^\prime}, q_{k^\prime})\}$ (see Figure \ref{fig:non_reduced_2}), which gives a contradiction since $D$ is reduced. 
Hence we have $(p_{k^\prime}, q_{k^\prime}) \notin D$, which implies that the inverse ladder move $L_{(p_k,q_k+1)}^{-1}(D)$ is defined. 

Continue this argument by replacing $D$ with $L_{(p_k,q_k+1)}^{-1}(D)$.
Then, by a sequence of inverse ladder moves, we can change $D$ to $\widetilde{D}$ such that there is no $(i,j) \in \widetilde{D}$ with $(i,j-1) \in SY_n \setminus \widetilde{D}$. 
Since inverse ladder moves preserve $RSP(w)$ by \cref{l:stable_ladder_RSP}, we have $\widetilde{D} \in RSP(w)$, which implies that $\widetilde{D} = D(w)$. 
Hence $D$ is obtained from $D(w)$ by a sequence of ladder moves, that is, $D \in \mathscr{L}(D(w))$.
This proves the theorem.
\begin{figure}[h]
\begin{center}
\scalebox{0.2}{
	\begin{tikzpicture}

  \fill (8,-2) coordinate (1) node[right=3pt]  {};
  \fill (8,-6) coordinate (2) node[right=3pt]  {};
  \fill (10,-4) coordinate (3) node[right=3pt]  {};
  \fill (9,-7) coordinate (4) node[right=3pt]  {};
  \fill (10,-7) coordinate (5) node[right=3pt]  {};
  \fill (11,-6) coordinate (6) node[right=3pt]  {};
  \fill (11,-5) coordinate (7) node[right=3pt]  {};
   \fill (6,5) coordinate (8) node[right=3pt]  {};
  \fill (8,10) coordinate (9) node[right=3pt]  {};
 \fill (6,-4) coordinate (10) node[right=3pt]  {};
 \fill (6,-7) coordinate (11) node[right=3pt]  {};
 \fill (7,-7) coordinate (12) node[right=3pt]  {};
 \fill (8,-8) coordinate (13) node[right=3pt]  {};
 \fill (8,-9) coordinate (14) node[right=3pt]  {};
   \fill (8,9) coordinate (15) node[above=3pt]  {};
   \fill (6,8) coordinate (16) node[right=3pt]  {};
   \fill (7,8) coordinate (17) node[right=3pt]  {};
   \fill (10,2) coordinate (18) node[right=3pt]  {};
   \fill (8,4) coordinate (19) node[right=3pt]  {};
   \fill (6,-1) coordinate (20) node[right=3pt]  {};
   \fill (8,7) coordinate (21) node[right=3pt]  {};
   \fill (8,6) coordinate (22) node[right=3pt]  {};
   \fill (9,8) coordinate (23) node[right=3pt]  {};
   \fill (10,8) coordinate (24) node[right=3pt]  {};
   \fill (9,5) coordinate (25) node[right=3pt]  {};
   \fill (10,5) coordinate (26) node[right=3pt]  {};
   \fill (11,7) coordinate (27) node[right=3pt]  {};
   \fill (11,6) coordinate (28) node[right=3pt]  {};
   \fill (7,5) coordinate (29) node[right=3pt]  {};
   \fill (6,2) coordinate (30) node[right=3pt]  {};
   \fill (10,-1) coordinate (31) node[right=3pt]  {};
   \fill (11,0) coordinate (32) node[right=3pt]  {};
   \fill (11,1) coordinate (33) node[right=3pt]  {};

	\draw[ultra thick] (21)--(19);
	\draw[ultra thick] (23)--(24);
	\draw[ultra thick] (25)--(26);
	\draw[ultra thick] (27)--(28);
	\draw[ultra thick] (25)--(8);

	\draw[ultra thick] (19)--(1);
	\draw[ultra thick] (18)--(30);
	\draw[ultra thick] (31)--(20);
	\draw[ultra thick] (32)--(33);

	\draw[ultra thick] (16)--(17);
	\draw[ultra thick] (10)--(3);
	\draw[ultra thick] (11)--(12);
	\draw[ultra thick] (1)--(2);
	\draw[ultra thick] (4)--(5);
	\draw[ultra thick] (7)--(6);
	\draw[ultra thick] (15)--(9);
	\draw[ultra thick] (13)--(14);

\draw[ultra thick] ([shift={(9,7)}]90:1.0) arc[radius=1.0, start angle=90, end angle= 180];
%\draw[ultra thick] ([shift={(8,5)}]90:1.0) arc[radius=1.0, start angle=180, end angle= 270];
\draw[ultra thick] ([shift={(11,6)}]90:1.0) arc[radius=1.0, start angle=0, end angle= 90];
\draw[ultra thick] ([shift={(10,4)}]90:1.0) arc[radius=1.0, start angle=270, end angle= 360];
\draw[ultra thick] ([shift={(11,0)}]90:1.0) arc[radius=1.0, start angle=0, end angle= 90];
\draw[ultra thick] ([shift={(10,-2)}]90:1.0) arc[radius=1.0, start angle=270, end angle= 360];
%\draw[ultra thick] ([shift={(8,3)}]90:1.0) arc[radius=1.0, start angle=0, end angle= 90];
\draw[ultra thick] ([shift={(7,7)}]90:1.0) arc[radius=1.0, start angle=270, end angle= 360];
\draw[ultra thick] ([shift={(8,-9)}]90:1.0) arc[radius=1.0, start angle=0, end angle= 90];
\draw[ultra thick] ([shift={(11,-6)}]90:1.0) arc[radius=1.0, start angle=0, end angle= 90];
\draw[ultra thick] ([shift={(8,-7)}]90:1.0) arc[radius=1.0, start angle=180, end angle= 270];
\draw[ultra thick] ([shift={(10,-8)}]90:1.0) arc[radius=1.0, start angle=270, end angle= 360];

\draw[ultra thick] [arrows = {|[scale=2]-Stealth[scale=2]}] (12.4,0) -- (14.4,0);
	\end{tikzpicture}
\hspace{10mm}
	\begin{tikzpicture}

  \fill (8,-2) coordinate (1) node[right=3pt]  {};
  \fill (8,-6) coordinate (2) node[right=3pt]  {};
  \fill (10,-4) coordinate (3) node[right=3pt]  {};
  \fill (9,-7) coordinate (4) node[right=3pt]  {};
  \fill (10,-7) coordinate (5) node[right=3pt]  {};
  \fill (11,-6) coordinate (6) node[right=3pt]  {};
  \fill (11,-5) coordinate (7) node[right=3pt]  {};
   \fill (6,5) coordinate (8) node[right=3pt]  {};
  \fill (8,10) coordinate (9) node[right=3pt]  {};
 \fill (6,-4) coordinate (10) node[right=3pt]  {};
 \fill (6,-7) coordinate (11) node[right=3pt]  {};
 \fill (7,-7) coordinate (12) node[right=3pt]  {};
 \fill (8,-8) coordinate (13) node[right=3pt]  {};
 \fill (8,-9) coordinate (14) node[right=3pt]  {};
   \fill (8,9) coordinate (15) node[above=3pt]  {};
   \fill (6,8) coordinate (16) node[right=3pt]  {};
   \fill (7,8) coordinate (17) node[right=3pt]  {};
   \fill (10,2) coordinate (18) node[right=3pt]  {};
   \fill (8,4) coordinate (19) node[right=3pt]  {};
   \fill (6,-1) coordinate (20) node[right=3pt]  {};
   \fill (8,7) coordinate (21) node[right=3pt]  {};
   \fill (8,6) coordinate (22) node[right=3pt]  {};
   \fill (9,8) coordinate (23) node[right=3pt]  {};
   \fill (10,8) coordinate (24) node[right=3pt]  {};
   \fill (9,5) coordinate (25) node[right=3pt]  {};
   \fill (10,5) coordinate (26) node[right=3pt]  {};
   \fill (11,7) coordinate (27) node[right=3pt]  {};
   \fill (11,6) coordinate (28) node[right=3pt]  {};
   \fill (7,5) coordinate (29) node[right=3pt]  {};
   \fill (6,2) coordinate (30) node[right=3pt]  {};
   \fill (10,-1) coordinate (31) node[right=3pt]  {};
   \fill (11,0) coordinate (32) node[right=3pt]  {};
   \fill (11,1) coordinate (33) node[right=3pt]  {};
  \fill (9,-4) coordinate (34) node[right=3pt]  {};
  \fill (7,-4) coordinate (35) node[right=3pt]  {};
  \fill (8,-3) coordinate (36) node[right=3pt]  {};
  \fill (8,-5) coordinate (37) node[right=3pt]  {};

	\draw[ultra thick] (21)--(22);
	\draw[ultra thick] (23)--(24);
	\draw[ultra thick] (25)--(26);
	\draw[ultra thick] (27)--(28);
	\draw[ultra thick] (29)--(8);
%	\draw[ultra thick] (25)--(29);

%	\draw[ultra thick] (19)--(22);
	\draw[ultra thick] (19)--(1);
	\draw[ultra thick] (18)--(30);
	\draw[ultra thick] (31)--(20);
	\draw[ultra thick] (32)--(33);

	\draw[ultra thick] (16)--(17);
	\draw[ultra thick] (34)--(3);
	\draw[ultra thick] (10)--(35);
	\draw[ultra thick] (11)--(12);
	\draw[ultra thick] (1)--(36);
	\draw[ultra thick] (37)--(2);
	\draw[ultra thick] (4)--(5);
	\draw[ultra thick] (7)--(6);
	\draw[ultra thick] (15)--(9);
	\draw[ultra thick] (13)--(14);

\draw[ultra thick] ([shift={(9,7)}]90:1.0) arc[radius=1.0, start angle=90, end angle= 180];
\draw[ultra thick] ([shift={(8,5)}]90:1.0) arc[radius=1.0, start angle=180, end angle= 270];
\draw[ultra thick] ([shift={(11,6)}]90:1.0) arc[radius=1.0, start angle=0, end angle= 90];
\draw[ultra thick] ([shift={(10,4)}]90:1.0) arc[radius=1.0, start angle=270, end angle= 360];
\draw[ultra thick] ([shift={(11,0)}]90:1.0) arc[radius=1.0, start angle=0, end angle= 90];
\draw[ultra thick] ([shift={(10,-2)}]90:1.0) arc[radius=1.0, start angle=270, end angle= 360];
\draw[ultra thick] ([shift={(8,3)}]90:1.0) arc[radius=1.0, start angle=0, end angle= 90];
\draw[ultra thick] ([shift={(7,7)}]90:1.0) arc[radius=1.0, start angle=270, end angle= 360];
\draw[ultra thick] ([shift={(8,-9)}]90:1.0) arc[radius=1.0, start angle=0, end angle= 90];
\draw[ultra thick] ([shift={(11,-6)}]90:1.0) arc[radius=1.0, start angle=0, end angle= 90];
\draw[ultra thick] ([shift={(8,-7)}]90:1.0) arc[radius=1.0, start angle=180, end angle= 270];
\draw[ultra thick] ([shift={(10,-8)}]90:1.0) arc[radius=1.0, start angle=270, end angle= 360];
\draw[ultra thick] ([shift={(9,-5)}]90:1.0) arc[radius=1.0, start angle=90, end angle= 180];
\draw[ultra thick] ([shift={(7,-5)}]90:1.0) arc[radius=1.0, start angle=270, end angle= 360];

	\end{tikzpicture}
}
\end{center}
\caption{\label{fig:non_reduced_1} The change $\mathscr{G}(D) \mapsto \mathscr{G}(D^\prime)$ when $q_k = n$.}
\end{figure}
\begin{figure}[h]
\centering
\begin{minipage}[b]{0.4\linewidth}
\begin{center}
\scalebox{0.25}{
	\begin{tikzpicture}

   \fill (8,-2) coordinate (1) node[right=3pt]  {};
  \fill (8,-6) coordinate (2) node[right=3pt]  {};
  \fill (10,-4) coordinate (3) node[right=3pt]  {};
  \fill (13,-4) coordinate (4) node[right=3pt]  {};
  \fill (12,-4) coordinate (5) node[right=3pt]  {};
  \fill (11,-6) coordinate (6) node[right=3pt]  {};
  \fill (11,-5) coordinate (7) node[right=3pt]  {};
   \fill (11,3) coordinate (8) node[right=3pt]  {};
  \fill (11,1) coordinate (9) node[right=3pt]  {};
 \fill (6,-4) coordinate (10) node[right=3pt]  {};
 \fill (11,-3) coordinate (11) node[right=3pt]  {};
 \fill (11,4) coordinate (12) node[right=3pt]  {};
 \fill (9,-4) coordinate (13) node[right=3pt]  {};
 \fill (7,-4) coordinate (14) node[right=3pt]  {};
   \fill (8,-3) coordinate (15) node[above=3pt]  {};
  \fill (8,-5) coordinate (16) node[right=3pt]  {};
 %  \fill (7,8) coordinate (17) node[right=3pt]  {};
   \fill (13,2) coordinate (18) node[right=3pt]  {};
   \fill (8,4) coordinate (19) node[right=3pt]  {};
   \fill (6,-1) coordinate (20) node[right=3pt]  {};
   \fill (8,3) coordinate (21) node[right=3pt]  {};
   \fill (8,1) coordinate (22) node[right=3pt]  {};
   \fill (7,2) coordinate (23) node[right=3pt]  {};
   \fill (9,2) coordinate (24) node[right=3pt]  {};
   \fill (12,2) coordinate (25) node[right=3pt]  {};
   \fill (10,2) coordinate (26) node[right=3pt]  {};
%   \fill (11,7) coordinate (27) node[right=3pt]  {};
 %  \fill (11,6) coordinate (28) node[right=3pt]  {};
%   \fill (7,5) coordinate (29) node[right=3pt]  {};
   \fill (6,2) coordinate (30) node[right=3pt]  {};
   \fill (13,-1) coordinate (31) node[right=3pt]  {};

	\draw[ultra thick] (19)--(21);
	\draw[ultra thick] (22)--(1);
	\draw[ultra thick] (23)--(30);
	\draw[ultra thick] (18)--(26);
	\draw[ultra thick] (26)--(24);
	\draw[ultra thick] (31)--(20);

%	\draw[ultra thick] (19)--(22);
	\draw[ultra thick] (13)--(3);
	\draw[ultra thick] (10)--(13);
	\draw[ultra thick] (9)--(12);
	\draw[ultra thick] (11)--(9);
	\draw[ultra thick] (1)--(15);
	\draw[ultra thick] (15)--(2);
	\draw[ultra thick] (4)--(5);
	\draw[ultra thick] (7)--(6);

\draw[ultra thick] ([shift={(9,1)}]90:1.0) arc[radius=1.0, start angle=90, end angle= 180];
\draw[ultra thick] ([shift={(7,1)}]90:1.0) arc[radius=1.0, start angle=270, end angle= 360];
\draw[ultra thick] ([shift={(10,-5)}]90:1.0) arc[radius=1.0, start angle=270, end angle= 360];
\draw[ultra thick] ([shift={(12,-5)}]90:1.0) arc[radius=1.0, start angle=90, end angle= 180];
%\draw[ultra thick] ([shift={(7,-5)}]90:1.0) arc[radius=1.0, start angle=270, end angle= 360];
%\draw[ultra thick] ([shift={(9,-5)}]90:1.0) arc[radius=1.0, start angle=90, end angle= 180];

\draw[ultra thick] [arrows = {|[scale=2]-Stealth[scale=2]}] (14.4,-1) -- (16.4,-1);
	\end{tikzpicture}
\hspace{10mm}
	\begin{tikzpicture}

  \fill (8,-2) coordinate (1) node[right=3pt]  {};
  \fill (8,-6) coordinate (2) node[right=3pt]  {};
  \fill (10,-4) coordinate (3) node[right=3pt]  {};
  \fill (13,-4) coordinate (4) node[right=3pt]  {};
  \fill (12,-4) coordinate (5) node[right=3pt]  {};
  \fill (11,-6) coordinate (6) node[right=3pt]  {};
  \fill (11,-5) coordinate (7) node[right=3pt]  {};
%   \fill (6,5) coordinate (8) node[right=3pt]  {};
%  \fill (8,10) coordinate (9) node[right=3pt]  {};
 \fill (6,-4) coordinate (10) node[right=3pt]  {};
 \fill (11,-3) coordinate (11) node[right=3pt]  {};
 \fill (11,4) coordinate (12) node[right=3pt]  {};
 \fill (9,-4) coordinate (13) node[right=3pt]  {};
 \fill (7,-4) coordinate (14) node[right=3pt]  {};
   \fill (8,-3) coordinate (15) node[above=3pt]  {};
  \fill (8,-5) coordinate (16) node[right=3pt]  {};
 %  \fill (7,8) coordinate (17) node[right=3pt]  {};
   \fill (13,2) coordinate (18) node[right=3pt]  {};
   \fill (8,4) coordinate (19) node[right=3pt]  {};
   \fill (6,-1) coordinate (20) node[right=3pt]  {};
   \fill (8,3) coordinate (21) node[right=3pt]  {};
   \fill (8,1) coordinate (22) node[right=3pt]  {};
   \fill (7,2) coordinate (23) node[right=3pt]  {};
   \fill (9,2) coordinate (24) node[right=3pt]  {};
   \fill (12,2) coordinate (25) node[right=3pt]  {};
   \fill (10,2) coordinate (26) node[right=3pt]  {};
   \fill (11,3) coordinate (27) node[right=3pt]  {};
  \fill (11,1) coordinate (28) node[right=3pt]  {};
%   \fill (7,5) coordinate (29) node[right=3pt]  {};
   \fill (6,2) coordinate (30) node[right=3pt]  {};
   \fill (13,-1) coordinate (31) node[right=3pt]  {};

	\draw[ultra thick] (19)--(21);
	\draw[ultra thick] (22)--(1);
	\draw[ultra thick] (23)--(30);
	\draw[ultra thick] (18)--(25);
	\draw[ultra thick] (26)--(24);
	\draw[ultra thick] (31)--(20);

%	\draw[ultra thick] (19)--(22);
	\draw[ultra thick] (13)--(3);
	\draw[ultra thick] (10)--(14);
	\draw[ultra thick] (27)--(12);
	\draw[ultra thick] (11)--(28);
	\draw[ultra thick] (1)--(15);
	\draw[ultra thick] (16)--(2);
	\draw[ultra thick] (4)--(5);
	\draw[ultra thick] (7)--(6);

\draw[ultra thick] ([shift={(9,1)}]90:1.0) arc[radius=1.0, start angle=90, end angle= 180];
\draw[ultra thick] ([shift={(7,1)}]90:1.0) arc[radius=1.0, start angle=270, end angle= 360];
\draw[ultra thick] ([shift={(12,1)}]90:1.0) arc[radius=1.0, start angle=90, end angle= 180];
\draw[ultra thick] ([shift={(10,1)}]90:1.0) arc[radius=1.0, start angle=270, end angle= 360];
\draw[ultra thick] ([shift={(10,-5)}]90:1.0) arc[radius=1.0, start angle=270, end angle= 360];
\draw[ultra thick] ([shift={(12,-5)}]90:1.0) arc[radius=1.0, start angle=90, end angle= 180];
\draw[ultra thick] ([shift={(7,-5)}]90:1.0) arc[radius=1.0, start angle=270, end angle= 360];
\draw[ultra thick] ([shift={(9,-5)}]90:1.0) arc[radius=1.0, start angle=90, end angle= 180];

	\end{tikzpicture}
}
\end{center}
\subcaption{The case $q_{k^\prime} < n$.}
\end{minipage}
\begin{minipage}[b]{0.4\linewidth}
\begin{center}
\scalebox{0.25}{
	\begin{tikzpicture}

  \fill (8,-2) coordinate (1) node[right=3pt]  {};
  \fill (8,-6) coordinate (2) node[right=3pt]  {};
  \fill (10,-4) coordinate (3) node[right=3pt]  {};
  \fill (13,-4) coordinate (4) node[right=3pt]  {};
  \fill (8,-5) coordinate (5) node[right=3pt]  {};
  \fill (11,-6) coordinate (6) node[right=3pt]  {};
  \fill (8,-3) coordinate (7) node[right=3pt]  {};
   \fill (6,5) coordinate (8) node[right=3pt]  {};
%  \fill (8,10) coordinate (9) node[right=3pt]  {};
 \fill (6,-4) coordinate (10) node[right=3pt]  {};
 \fill (11,-3) coordinate (11) node[right=3pt]  {};
 \fill (11,4) coordinate (12) node[right=3pt]  {};
   \fill (13,2) coordinate (18) node[right=3pt]  {};
   \fill (8,4) coordinate (19) node[right=3pt]  {};
   \fill (6,-1) coordinate (20) node[right=3pt]  {};
   \fill (8,7) coordinate (21) node[right=3pt]  {};
   \fill (8,6) coordinate (22) node[right=3pt]  {};
%   \fill (12,5) coordinate (23) node[right=3pt]  {};
   \fill (13,5) coordinate (24) node[right=3pt]  {};
   \fill (9,5) coordinate (25) node[right=3pt]  {};
%   \fill (10,5) coordinate (26) node[right=3pt]  {};
   \fill (11,7) coordinate (27) node[right=3pt]  {};
%   \fill (11,6) coordinate (28) node[right=3pt]  {};
   \fill (7,5) coordinate (29) node[right=3pt]  {};
   \fill (6,2) coordinate (30) node[right=3pt]  {};
   \fill (13,-1) coordinate (31) node[right=3pt]  {};
  \fill (9,-4) coordinate (32) node[right=3pt]  {};
 \fill (7,-4) coordinate (33) node[right=3pt]  {};

	\draw[ultra thick] (21)--(22);
	\draw[ultra thick] (25)--(24);
%	\draw[ultra thick] (25)--(26);
	\draw[ultra thick] (27)--(12);
	\draw[ultra thick] (29)--(8);

	\draw[ultra thick] (19)--(1);
	\draw[ultra thick] (18)--(30);
	\draw[ultra thick] (31)--(20);

%	\draw[ultra thick] (16)--(17);
	\draw[ultra thick] (10)--(33);
	\draw[ultra thick] (11)--(12);
	\draw[ultra thick] (1)--(7);
	\draw[ultra thick] (5)--(2);
	\draw[ultra thick] (4)--(32);
	\draw[ultra thick] (11)--(6);
%	\draw[ultra thick] (15)--(9);
%	\draw[ultra thick] (13)--(14);

\draw[ultra thick] ([shift={(7,4)}]90:1.0) arc[radius=1.0, start angle=270, end angle= 360];
\draw[ultra thick] ([shift={(9,4)}]90:1.0) arc[radius=1.0, start angle=90, end angle= 180];
%\draw[ultra thick] ([shift={(8,5)}]90:1.0) arc[radius=1.0, start angle=180, end angle= 270];
\draw[ultra thick] ([shift={(8,-4)}]90:1.0) arc[radius=1.0, start angle=180, end angle= 270];
\draw[ultra thick] ([shift={(8,-6)}]90:1.0) arc[radius=1.0, start angle=0, end angle= 90];

\draw[ultra thick] [arrows = {|[scale=2]-Stealth[scale=2]}] (14.4,0) -- (16.4,0);
	\end{tikzpicture}
\hspace{10mm}
	\begin{tikzpicture}

  \fill (8,-2) coordinate (1) node[right=3pt]  {};
  \fill (8,-6) coordinate (2) node[right=3pt]  {};
  \fill (10,-4) coordinate (3) node[right=3pt]  {};
  \fill (13,-4) coordinate (4) node[right=3pt]  {};
  \fill (8,-5) coordinate (5) node[right=3pt]  {};
  \fill (11,-6) coordinate (6) node[right=3pt]  {};
  \fill (8,-3) coordinate (7) node[right=3pt]  {};
   \fill (6,5) coordinate (8) node[right=3pt]  {};
  \fill (11,-5) coordinate (9) node[right=3pt]  {};
 \fill (6,-4) coordinate (10) node[right=3pt]  {};
 \fill (11,-3) coordinate (11) node[right=3pt]  {};
 \fill (11,4) coordinate (12) node[right=3pt]  {};
 \fill (12,-4) coordinate (13) node[right=3pt]  {};
   \fill (13,2) coordinate (18) node[right=3pt]  {};
   \fill (8,4) coordinate (19) node[right=3pt]  {};
   \fill (6,-1) coordinate (20) node[right=3pt]  {};
   \fill (8,7) coordinate (21) node[right=3pt]  {};
   \fill (8,6) coordinate (22) node[right=3pt]  {};
   \fill (12,5) coordinate (23) node[right=3pt]  {};
   \fill (13,5) coordinate (24) node[right=3pt]  {};
   \fill (9,5) coordinate (25) node[right=3pt]  {};
   \fill (10,5) coordinate (26) node[right=3pt]  {};
   \fill (11,7) coordinate (27) node[right=3pt]  {};
   \fill (11,6) coordinate (28) node[right=3pt]  {};
   \fill (7,5) coordinate (29) node[right=3pt]  {};
   \fill (6,2) coordinate (30) node[right=3pt]  {};
   \fill (13,-1) coordinate (31) node[right=3pt]  {};
  \fill (9,-4) coordinate (32) node[right=3pt]  {};
 \fill (7,-4) coordinate (33) node[right=3pt]  {};

	\draw[ultra thick] (21)--(22);
	\draw[ultra thick] (23)--(24);
	\draw[ultra thick] (26)--(25);
	\draw[ultra thick] (27)--(28);
	\draw[ultra thick] (29)--(8);

	\draw[ultra thick] (19)--(1);
	\draw[ultra thick] (18)--(30);
	\draw[ultra thick] (31)--(20);

%	\draw[ultra thick] (16)--(17);
	\draw[ultra thick] (10)--(33);
	\draw[ultra thick] (11)--(12);
	\draw[ultra thick] (1)--(7);
	\draw[ultra thick] (5)--(2);
	\draw[ultra thick] (4)--(13);
	\draw[ultra thick] (3)--(32);
	\draw[ultra thick] (9)--(6);

\draw[ultra thick] ([shift={(7,4)}]90:1.0) arc[radius=1.0, start angle=270, end angle= 360];
\draw[ultra thick] ([shift={(9,4)}]90:1.0) arc[radius=1.0, start angle=90, end angle= 180];
%\draw[ultra thick] ([shift={(8,5)}]90:1.0) arc[radius=1.0, start angle=180, end angle= 270];
\draw[ultra thick] ([shift={(8,-4)}]90:1.0) arc[radius=1.0, start angle=180, end angle= 270];
\draw[ultra thick] ([shift={(8,-6)}]90:1.0) arc[radius=1.0, start angle=0, end angle= 90];
\draw[ultra thick] ([shift={(10,4)}]90:1.0) arc[radius=1.0, start angle=270, end angle= 360];
\draw[ultra thick] ([shift={(12,4)}]90:1.0) arc[radius=1.0, start angle=90, end angle= 180];
%\draw[ultra thick] ([shift={(8,5)}]90:1.0) arc[radius=1.0, start angle=180, end angle= 270];
\draw[ultra thick] ([shift={(11,-4)}]90:1.0) arc[radius=1.0, start angle=180, end angle= 270];
\draw[ultra thick] ([shift={(11,-6)}]90:1.0) arc[radius=1.0, start angle=0, end angle= 90];

	\end{tikzpicture}
}
\end{center}
\subcaption{The case $q_{k^\prime} \geq n$.}
\end{minipage}
\caption{\label{fig:non_reduced_2} The change $\mathscr{G}(D) \mapsto \mathscr{G}(D^\prime)$ when $q_k < n$.}
\end{figure}
\end{proof}

\section{Mitosis recursion for reduced skew pipe dreams}

In this section, we prove that the set $\mathscr{M}(w)$ can be constructed by transposed skew mitosis operators. 
The following is the main result of this section. 

\begin{thm}\label{t:main_result_2}
For $w \in W$ and $1 \leq j \leq n$ such that $\ell(w) < \ell(w s_j)$, it holds that
\[\mathscr{L}(D(w s_j)) = M_j (\mathscr{L}(D(w))) = {\rm mitosis}_j^\top (\mathscr{L}(D(w))).\]
In particular, for $(j_1, \ldots, j_\ell) \in R(w)$, the following equalities hold:
\[\mathscr{L}(D(w)) = \mathscr{M}(w) = M_{j_\ell} \cdots M_{j_1} (SY_n) = {\rm mitosis}_{j_\ell}^\top \cdots {\rm mitosis}_{j_1}^\top (SY_n).\]
\end{thm}

To prove \cref{t:main_result_2}, we prepare some lemmas. 

\begin{lem}\label{l:mitosis_inclusion}
For $w \in W$ and $1 \leq j \leq n$ such that $\ell(w) < \ell(w s_j)$, it holds that
\[{\rm mitosis}_j^\top (\mathscr{L}(D(w))) \subseteq M_j (\mathscr{L}(D(w))) \subseteq \mathscr{L}(D(w s_j)).\]
\end{lem}

\begin{proof}
Since we have ${\rm mitosis}_j^\top (\mathscr{L}(D(w))) \subseteq M_j (\mathscr{L}(D(w)))$ by definition, it suffices to prove that $M_j (\mathscr{L}(D(w))) \subseteq \mathscr{L}(D(w s_j))$. 
Since $\mathscr{L}(D(w s_j)) = RSP(w s_j)$ by \cref{t:main_result_1}, let us show that $M_j (\mathscr{L}(D(w))) \subseteq RSP(w s_j)$. 
Take $D \in \mathscr{L}(D(w))$, and assume that $M_j (D) \neq \emptyset$. 
Define $r_0$ as in \eqref{eq:mitosis_removed_point} for $i = j$. 
The elements of $M_j (D)$ are obtained from $D \setminus \{(p_{r_0}, q_{r_0})\}$ by some ladder moves. 
Since $RSP(w s_j)$ is stable under ladder moves by \cref{l:stable_ladder_RSP}, it is sufficient to prove that $D \setminus \{(p_{r_0}, q_{r_0})\} \in RSP(w s_j)$. 
We divide the proof into two cases: $j = 1$ or $j > 1$. 

If $j = 1$, then by removing $\{(p_{r_0}, q_{r_0})\} \in D$, the diagram $\mathscr{G}(D)$ is changed to $\mathscr{G}(D \setminus \{(p_{r_0}, q_{r_0})\})$ as in Figure \ref{fig:mitosis_removed_1}. 
Since $D \in \mathscr{L}(D(w)) = RSP(w)$, this implies that $D \setminus \{(p_{r_0}, q_{r_0})\} \in RSP(w s_j)$.
\begin{figure}[h]
\begin{center}
\scalebox{0.25}{
	\begin{tikzpicture}

  \fill (8,-2) coordinate (1) node[right=3pt]  {};
  \fill (8,-6) coordinate (2) node[right=3pt]  {};
  \fill (10,-4) coordinate (3) node[right=3pt]  {};
  \fill (9,-7) coordinate (4) node[right=3pt]  {};
  \fill (10,-7) coordinate (5) node[right=3pt]  {};
  \fill (11,-6) coordinate (6) node[right=3pt]  {};
  \fill (11,-5) coordinate (7) node[right=3pt]  {};
%   \fill (6,5) coordinate (8) node[right=3pt]  {};
%  \fill (8,10) coordinate (9) node[right=3pt]  {};
 \fill (6,-4) coordinate (10) node[right=3pt]  {};
 \fill (6,-7) coordinate (11) node[right=3pt]  {};
 \fill (7,-7) coordinate (12) node[right=3pt]  {};
 \fill (8,-8) coordinate (13) node[right=3pt]  {};
 \fill (8,-9) coordinate (14) node[right=3pt]  {};
%   \fill (8,9) coordinate (15) node[above=3pt]  {};
%   \fill (6,8) coordinate (16) node[right=3pt]  {};
%   \fill (7,8) coordinate (17) node[right=3pt]  {};
   \fill (10,2) coordinate (18) node[right=3pt]  {};
   \fill (8,4) coordinate (19) node[above=3pt]  {\fontsize{50pt}{15pt}\selectfont \bf $L_1$};
   \fill (6,-1) coordinate (20) node[right=3pt]  {};
   \fill (8,7) coordinate (21) node[right=3pt]  {};
   \fill (6,2) coordinate (30) node[right=3pt]  {};
   \fill (10,-1) coordinate (31) node[right=3pt]  {};
   \fill (11,0) coordinate (32) node[right=3pt]  {};
   \fill (11,1) coordinate (33) node[right=3pt]  {};

	\draw[ultra thick] (19)--(1);
	\draw[ultra thick] (18)--(30);
	\draw[ultra thick] (31)--(20);
	\draw[ultra thick] (32)--(33);

%	\draw[ultra thick] (16)--(17);
	\draw[ultra thick] (10)--(3);
	\draw[ultra thick] (11)--(12);
	\draw[ultra thick] (1)--(2);
	\draw[ultra thick] (4)--(5);
	\draw[ultra thick] (7)--(6);
%	\draw[ultra thick] (15)--(9);
	\draw[ultra thick] (13)--(14);

%\draw[ultra thick] ([shift={(9,7)}]90:1.0) arc[radius=1.0, start angle=90, end angle= 180];
%\draw[ultra thick] ([shift={(8,5)}]90:1.0) arc[radius=1.0, start angle=180, end angle= 270];
%\draw[ultra thick] ([shift={(11,6)}]90:1.0) arc[radius=1.0, start angle=0, end angle= 90];
%\draw[ultra thick] ([shift={(10,4)}]90:1.0) arc[radius=1.0, start angle=270, end angle= 360];
\draw[ultra thick] ([shift={(11,0)}]90:1.0) arc[radius=1.0, start angle=0, end angle= 90];
\draw[ultra thick] ([shift={(10,-2)}]90:1.0) arc[radius=1.0, start angle=270, end angle= 360];
%\draw[ultra thick] ([shift={(8,3)}]90:1.0) arc[radius=1.0, start angle=0, end angle= 90];
%\draw[ultra thick] ([shift={(7,7)}]90:1.0) arc[radius=1.0, start angle=270, end angle= 360];
\draw[ultra thick] ([shift={(8,-9)}]90:1.0) arc[radius=1.0, start angle=0, end angle= 90];
\draw[ultra thick] ([shift={(11,-6)}]90:1.0) arc[radius=1.0, start angle=0, end angle= 90];
\draw[ultra thick] ([shift={(8,-7)}]90:1.0) arc[radius=1.0, start angle=180, end angle= 270];
\draw[ultra thick] ([shift={(10,-8)}]90:1.0) arc[radius=1.0, start angle=270, end angle= 360];

\draw[ultra thick] [arrows = {|[scale=2]-Stealth[scale=2]}] (12.4,-2) -- (14.4,-2);
	\end{tikzpicture}
\hspace{10mm}
	\begin{tikzpicture}

  \fill (8,-2) coordinate (1) node[right=3pt]  {};
  \fill (8,-6) coordinate (2) node[right=3pt]  {};
  \fill (10,-4) coordinate (3) node[right=3pt]  {};
  \fill (9,-7) coordinate (4) node[right=3pt]  {};
  \fill (10,-7) coordinate (5) node[right=3pt]  {};
  \fill (11,-6) coordinate (6) node[right=3pt]  {};
  \fill (11,-5) coordinate (7) node[right=3pt]  {};
%   \fill (6,5) coordinate (8) node[right=3pt]  {};
%  \fill (8,10) coordinate (9) node[right=3pt]  {};
 \fill (6,-4) coordinate (10) node[right=3pt]  {};
 \fill (6,-7) coordinate (11) node[right=3pt]  {};
 \fill (7,-7) coordinate (12) node[right=3pt]  {};
 \fill (8,-8) coordinate (13) node[right=3pt]  {};
 \fill (8,-9) coordinate (14) node[right=3pt]  {};
%   \fill (8,9) coordinate (15) node[above=3pt]  {};
%   \fill (6,8) coordinate (16) node[right=3pt]  {};
%   \fill (7,8) coordinate (17) node[right=3pt]  {};
   \fill (10,2) coordinate (18) node[right=3pt]  {};
   \fill (8,4) coordinate (19) node[above=3pt]  {\fontsize{50pt}{15pt}\selectfont \bf $L_1$};
   \fill (6,-1) coordinate (20) node[right=3pt]  {};
   \fill (6,2) coordinate (30) node[right=3pt]  {};
   \fill (10,-1) coordinate (31) node[right=3pt]  {};
   \fill (11,0) coordinate (32) node[right=3pt]  {};
   \fill (11,1) coordinate (33) node[right=3pt]  {};
  \fill (9,-4) coordinate (34) node[right=3pt]  {};
  \fill (7,-4) coordinate (35) node[right=3pt]  {};
  \fill (8,-3) coordinate (36) node[right=3pt]  {};
  \fill (8,-5) coordinate (37) node[right=3pt]  {};

	\draw[ultra thick] (19)--(1);
	\draw[ultra thick] (18)--(30);
	\draw[ultra thick] (31)--(20);
	\draw[ultra thick] (32)--(33);

%	\draw[ultra thick] (16)--(17);
	\draw[ultra thick] (34)--(3);
	\draw[ultra thick] (10)--(35);
	\draw[ultra thick] (11)--(12);
	\draw[ultra thick] (1)--(36);
	\draw[ultra thick] (37)--(2);
	\draw[ultra thick] (4)--(5);
	\draw[ultra thick] (7)--(6);
%	\draw[ultra thick] (15)--(9);
	\draw[ultra thick] (13)--(14);

%\draw[ultra thick] ([shift={(9,7)}]90:1.0) arc[radius=1.0, start angle=90, end angle= 180];
%\draw[ultra thick] ([shift={(8,5)}]90:1.0) arc[radius=1.0, start angle=180, end angle= 270];
%\draw[ultra thick] ([shift={(11,6)}]90:1.0) arc[radius=1.0, start angle=0, end angle= 90];
%\draw[ultra thick] ([shift={(10,4)}]90:1.0) arc[radius=1.0, start angle=270, end angle= 360];
\draw[ultra thick] ([shift={(11,0)}]90:1.0) arc[radius=1.0, start angle=0, end angle= 90];
\draw[ultra thick] ([shift={(10,-2)}]90:1.0) arc[radius=1.0, start angle=270, end angle= 360];
%\draw[ultra thick] ([shift={(8,3)}]90:1.0) arc[radius=1.0, start angle=0, end angle= 90];
%\draw[ultra thick] ([shift={(7,7)}]90:1.0) arc[radius=1.0, start angle=270, end angle= 360];
\draw[ultra thick] ([shift={(8,-9)}]90:1.0) arc[radius=1.0, start angle=0, end angle= 90];
\draw[ultra thick] ([shift={(11,-6)}]90:1.0) arc[radius=1.0, start angle=0, end angle= 90];
\draw[ultra thick] ([shift={(8,-7)}]90:1.0) arc[radius=1.0, start angle=180, end angle= 270];
\draw[ultra thick] ([shift={(10,-8)}]90:1.0) arc[radius=1.0, start angle=270, end angle= 360];
\draw[ultra thick] ([shift={(9,-5)}]90:1.0) arc[radius=1.0, start angle=90, end angle= 180];
\draw[ultra thick] ([shift={(7,-5)}]90:1.0) arc[radius=1.0, start angle=270, end angle= 360];

	\end{tikzpicture}
}
\end{center}
\caption{\label{fig:mitosis_removed_1} The change $\mathscr{G}(D) \mapsto \mathscr{G}(D \setminus \{(p_{r_0}, q_{r_0})\})$ when $j = 1$.}
\end{figure}

If $j > 1$, then by removing $\{(p_{r_0}, q_{r_0})\} \in D$, the diagram $\mathscr{G}(D)$ is changed to $\mathscr{G}(D \setminus \{(p_{r_0}, q_{r_0})\})$ as in Figure \ref{fig:mitosis_removed_2}. 
Hence we have $D \setminus \{(p_{r_0}, q_{r_0})\} \in RSP(w s_j)$ since $D \in \mathscr{L}(D(w)) = RSP(w)$.
This completes the proof of the lemma. 
\begin{figure}[h]
\centering
\begin{minipage}[b]{0.4\linewidth}
\begin{center}
\scalebox{0.25}{
	\begin{tikzpicture}

   \fill (8,-2) coordinate (1) node[right=3pt]  {};
  \fill (8,-6) coordinate (2) node[right=3pt]  {};
  \fill (10,-4) coordinate (3) node[right=3pt]  {};
  \fill (13,-4) coordinate (4) node[right=3pt]  {};
  \fill (12,-4) coordinate (5) node[right=3pt]  {};
  \fill (11,-6) coordinate (6) node[right=3pt]  {};
  \fill (11,-5) coordinate (7) node[right=3pt]  {};
   \fill (11,3) coordinate (8) node[right=3pt]  {};
  \fill (11,1) coordinate (9) node[right=3pt]  {};
 \fill (6,-4) coordinate (10) node[right=3pt]  {};
 \fill (11,-3) coordinate (11) node[right=3pt]  {};
 \fill (11,4) coordinate (12) node[above=3pt]  {\fontsize{50pt}{15pt}\selectfont \bf $L_{j-1}$};
 \fill (9,-4) coordinate (13) node[right=3pt]  {};
 \fill (7,-4) coordinate (14) node[right=3pt]  {};
   \fill (8,-3) coordinate (15) node[above=3pt]  {};
  \fill (8,-5) coordinate (16) node[right=3pt]  {};
 %  \fill (7,8) coordinate (17) node[right=3pt]  {};
   \fill (13,2) coordinate (18) node[right=3pt]  {};
   \fill (8,4) coordinate (19) node[above=3pt]  {\fontsize{50pt}{15pt}\selectfont \bf $L_j$};
   \fill (6,-1) coordinate (20) node[right=3pt]  {};
   \fill (8,3) coordinate (21) node[right=3pt]  {};
   \fill (8,1) coordinate (22) node[right=3pt]  {};
   \fill (7,2) coordinate (23) node[right=3pt]  {};
   \fill (9,2) coordinate (24) node[right=3pt]  {};
   \fill (12,2) coordinate (25) node[right=3pt]  {};
   \fill (10,2) coordinate (26) node[right=3pt]  {};
%   \fill (11,7) coordinate (27) node[right=3pt]  {};
 %  \fill (11,6) coordinate (28) node[right=3pt]  {};
%   \fill (7,5) coordinate (29) node[right=3pt]  {};
   \fill (6,2) coordinate (30) node[right=3pt]  {};
   \fill (13,-1) coordinate (31) node[right=3pt]  {};

	\draw[ultra thick] (19)--(22);
	\draw[ultra thick] (22)--(1);
	\draw[ultra thick] (24)--(30);
	\draw[ultra thick] (18)--(26);
	\draw[ultra thick] (26)--(24);
	\draw[ultra thick] (31)--(20);

%	\draw[ultra thick] (19)--(22);
	\draw[ultra thick] (13)--(3);
	\draw[ultra thick] (10)--(13);
	\draw[ultra thick] (9)--(12);
	\draw[ultra thick] (11)--(9);
	\draw[ultra thick] (1)--(15);
	\draw[ultra thick] (15)--(2);
	\draw[ultra thick] (4)--(5);
	\draw[ultra thick] (7)--(6);

\draw[ultra thick] ([shift={(10,-5)}]90:1.0) arc[radius=1.0, start angle=270, end angle= 360];
\draw[ultra thick] ([shift={(12,-5)}]90:1.0) arc[radius=1.0, start angle=90, end angle= 180];
%\draw[ultra thick] ([shift={(7,-5)}]90:1.0) arc[radius=1.0, start angle=270, end angle= 360];
%\draw[ultra thick] ([shift={(9,-5)}]90:1.0) arc[radius=1.0, start angle=90, end angle= 180];

\draw[ultra thick] [arrows = {|[scale=2]-Stealth[scale=2]}] (14.4,-1) -- (16.4,-1);
	\end{tikzpicture}
\hspace{10mm}
	\begin{tikzpicture}

  \fill (8,-2) coordinate (1) node[right=3pt]  {};
  \fill (8,-6) coordinate (2) node[right=3pt]  {};
  \fill (10,-4) coordinate (3) node[right=3pt]  {};
  \fill (13,-4) coordinate (4) node[right=3pt]  {};
  \fill (12,-4) coordinate (5) node[right=3pt]  {};
  \fill (11,-6) coordinate (6) node[right=3pt]  {};
  \fill (11,-5) coordinate (7) node[right=3pt]  {};
%   \fill (6,5) coordinate (8) node[right=3pt]  {};
%  \fill (8,10) coordinate (9) node[right=3pt]  {};
 \fill (6,-4) coordinate (10) node[right=3pt]  {};
 \fill (11,-3) coordinate (11) node[right=3pt]  {};
 \fill (11,4) coordinate (12) node[above=3pt]  {\fontsize{50pt}{15pt}\selectfont \bf $L_{j-1}$};
 \fill (9,-4) coordinate (13) node[right=3pt]  {};
 \fill (7,-4) coordinate (14) node[right=3pt]  {};
   \fill (8,-3) coordinate (15) node[above=3pt]  {};
  \fill (8,-5) coordinate (16) node[right=3pt]  {};
 %  \fill (7,8) coordinate (17) node[right=3pt]  {};
   \fill (13,2) coordinate (18) node[right=3pt]  {};
   \fill (8,4) coordinate (19) node[above=3pt]  {\fontsize{50pt}{15pt}\selectfont \bf $L_j$};
   \fill (6,-1) coordinate (20) node[right=3pt]  {};
   \fill (8,3) coordinate (21) node[right=3pt]  {};
   \fill (8,1) coordinate (22) node[right=3pt]  {};
   \fill (7,2) coordinate (23) node[right=3pt]  {};
   \fill (9,2) coordinate (24) node[right=3pt]  {};
   \fill (12,2) coordinate (25) node[right=3pt]  {};
   \fill (10,2) coordinate (26) node[right=3pt]  {};
   \fill (11,3) coordinate (27) node[right=3pt]  {};
  \fill (11,1) coordinate (28) node[right=3pt]  {};
%   \fill (7,5) coordinate (29) node[right=3pt]  {};
   \fill (6,2) coordinate (30) node[right=3pt]  {};
   \fill (13,-1) coordinate (31) node[right=3pt]  {};

	\draw[ultra thick] (19)--(22);
	\draw[ultra thick] (22)--(1);
	\draw[ultra thick] (24)--(30);
	\draw[ultra thick] (18)--(26);
	\draw[ultra thick] (26)--(24);
	\draw[ultra thick] (31)--(20);

%	\draw[ultra thick] (19)--(22);
	\draw[ultra thick] (13)--(3);
	\draw[ultra thick] (10)--(14);
	\draw[ultra thick] (28)--(12);
	\draw[ultra thick] (11)--(28);
	\draw[ultra thick] (1)--(15);
	\draw[ultra thick] (16)--(2);
	\draw[ultra thick] (4)--(5);
	\draw[ultra thick] (7)--(6);

\draw[ultra thick] ([shift={(10,-5)}]90:1.0) arc[radius=1.0, start angle=270, end angle= 360];
\draw[ultra thick] ([shift={(12,-5)}]90:1.0) arc[radius=1.0, start angle=90, end angle= 180];
\draw[ultra thick] ([shift={(7,-5)}]90:1.0) arc[radius=1.0, start angle=270, end angle= 360];
\draw[ultra thick] ([shift={(9,-5)}]90:1.0) arc[radius=1.0, start angle=90, end angle= 180];

	\end{tikzpicture}
}
\end{center}
\subcaption{The case $q_{r_0} < n$.}
\end{minipage}
\begin{minipage}[b]{0.4\linewidth}
\begin{center}
\scalebox{0.25}{
	\begin{tikzpicture}

  \fill (8,-2) coordinate (1) node[right=3pt]  {};
  \fill (8,-6) coordinate (2) node[right=3pt]  {};
  \fill (10,-4) coordinate (3) node[right=3pt]  {};
  \fill (13,-4) coordinate (4) node[right=3pt]  {};
  \fill (8,-5) coordinate (5) node[right=3pt]  {};
  \fill (11,-6) coordinate (6) node[right=3pt]  {};
  \fill (8,-3) coordinate (7) node[right=3pt]  {};
   \fill (6,5) coordinate (8) node[right=3pt]  {};
%  \fill (8,10) coordinate (9) node[right=3pt]  {};
 \fill (6,-4) coordinate (10) node[right=3pt]  {};
 \fill (11,-3) coordinate (11) node[right=3pt]  {};
 \fill (11,4) coordinate (12) node[right=3pt]  {};
% \fill (8,-8) coordinate (13) node[right=3pt]  {};
% \fill (8,-9) coordinate (14) node[right=3pt]  {};
%   \fill (8,9) coordinate (15) node[above=3pt]  {};
 %  \fill (6,8) coordinate (16) node[right=3pt]  {};
 %  \fill (7,8) coordinate (17) node[right=3pt]  {};
   \fill (13,2) coordinate (18) node[right=3pt]  {};
   \fill (8,4) coordinate (19) node[right=3pt]  {};
   \fill (6,-1) coordinate (20) node[right=3pt]  {};
   \fill (8,7) coordinate (21) node[above=3pt]  {\fontsize{50pt}{15pt}\selectfont \bf $L_j$};
   \fill (8,6) coordinate (22) node[right=3pt]  {};
%   \fill (12,5) coordinate (23) node[right=3pt]  {};
   \fill (13,5) coordinate (24) node[right=3pt]  {};
   \fill (9,5) coordinate (25) node[right=3pt]  {};
%   \fill (10,5) coordinate (26) node[right=3pt]  {};
   \fill (11,7) coordinate (27) node[above=3pt]  {\fontsize{50pt}{15pt}\selectfont \bf $L_{j-1}$};
%   \fill (11,6) coordinate (28) node[right=3pt]  {};
   \fill (7,5) coordinate (29) node[right=3pt]  {};
   \fill (6,2) coordinate (30) node[right=3pt]  {};
   \fill (13,-1) coordinate (31) node[right=3pt]  {};
  \fill (9,-4) coordinate (32) node[right=3pt]  {};
 \fill (7,-4) coordinate (33) node[right=3pt]  {};

	\draw[ultra thick] (21)--(19);
	\draw[ultra thick] (25)--(24);
%	\draw[ultra thick] (25)--(26);
	\draw[ultra thick] (27)--(12);
	\draw[ultra thick] (25)--(8);

	\draw[ultra thick] (19)--(1);
	\draw[ultra thick] (18)--(30);
	\draw[ultra thick] (31)--(20);

%	\draw[ultra thick] (16)--(17);
	\draw[ultra thick] (10)--(33);
	\draw[ultra thick] (11)--(12);
	\draw[ultra thick] (1)--(7);
	\draw[ultra thick] (5)--(2);
	\draw[ultra thick] (4)--(32);
	\draw[ultra thick] (11)--(6);
%	\draw[ultra thick] (15)--(9);
%	\draw[ultra thick] (13)--(14);

%\draw[ultra thick] ([shift={(7,4)}]90:1.0) arc[radius=1.0, start angle=270, end angle= 360];
%\draw[ultra thick] ([shift={(9,4)}]90:1.0) arc[radius=1.0, start angle=90, end angle= 180];
%\draw[ultra thick] ([shift={(8,5)}]90:1.0) arc[radius=1.0, start angle=180, end angle= 270];
\draw[ultra thick] ([shift={(8,-4)}]90:1.0) arc[radius=1.0, start angle=180, end angle= 270];
\draw[ultra thick] ([shift={(8,-6)}]90:1.0) arc[radius=1.0, start angle=0, end angle= 90];

\draw[ultra thick] [arrows = {|[scale=2]-Stealth[scale=2]}] (14.4,0) -- (16.4,0);
	\end{tikzpicture}
\hspace{10mm}
	\begin{tikzpicture}

  \fill (8,-2) coordinate (1) node[right=3pt]  {};
  \fill (8,-6) coordinate (2) node[right=3pt]  {};
  \fill (10,-4) coordinate (3) node[right=3pt]  {};
  \fill (13,-4) coordinate (4) node[right=3pt]  {};
  \fill (8,-5) coordinate (5) node[right=3pt]  {};
  \fill (11,-6) coordinate (6) node[right=3pt]  {};
  \fill (8,-3) coordinate (7) node[right=3pt]  {};
   \fill (6,5) coordinate (8) node[right=3pt]  {};
  \fill (11,-5) coordinate (9) node[right=3pt]  {};
 \fill (6,-4) coordinate (10) node[right=3pt]  {};
 \fill (11,-3) coordinate (11) node[right=3pt]  {};
 \fill (11,4) coordinate (12) node[right=3pt]  {};
 \fill (12,-4) coordinate (13) node[right=3pt]  {};
% \fill (8,-9) coordinate (14) node[right=3pt]  {};
%   \fill (8,9) coordinate (15) node[above=3pt]  {};
 %  \fill (6,8) coordinate (16) node[right=3pt]  {};
 %  \fill (7,8) coordinate (17) node[right=3pt]  {};
   \fill (13,2) coordinate (18) node[right=3pt]  {};
   \fill (8,4) coordinate (19) node[right=3pt]  {};
   \fill (6,-1) coordinate (20) node[right=3pt]  {};
   \fill (8,7) coordinate (21) node[above=3pt]  {\fontsize{50pt}{15pt}\selectfont \bf $L_j$};
   \fill (8,6) coordinate (22) node[right=3pt]  {};
   \fill (12,5) coordinate (23) node[right=3pt]  {};
   \fill (13,5) coordinate (24) node[right=3pt]  {};
   \fill (9,5) coordinate (25) node[right=3pt]  {};
   \fill (10,5) coordinate (26) node[right=3pt]  {};
   \fill (11,7) coordinate (27) node[above=3pt]  {\fontsize{50pt}{15pt}\selectfont \bf $L_{j-1}$};
   \fill (11,6) coordinate (28) node[right=3pt]  {};
   \fill (7,5) coordinate (29) node[right=3pt]  {};
   \fill (6,2) coordinate (30) node[right=3pt]  {};
   \fill (13,-1) coordinate (31) node[right=3pt]  {};
  \fill (9,-4) coordinate (32) node[right=3pt]  {};
 \fill (7,-4) coordinate (33) node[right=3pt]  {};

	\draw[ultra thick] (21)--(19);
	\draw[ultra thick] (8)--(24);
%	\draw[ultra thick] (26)--(25);
	\draw[ultra thick] (27)--(12);
%	\draw[ultra thick] (29)--(8);

	\draw[ultra thick] (19)--(1);
	\draw[ultra thick] (18)--(30);
	\draw[ultra thick] (31)--(20);

%	\draw[ultra thick] (16)--(17);
	\draw[ultra thick] (10)--(33);
	\draw[ultra thick] (11)--(12);
	\draw[ultra thick] (1)--(7);
	\draw[ultra thick] (5)--(2);
	\draw[ultra thick] (4)--(13);
	\draw[ultra thick] (3)--(32);
	\draw[ultra thick] (9)--(6);
%	\draw[ultra thick] (15)--(9);
%	\draw[ultra thick] (13)--(14);

\draw[ultra thick] ([shift={(8,-4)}]90:1.0) arc[radius=1.0, start angle=180, end angle= 270];
\draw[ultra thick] ([shift={(8,-6)}]90:1.0) arc[radius=1.0, start angle=0, end angle= 90];
\draw[ultra thick] ([shift={(11,-4)}]90:1.0) arc[radius=1.0, start angle=180, end angle= 270];
\draw[ultra thick] ([shift={(11,-6)}]90:1.0) arc[radius=1.0, start angle=0, end angle= 90];

	\end{tikzpicture}
}
\end{center}
\subcaption{The case $q_{r_0} \geq n$.}
\end{minipage}
\caption{\label{fig:mitosis_removed_2} The change $\mathscr{G}(D) \mapsto \mathscr{G}(D \setminus \{(p_{r_0}, q_{r_0})\})$ when $j > 1$.}
\end{figure}
\end{proof}

For $1 \leq j \leq n$, set $SY_n^{(j)} \coloneqq SY_n \sqcup \{(0, n+j)\}$, and denote by $\mathcal{SPD}_n^{(j)}$ the power set of $SY_n^{(j)}$. 
Since $SY_n \subseteq SY_n^{(j)}$, the set $\mathcal{SPD}_n$ of skew pipe dreams can be regarded as a subset of $\mathcal{SPD}_n^{(j)}$. 
For $(p, q) \in SY_n$, define the ladder move $L_{p,q}^{(j)}$ on $\mathcal{SPD}_n^{(j)}$ in a way similar to the ladder move $L_{p,q}$ on $\mathcal{SPD}_n$.
When we consider $\mathcal{SPD}_n^{(j)}$, we allow ladder moves that move $+$ to the additional box $(0, n+j) \in SY_n^{(j)}$.
For $D \in \mathcal{SPD}_n \subseteq \mathcal{SPD}_n^{(j)}$, denote by $\mathscr{L}^{(j)}(D)$ the set of elements of $\mathcal{SPD}_n^{(j)}$ obtained from $D$ by applying sequences of ladder moves $L_{p,q}^{(j)}$ for $(p, q) \in SY_n$. 
The inverse ladder move $(L_{p,q}^{(j)})^{-1}$ on $\mathcal{SPD}_n^{(j)}$ for $(p, q) \in SY_n^{(j)}$ is similarly defined. 

\begin{defi}
Let $x, y$ be indeterminates, and $\z[x, y]$ the polynomial ring with $\mathbb{Z}$-coefficients. 
For $D \in \mathcal{SPD}_n^{(j)}$, define a monomial ${\bm x}^D = x^{d(x, D)} y^{d(y, D)} \in \z[x, y]$ by 
\begin{align*}
&d(x, D) \coloneqq |\{(p, q) \in D \mid q \in \{n-j+1, n+j-1\}\}|,\\
&d(y, D) \coloneqq |\{(p, q) \in D \mid p \leq n-j+1,\ q \in \{n-j+2, n+j\}\}|. 
\end{align*}
\end{defi}

\begin{ex}\label{ex:extended_skew_pipe}
Let $n = 3$, and $j = 2$.
Then we have 
\[SY_3^{(2)} = \begin{ytableau}
\none & \none & \none & \none & \mbox{} \\
\mbox{} & \mbox{} & \mbox{} & \mbox{} & \mbox{} \\
\none & \mbox{} & \mbox{} & \mbox{} & \none \\
\none & \none & \mbox{} & \none & \none 
\end{ytableau}.\]
In addition, for $D = \{(1, 2), (1, 5), (2, 2), (2, 4), (3, 3)\} \in \mathcal{SPD}_3 \subseteq \mathcal{SPD}_3^{(2)}$, it follows that 
\[L_{1, 2}^{(2)} (D) = \begin{ytableau}
\none & \none & \none & \none & +\\
\mbox{} & \mbox{} & \mbox{} & \mbox{} & + \\
\none & + & \mbox{} & + & \none \\
\none & \none & + & \none & \none 
\end{ytableau}.\]
The variables $x$ and $y$ are assigned to boxes in $SY_3^{(2)}$ as follows: 
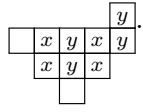
\[\begin{ytableau}
\none & \none & \none & \none & y \\
\mbox{} & x & y & x & y \\
\none & x & y & x & \none \\
\none & \none & \mbox{} & \none & \none 
\end{ytableau}.\]
Hence we have ${\bm x}^D = x^3 y$ and ${\bm x}^{L_{1, 2}^{(2)} (D)} = x^2 y^2$. 
\end{ex}

The set $SY_n^{(j)}$ is naturally embedded in the diagram $SY_{n+1}$ by regarding the $0$-th row of $SY_n^{(j)}$ as a part of the $1$-st row of $SY_{n+1}$ and by identifying the $k$-th row of $SY_n^{(j)}$ with the $(k+1)$-st row of $SY_{n+1}$. 
Then, for an element $w$ of the Weyl group $W$ of type $C_n$, the set $\mathscr{L}^{(j)}(D(w)) \subseteq \mathcal{SPD}_n^{(j)}$ is naturally thought of as a subset of $\mathscr{L}(D(w s_{n+1} \cdots s_2 s_1 s_2 \cdots s_{n+1})) \subseteq \mathcal{SPD}_{n+1}$. 
In particular, elements of $\mathscr{L}^{(j)}(D(w))$ have the same properties as reduced skew pipe dreams. 

\begin{lem}\label{l:+_in_adding_box_case}
For $w \in W$ and $1 \leq j \leq n$ such that $\ell(w) < \ell(w s_j)$, it holds that
\[\{D \in \mathscr{L}^{(j)}(D(w)) \mid (0, n+j) \in D\} = \{D \cup \{(0, n+j)\} \mid D \in \mathscr{L}(D(w s_j))\}.\]
\end{lem}

\begin{proof}
We prove the assertion only when $j > 1$; 
a proof of the assertion when $j = 1$ is similar. 
Take $D \in \mathscr{L}^{(j)}(D(w))$ such that $(0, n+j) \in D$. 
Then we can change $D$ to an element $D^\prime$ of $\mathscr{L}(D(w))$ by a sequence of inverse ladder moves. 
We may assume that $D^\prime = (L_{0, n+j}^{(j)})^{-1} (D'')$ for some $D'' \in \mathscr{L}^{(j)}(D(w))$ which is obtained from $D$ by a sequence of inverse ladder moves. 
Note that $D''$ is obtained from $D^\prime$ by a ladder move and that the change $D^\prime \mapsto D'' \setminus \{(0, n+j)\}$ is given in a form similar to Figure \ref{fig:mitosis_removed_2}. 
Since $D^\prime \in \mathscr{L}(D(w)) = RSP(w)$, this implies that $D'' \setminus \{(0, n+j)\} \in RSP(w s_j) = \mathscr{L}(D(w s_j))$. 
Hence we have $D \setminus \{(0, n+j)\} \in \mathscr{L}(D(w s_j))$.

Conversely, take $D \in \mathscr{L}(D(w s_j))$. 
If the inverse ladder move $D^\prime \coloneqq (L_{0, n+j}^{(j)})^{-1} (D \cup \{(0, n+j)\})$ is defined, then the change $D^\prime \mapsto D$ is given in a form similar to Figure \ref{fig:mitosis_removed_2}. 
Since $D \in \mathscr{L}(D(w s_j)) = RSP(w s_j)$, this implies that $D^\prime \in RSP(w) = \mathscr{L}(D(w))$. 
Hence we see that $D \cup \{(0, n+j)\} \in \mathscr{L}^{(j)}(D(w))$. 
Let us consider the case that the inverse ladder move $(L_{0, n+j}^{(j)})^{-1} (D \cup \{(0, n+j)\})$ is not defined. 
Since $D \in \mathscr{L}(D(w s_j)) = RSP(w s_j)$ and $\ell(w) < \ell(w s_j)$, it follows that $L_{j-1}$ and $L_j$ do not cross in the diagram $\mathscr{G}(D)$. 
Hence there exists $(p, q) \in SY_n \setminus D$ such that $q \in \{n-j+1, n+j-1\}$. 
Take $1 \leq k \leq N$ to be the maximum positive number such that $q_k \in \{n-j+1, n+j-1\}$ and such that $(p_k, q_k) \notin D$. 
Then the $(n-j+1)$-st and $(n-j+2)$-nd columns of $\mathscr{G}(D)$ are given as in Figure \ref{fig:extended_inverse_ladder}. 
If the inverse ladder move $D'' \coloneqq (L_{p_k, q_k+1}^{(j)})^{-1} (D \cup \{(0, n+j)\})$ is defined, then the inverse ladder move $(L_{0, n+j}^{(j)})^{-1} (D'')$ is also defined, and we have $(L_{0, n+j}^{(j)})^{-1} (D'') \in RSP(w) = \mathscr{L}(D(w))$. 
Hence it follows that $D \cup \{(0, n+j)\} \in \mathscr{L}^{(j)}(D(w))$. 
If the inverse ladder move $(L_{p_k, q_k+1}^{(j)})^{-1} (D \cup \{(0, n+j)\})$ is not defined, then repeat this argument using the fact that two pipes in $\mathscr{G}(D)$ do not cross more than once since $D$ is reduced.
Then we conclude that $D \cup \{(0, n+j)\}$ can be changed to an element of $\mathscr{L}(D(w))$ by a sequence of inverse ladder moves, which implies that $D \cup \{(0, n+j)\} \in \mathscr{L}^{(j)}(D(w))$. 
This completes the proof of the lemma.
\begin{figure}[h]
\centering
\begin{minipage}[b]{0.4\linewidth}
\begin{center}
\scalebox{0.25}{
	\begin{tikzpicture}

  \fill (8,-2) coordinate (1) node[right=3pt]  {};
  \fill (8,-6) coordinate (2) node[right=3pt]  {};
  \fill (10,-4) coordinate (3) node[right=3pt]  {};
  \fill (13,-4) coordinate (4) node[right=3pt]  {};
  \fill (12,-4) coordinate (5) node[right=3pt]  {};
  \fill (11,-6) coordinate (6) node[right=3pt]  {};
  \fill (11,-5) coordinate (7) node[right=3pt]  {};
%   \fill (6,5) coordinate (8) node[right=3pt]  {};
%  \fill (8,10) coordinate (9) node[right=3pt]  {};
 \fill (6,-4) coordinate (10) node[right=3pt]  {};
 \fill (11,-3) coordinate (11) node[right=3pt]  {};
 \fill (11,4) coordinate (12) node[above=3pt]  {\fontsize{50pt}{15pt}\selectfont \bf $L_{j-1}$};
 \fill (9,-4) coordinate (13) node[right=3pt]  {};
 \fill (7,-4) coordinate (14) node[right=3pt]  {};
   \fill (8,-3) coordinate (15) node[above=3pt]  {};
  \fill (8,-5) coordinate (16) node[right=3pt]  {};
 %  \fill (7,8) coordinate (17) node[right=3pt]  {};
   \fill (13,2) coordinate (18) node[right=3pt]  {};
   \fill (8,4) coordinate (19) node[above=3pt]  {\fontsize{50pt}{15pt}\selectfont \bf $L_j$};
   \fill (6,-1) coordinate (20) node[right=3pt]  {};
   \fill (8,3) coordinate (21) node[right=3pt]  {};
   \fill (8,1) coordinate (22) node[right=3pt]  {};
   \fill (7,2) coordinate (23) node[right=3pt]  {};
   \fill (9,2) coordinate (24) node[right=3pt]  {};
   \fill (12,2) coordinate (25) node[right=3pt]  {};
   \fill (10,2) coordinate (26) node[right=3pt]  {};
   \fill (11,3) coordinate (27) node[right=3pt]  {};
  \fill (11,1) coordinate (28) node[right=3pt]  {};
%   \fill (7,5) coordinate (29) node[right=3pt]  {};
   \fill (6,2) coordinate (30) node[right=3pt]  {};
   \fill (13,-1) coordinate (31) node[right=3pt]  {};

	\draw[ultra thick] (19)--(22);
	\draw[ultra thick] (22)--(1);
	\draw[ultra thick] (24)--(30);
	\draw[ultra thick] (18)--(26);
	\draw[ultra thick] (26)--(24);
	\draw[ultra thick] (31)--(20);

	\draw[ultra thick] (11)--(7);
	\draw[ultra thick] (13)--(3);
	\draw[ultra thick] (5)--(3);
	\draw[ultra thick] (10)--(14);
	\draw[ultra thick] (28)--(12);
	\draw[ultra thick] (11)--(28);
	\draw[ultra thick] (1)--(15);
	\draw[ultra thick] (16)--(2);
	\draw[ultra thick] (4)--(5);
	\draw[ultra thick] (7)--(6);

\draw[ultra thick] ([shift={(7,-5)}]90:1.0) arc[radius=1.0, start angle=270, end angle= 360];
\draw[ultra thick] ([shift={(9,-5)}]90:1.0) arc[radius=1.0, start angle=90, end angle= 180];

	\end{tikzpicture}
}
\end{center}
\subcaption{The case $q_k < n$.}
\end{minipage}
\begin{minipage}[b]{0.4\linewidth}
\begin{center}
\scalebox{0.25}{
	\begin{tikzpicture}

  \fill (8,-2) coordinate (1) node[right=3pt]  {};
  \fill (8,-6) coordinate (2) node[right=3pt]  {};
  \fill (10,-4) coordinate (3) node[right=3pt]  {};
  \fill (13,-4) coordinate (4) node[right=3pt]  {};
  \fill (8,-5) coordinate (5) node[right=3pt]  {};
  \fill (11,-6) coordinate (6) node[right=3pt]  {};
  \fill (8,-3) coordinate (7) node[right=3pt]  {};
   \fill (6,5) coordinate (8) node[right=3pt]  {};
  \fill (11,-5) coordinate (9) node[right=3pt]  {};
 \fill (6,-4) coordinate (10) node[right=3pt]  {};
 \fill (11,-3) coordinate (11) node[right=3pt]  {};
 \fill (11,4) coordinate (12) node[right=3pt]  {};
 \fill (12,-4) coordinate (13) node[right=3pt]  {};
% \fill (8,-9) coordinate (14) node[right=3pt]  {};
%   \fill (8,9) coordinate (15) node[above=3pt]  {};
 %  \fill (6,8) coordinate (16) node[right=3pt]  {};
 %  \fill (7,8) coordinate (17) node[right=3pt]  {};
   \fill (13,2) coordinate (18) node[right=3pt]  {};
   \fill (8,4) coordinate (19) node[right=3pt]  {};
   \fill (6,-1) coordinate (20) node[right=3pt]  {};
   \fill (8,7) coordinate (21) node[above=3pt]  {\fontsize{50pt}{15pt}\selectfont \bf $L_j$};
   \fill (8,6) coordinate (22) node[right=3pt]  {};
   \fill (12,5) coordinate (23) node[right=3pt]  {};
   \fill (13,5) coordinate (24) node[right=3pt]  {};
   \fill (9,5) coordinate (25) node[right=3pt]  {};
   \fill (10,5) coordinate (26) node[right=3pt]  {};
   \fill (11,7) coordinate (27) node[above=3pt]  {\fontsize{50pt}{15pt}\selectfont \bf $L_{j-1}$};
   \fill (11,6) coordinate (28) node[right=3pt]  {};
   \fill (7,5) coordinate (29) node[right=3pt]  {};
   \fill (6,2) coordinate (30) node[right=3pt]  {};
   \fill (13,-1) coordinate (31) node[right=3pt]  {};
  \fill (9,-4) coordinate (32) node[right=3pt]  {};
 \fill (7,-4) coordinate (33) node[right=3pt]  {};

	\draw[ultra thick] (21)--(19);
	\draw[ultra thick] (8)--(24);
	\draw[ultra thick] (32)--(33);
	\draw[ultra thick] (27)--(12);
%	\draw[ultra thick] (29)--(8);

	\draw[ultra thick] (19)--(1);
	\draw[ultra thick] (18)--(30);
	\draw[ultra thick] (31)--(20);

%	\draw[ultra thick] (16)--(17);
	\draw[ultra thick] (10)--(33);
	\draw[ultra thick] (11)--(12);
	\draw[ultra thick] (1)--(7);
	\draw[ultra thick] (5)--(2);
	\draw[ultra thick] (4)--(13);
	\draw[ultra thick] (3)--(32);
	\draw[ultra thick] (9)--(6);
	\draw[ultra thick] (5)--(7);
%	\draw[ultra thick] (13)--(14);

\draw[ultra thick] ([shift={(11,-4)}]90:1.0) arc[radius=1.0, start angle=180, end angle= 270];
\draw[ultra thick] ([shift={(11,-6)}]90:1.0) arc[radius=1.0, start angle=0, end angle= 90];

	\end{tikzpicture}
}
\end{center}
\subcaption{The case $q_k \geq n$.}
\end{minipage}
\caption{\label{fig:extended_inverse_ladder} The diagram $\mathscr{G}(D)$ when $(L_{0, n+j}^{(j)})^{-1} (D \cup \{(0, n+j)\})$ is not defined.}
\end{figure}
\end{proof}

By the proof of \cref{l:+_in_adding_box_case}, we know that there is no ladder moves on $D \in \mathscr{L}(D(w s_j))$ that move $+$ to the box $(0, n+j)$. 
Hence we obtain the following.

\begin{lem}\label{l:not_extended_ws_j}
For $w \in W$ and $1 \leq j \leq n$ such that $\ell(w) < \ell(w s_j)$, the equality $\mathscr{L}^{(j)}(D(w s_j)) = \mathscr{L}(D(w s_j))$ holds.
\end{lem}

Let $\Pi_j \coloneqq \{(p, q) \in \z_{>0}^2 \mid 1 \leq p \leq \gamma_j,\ 1 \leq q \leq 2\}$, where 
\[\gamma_j \coloneqq 
\begin{cases}
n+1 &(\text{if}\ j = 1),\\
2(n-j+1)+1 &(\text{if}\ j > 1). 
\end{cases}\] 
For $D \in \mathscr{L}^{(j)}(D(w))$, define a subset $D^{(j)} \subseteq \Pi_j$ as follows:
\begin{itemize}
\item the $1$-st row of $D^{(j)}$ is defined to be the boxes $(0, n+j-1), (0, n+j)$ in the $0$-th row of $D$;
\item if $j = 1$, then the $(k+1)$-st row of $D^{(j)}$ for $1 \leq k \leq n$ is given as the boxes $(k, n), (k, n+1)$ in the $k$-th row of $D$;
\item if $j > 1$, then the $2k$-th row of $D^{(j)}$ for $1 \leq k \leq n-j+1$ is defined to be the boxes $(k, n-j+1), (k, n-j+2)$ in the $k$-th row of $D$;
\item if $j > 1$, then the $(2k+1)$-st row of $D^{(j)}$ for $1 \leq k \leq n-j+1$ is given as the boxes $(k, n+j-1), (k, n+j)$ in the $k$-th row of $D$. 
\end{itemize}

\begin{ex}
Let $n = 3$, $j = 2$, and take $D$ as in \cref{ex:extended_skew_pipe}.
Then we have $\gamma_2 = 5$ and $\Pi_2 = \begin{ytableau}
\mbox{} &\mbox{} \\
\mbox{} &\mbox{} \\
\mbox{} &\mbox{} \\
\mbox{} &\mbox{} \\
\mbox{} &\mbox{} 
\end{ytableau}$. 
In addition, it holds that 
\[D^{(2)} = \begin{ytableau}
\mbox{} &\mbox{} \\
+ &\mbox{} \\
\mbox{} &+ \\
+ &\mbox{} \\
+&\mbox{} 
\end{ytableau},\quad L_{1, 2}^{(2)} (D)^{(2)} = \begin{ytableau}
\mbox{} &+ \\
\mbox{} &\mbox{} \\
\mbox{} &+ \\
+ &\mbox{} \\
+&\mbox{} 
\end{ytableau}.\]
\end{ex}

For $w \in W$, we write 
\[F_w(x, y) \coloneqq \sum_{D \in \mathscr{L}(D(w))} {\bm x}^D,\quad F_w^{(j)}(x, y) \coloneqq \sum_{D \in \mathscr{L}^{(j)}(D(w))} {\bm x}^D.\]
Define a $\z$-algebra automorphism $T_{x,y} \colon \z[x, y] \rightarrow \z[x, y]$ by permuting $x$ and $y$, that is, $T_{x, y}(f(x, y)) \coloneqq f(y, x)$ for $f(x, y) \in \z[x, y]$. 

\begin{lem}\label{l:invariant_under_permutation}
The polynomial $F_w^{(j)}(x, y)$ is invariant under the action of $T_{x, y}$. 
\end{lem}

\begin{proof}
For $D \in \mathscr{L}^{(j)}(D(w))$, let us consider $D^{(j)} \subseteq \Pi_j$. 
Since $(0, n+j-1), (n-j+1, n+j) \notin SY_n^{(j)}$, it follows that $(1, 1), (\gamma_j, 2) \notin D^{(j)}$. 
Let us ignore rows in $D^{(j)}$ of the form $\begin{ytableau}
+&+
\end{ytableau}$. 
Then, in a way similar to the argument in the proof of \cref{t:main_result_1}, we deduce that $D^{(j)}$ does not contain a part of the form $\begin{ytableau}
\mbox{} &+\\
+&\mbox{} 
\end{ytableau}$. 
Hence $D^{(j)}$ is uniquely divided into rectangles of the form
\[\ytableausetup{boxsize=1.2em}\begin{ytableau}
\mbox{} &+ \\
\scriptstyle \vdots &\scriptstyle \vdots \\
\mbox{} &+ \\
\mbox{} &\mbox{} \\
+ &\mbox{} \\
\scriptstyle \vdots &\scriptstyle \vdots \\
+&\mbox{} 
\end{ytableau}\]
(cf.\ \cite[Definition 11]{Mil}), where rows of the form $\ytableausetup{smalltableaux}\begin{ytableau}
\mbox{} &+
\end{ytableau}$ or $\begin{ytableau}
+ &\mbox{}
\end{ytableau}$ may not exist. 
For $c_1, c_2 \in \z_{\geq 0}$, let $R(c_1, c_2)$ denote such rectangle with $c_1$ $+$'s in the $1$-st column and with $c_2$ $+$'s in the $2$-nd column.
We define another rectangle $\tau(R(c_1, c_2))$ by $\tau(R(c_1, c_2)) \coloneqq R(c_2, c_1)$ (cf.\ \cite[Lemma 12]{Mil}). 
Applying $\tau$ to all rectangles $R(c_1, c_2)$ in $D^{(j)}$, we obtain $\tau(D^{(j)}) \subseteq \Pi_j$. 
Let $D \mapsto \tau(D)$ be the corresponding modification of $D$. 
Note that the modification $D \mapsto \tau(D)$ can be realized as a sequence of ladder moves and inverse ladder moves. 
Hence we have $\tau(D) \in \mathscr{L}^{(j)}(D(w))$.
In addition, the definition of $\tau$ implies that ${\bm x}^{\tau(D)} = T_{x,y}({\bm x}^D)$. 
Hence it follows that 
\[2F_w^{(j)}(x, y) = \sum_{D \in \mathscr{L}^{(j)}(D(w))} {\bm x}^D + \sum_{D \in \mathscr{L}^{(j)}(D(w))} {\bm x}^{\tau(D)} = \sum_{D \in \mathscr{L}^{(j)}(D(w))} ({\bm x}^D + T_{x,y}({\bm x}^D)),\]
which is invariant under the action of $T_{x, y}$.
This proves the lemma.
\end{proof}

We set
\[\widetilde{F}_w^{(j)}(x, y) \coloneqq \sum_{D \in \mathscr{L}(D(w));\ {\rm mitosis}_j^\top (D) = \emptyset} {\bm x}^D.\]
By the definition of ${\rm mitosis}_j^\top$, the argument in the proof of \cref{l:invariant_under_permutation} can also be applied to the polynomial $\widetilde{F}_w^{(j)}(x, y)$, which implies the following. 

\begin{lem}\label{l:invariant_under_permutation_mitosis_empty}
The polynomial $\widetilde{F}_w^{(j)}(x, y)$ is invariant under the action of $T_{x, y}$. 
\end{lem}

\begin{lem}\label{l:lemma_F(y, x)}
If $\ell(w) < \ell(w s_j)$, then it holds that 
\begin{equation}\label{eq:lemma_F(y, x)}
\begin{aligned}
F_w(x, y) + yF_{w s_j}(x, y) = T_{x, y} (F_w(x, y)) + x F_{w s_j}(x, y) = F_w^{(j)}(x, y).
\end{aligned}
\end{equation}
\end{lem}

\begin{proof}
Since $y {\bm x}^D = {\bm x}^{D \cup \{(0, n+j)\}}$ for $D \in \mathscr{L}(D(ws_j))$, it follows by \cref{l:+_in_adding_box_case} that 
\begin{equation}\label{eq:+_in_adding_box_case}
\begin{aligned}
yF_{w s_j}(x, y) = \sum_{D \in \mathscr{L}(D(ws_j))} y {\bm x}^D = \sum_{D \in \mathscr{L}^{(j)}(D(w));\ (0, n+j) \in D} {\bm x}^D.
\end{aligned}
\end{equation}
Hence we have
\begin{equation}\label{eq:first_assertion}
\begin{aligned}
F_w(x, y) + yF_{w s_j}(x, y) &= \sum_{D \in \mathscr{L}(D(w))} {\bm x}^D + \sum_{D \in \mathscr{L}^{(j)}(D(w));\ (0, n+j) \in D} {\bm x}^D\\
&= \sum_{D \in \mathscr{L}^{(j)}(D(w))} {\bm x}^D = F_w^{(j)}(x, y).
\end{aligned}
\end{equation}
In addition, it follows that 
\begin{align*}
T_{x, y} (F_w(x, y)) + xF_{w s_j}(x, y) &= T_{x, y} (F_w(x, y)) + T_{x, y} (yF_{w s_j}(x, y))\quad(\text{by Lemmas}\ \ref{l:not_extended_ws_j}\ \text{and}\ \ref{l:invariant_under_permutation})\\
&= T_{x, y} (F_w^{(j)}(x, y))\quad(\text{by}\ \eqref{eq:first_assertion})\\
&= F_w^{(j)}(x, y)\quad(\text{by Lemma}\ \ref{l:invariant_under_permutation}).
\end{align*}
From these, we conclude the lemma.
\end{proof}

\begin{lem}\label{l:mitosis_polynomial}
If $\ell(w) < \ell(w s_j)$, then it holds that 
\begin{equation}\label{eq:mitosis_polynomial_lemma}
\begin{aligned}
(x - y) \sum_{D \in {\rm mitosis}_j^\top (\mathscr{L}(D(w)))} {\bm x}^D = F_w(x, y) - T_{x, y} (F_w(x, y)).
\end{aligned}
\end{equation}
\end{lem}

\begin{proof}
By the definition of ${\rm mitosis}_j^\top$, it follows that ${\rm mitosis}_j^\top(D) \cap {\rm mitosis}_j^\top(D^\prime) = \emptyset$ for all $D, D^\prime \in \mathscr{L}(D(w))$ such that $D \neq D^\prime$. 
Hence we have 
\[\sum_{D \in {\rm mitosis}_j^\top (\mathscr{L}(D(w)))} {\bm x}^D = \sum_{D \in \mathscr{L}(D(w))} \sum_{D^\prime \in {\rm mitosis}_j^\top (D)} {\bm x}^{D^\prime}.\]
Take $D \in \mathscr{L}(D(w))$ such that ${\rm mitosis}_j^\top(D) \neq \emptyset$, and define $r_0$ as in \eqref{eq:mitosis_removed_point}. 
Then we can write the elements of ${\rm mitosis}_j^\top (D)$ as
\[D \setminus \{(p_{r_0}, q_{r_0})\}, L_{p_{r_1}, q_{r_1}}(D \setminus \{(p_{r_0}, q_{r_0})\}), \ldots, L_{p_{r_k}, q_{r_k}} \cdots L_{p_{r_1}, q_{r_1}}(D \setminus \{(p_{r_0}, q_{r_0})\})\]
for some $1 \leq r_k < \cdots < r_1 < r_0$ such that $q_{r_1}, \ldots, q_{r_k} \in \{n-j+1, n+j-1\}$. 
Set $\widehat{D} \coloneqq D \setminus \{(p_{r_0}, q_{r_0}), (p_{r_1}, q_{r_1}), \ldots, (p_{r_k}, q_{r_k})\}$. 
Then it follows that 
\begin{align*}
(x - y) \sum_{D^\prime \in {\rm mitosis}_j^\top (D)} {\bm x}^{D^\prime} &= (x - y) (x^k {\bm x}^{\widehat{D}} + x^{k-1} y {\bm x}^{\widehat{D}} + \cdots + y^k {\bm x}^{\widehat{D}})\\
&= (x^{k+1} - y^{k+1}) {\bm x}^{\widehat{D}} = {\bm x}^D - y^{k+1} {\bm x}^{\widehat{D}}. 
\end{align*}
Hence the left hand side of \eqref{eq:mitosis_polynomial_lemma} is computed as 
\begin{align*}
&(x - y) \sum_{D \in \mathscr{L}(D(w))} \sum_{D^\prime \in {\rm mitosis}_j^\top (D)} {\bm x}^{D^\prime} \\
=\ &\sum_{k \in \z_{\geq 0}} \sum_{D \in \mathscr{L}(D(w));\ |{\rm mitosis}_j^\top(D)| = k+1} ({\bm x}^D - y^{k+1} {\bm x}^{\widehat{D}})\\
=\ &F_w(x, y) - \widetilde{F}_w^{(j)}(x, y) - \sum_{k \in \z_{\geq 0}} y^{k+1} \sum_{D \in \mathscr{L}(D(w));\ |{\rm mitosis}_j^\top(D)| = k+1} {\bm x}^{\widehat{D}}. 
\end{align*}
The argument in the proof of \cref{l:invariant_under_permutation} implies that $\sum_{D \in \mathscr{L}(D(w));\ |{\rm mitosis}_j^\top(D)| = k+1} {\bm x}^{\widehat{D}}$ is invariant under the action of $T_{x, y}$. 
Hence we see that 
\begin{align*}
\sum_{k \in \z_{\geq 0}} y^{k+1} \sum_{D \in \mathscr{L}(D(w));\ |{\rm mitosis}_j^\top(D)| = k+1} {\bm x}^{\widehat{D}} &= T_{x, y} \left(\sum_{k \in \z_{\geq 0}} x^{k+1} \sum_{D \in \mathscr{L}(D(w));\ |{\rm mitosis}_j^\top(D)| = k+1} {\bm x}^{\widehat{D}}\right)\\
&= T_{x, y} \left(\sum_{k \in \z_{\geq 0}} \sum_{D \in \mathscr{L}(D(w));\ |{\rm mitosis}_j^\top(D)| = k+1} {\bm x}^D\right) \\
&= T_{x, y} (F_w(x, y) - \widetilde{F}_w^{(j)}(x, y))\\
&= T_{x, y} (F_w(x, y)) - \widetilde{F}_w^{(j)}(x, y)\quad (\text{by Lemma}\ \ref{l:invariant_under_permutation_mitosis_empty}).
\end{align*}
From these, we conclude the lemma. 
\end{proof}

Now we show \cref{t:main_result_2}. 

\begin{proof}[{Proof of \cref{t:main_result_2}}]
By \cref{l:mitosis_inclusion}, it suffices to prove that 
\[|\mathscr{L}(D(w s_j))| = |{\rm mitosis}_j^\top (\mathscr{L}(D(w)))|.\]
It follows by \cref{l:mitosis_polynomial} that 
\begin{align*}
(x - y) \sum_{D \in {\rm mitosis}_j^\top (\mathscr{L}(D(w)))} {\bm x}^D &= F_w(x, y) - T_{x, y} (F_w(x, y))\\
&= xF_{w s_j}(x, y) - y F_{w s_j}(x, y)\quad(\text{by Lemma}\ \ref{l:lemma_F(y, x)})\\
&= (x - y) F_{w s_j}(x, y).
\end{align*}
Hence we deduce that 
\[\sum_{D \in {\rm mitosis}_j^\top (\mathscr{L}(D(w)))} {\bm x}^D = F_{w s_j}(x, y) = \sum_{D \in \mathscr{L}(D(w s_j))} {\bm x}^D.\]
By comparing the numbers of terms, we conclude that 
\[|{\rm mitosis}_j^\top (\mathscr{L}(D(w)))| = |\mathscr{L}(D(w s_j))|.\]
This proves the theorem.
\end{proof}

\bibliographystyle{jplain} 

\begin{thebibliography}{99}
\bibitem{BZ}
A. Berenstein and A. Zelevinsky, Tensor product multiplicities, canonical bases and totally positive varieties, Invent.\ Math.\ {\bf 143} (2001), no.\ 1, 77--128.
\bibitem{BerBil}
N. Bergeron and S. Billey, RC-graphs and Schubert polynomials, Experiment.\ Math.\ {\bf 2} (1993), no.\ 4, 257--269.
\bibitem{BJS}
S. Billey, W. Jockusch, and R. P. Stanley, Some combinatorial properties of Schubert polynomials, J. Algebraic Combin.\ {\bf 2} (1993), 345--374.
\bibitem{Bri}
M. Brion, Lectures on the geometry of flag varieties, in Topics in Cohomological Studies of Algebraic Varieties, Trends Math., Birkh\"{a}user, Basel, 2005, 33--85.
\bibitem{BjoBre}
A. Bj\"{o}rner and F. Brenti, Combinatorics of Coxeter groups, Graduate Texts in Mathematics Vol.~231, Springer, New York, 2005.
\bibitem{Cal}
P. Caldero, Toric degenerations of Schubert varieties, Transform.\ Groups {\bf 7} (2002), 51--60. 
\bibitem{FK}
S. Fomin and A. N. Kirillov, The Yang--Baxter equation, symmetric functions, and Schubert polynomials, Discrete Math.\ {\bf 153} (1996), no.\ 1-3, 123--143.
\bibitem{FK_typeB}
S. Fomin and A. N. Kirillov, Combinatorial $B_n$-analogues of Schubert polynomials, Trans.\ Amer.\ Math.\ Soc.\ {\bf 348} (1996), no.\ 9, 3591--3620.
\bibitem{FS}
S. Fomin and R. P. Stanley, Schubert polynomials and the nil-Coxeter algebra, Adv.\ Math.\ {\bf 103} (1994), 196--207.
\bibitem{Fuj}
N. Fujita, Schubert calculus from polyhedral parametrizations of Demazure crystals, Adv.\ Math.\ {\bf 397} (2022), Paper No.\ 108201, 42 pages.
\bibitem{Jan}
J. C. Jantzen, Representations of Algebraic Groups, 2nd ed., Math.\ Surveys Monographs Vol.~107, Amer.\ Math.\ Soc., Providence, RI, 2003.
\bibitem{Kas4}
M. Kashiwara, The crystal base and Littelmann's refined Demazure character formula, Duke Math.\ J. {\bf 71} (1993), no.\ 3, 839--858.
\bibitem{Kas5}
M. Kashiwara, On crystal bases, in Representations of Groups (Banff, AB, 1994), CMS Conf. Proc. Vol.~16, Amer. Math. Soc., Providence, RI, 1995, 155--197.
\bibitem{Kav1}
K. Kaveh, Note on cohomology rings of spherical varieties and volume polynomial, J. Lie Theory {\bf 21} (2011), no.\ 2, 263--283.
\bibitem{Kav2}
K. Kaveh, Crystal bases and Newton--Okounkov bodies, Duke Math.\ J. {\bf 164} (2015), no.\ 13, 2461--2506.
\bibitem{KiNa}
A. N. Kirillov and H. Naruse, Construction of double Grothendieck polynomials of classical types using idCoxeter algebras, Tokyo J. Math.\ {\bf 39} (2017), no.\ 3, 695--728.
\bibitem{Kir_IMRN}
V. Kiritchenko, Gelfand--Zetlin polytopes and flag varieties, Int.\ Math.\ Res.\ Not.\ IMRN {\bf 2010} (2010), no.\ 13, 2512--2531.
\bibitem{Kir}
V. Kiritchenko, Geometric mitosis, Math.\ Res.\ Lett.\ {\bf 23} (2016), no.\ 4, 1069--1096.
\bibitem{Kir2}
V. Kiritchenko, Simple geometric mitosis, preprint 2023, arXiv:2301.00225v1. 
\bibitem{KP}
V. Kiritchenko and M. Padalko, Schubert calculus on Newton--Okounkov polytopes, in Interactions with Lattice Polytopes, Springer Proc.\ Math.\ Stat.\ Vol.\ 386, Springer, Cham, 2022, 233--249.
\bibitem{KST}
V. Kiritchenko, E. Smirnov, and V. Timorin, Schubert calculus and Gelfand--Zetlin polytopes, Russian Math.\ Surveys {\bf 67} (2012), no.\ 4, 685--719.
\bibitem{KnM}
A. Knutson and E. Miller, Gr\"{o}bner geometry of Schubert polynomials, Ann.\ of Math.\ (2) {\bf 161} (2005), no.\ 3, 1245--1318.
\bibitem{Kog}
M. Kogan, Schubert geometry of flag varieties and Gelfand--Cetlin theory, Ph.D. thesis, Massachusetts Institute of Technology, Boston, 2000.
\bibitem{KoM}
M. Kogan and E. Miller, Toric degeneration of Schubert varieties and Gelfand--Tsetlin polytopes, Adv.\ Math.\ {\bf 193} (2005), no.\ 1, 1--17.
\bibitem{LakSes}
V. Lakshmibai and C. S. Seshadri, Geometry of $G/P$.\ II.\ The work of de Concini and Procesi and the basic conjectures, Proc.\ Indian Acad.\ Sci.\ Sect.\ A {\bf 87} (1978), no.\ 2, 1--54.
\bibitem{LS}
A. Lascoux and M.-P.\ Sch\"{u}tzenberger, Polyn\^{o}mes de Schubert, C. R. Acad.\ Sci.\ Paris S\'{e}r.\ I Math.\ {\bf 294} (1982), no.\ 13, 447--450.
\bibitem{Lit}
P. Littelmann, Cones, crystals, and patterns, Transform.\ Groups {\bf 3} (1998), no.\ 2, 145--179.
\bibitem{Man}
L. Manivel, Symmetric functions, Schubert polynomials and degeneracy loci, SMF/AMS Texts and Monographs Vol.~6, Amer.\ Math.\ Soc., Providence, RI, 2001.
\bibitem{Mil}
E. Miller, Mitosis recursion for coefficients of Schubert polynomials, J. Comb.\ Theory Ser.\ A {\bf 103} (2003), no.\ 2, 223--235.
\bibitem{Mor}
S. Morier-Genoud, Geometric lifting of the canonical basis and semitoric degenerations of the Richardson varieties, Trans.\ Amer.\ Math.\ Soc.\ {\bf 360} (2008), no.\ 1, 215--235.
\bibitem{ST}
E. Smirnov and A. Tutubalina, Pipe dreams for Schubert polynomials of the classical groups, European J. Combin.\ {\bf 107} (2023), Paper No.\ 103613, 46 pages. 
\end{thebibliography}
\def\cprime{$'$} 

\end{document}